\documentclass[12pt]{article}
\usepackage{amsmath,amssymb,amsfonts,amsthm,mathrsfs}
\usepackage[page,title,titletoc,header]{appendix}
\usepackage{enumerate}
\usepackage{graphicx}
\usepackage{color}
\usepackage{geometry}
\usepackage{hyperref}
\usepackage{algpseudocodex}
\usepackage{algorithm}
\usepackage[numbers,square]{natbib}
\usepackage{makecell}
\usepackage{multirow}
\usepackage{array}
\usepackage{threeparttable}
\usepackage{subcaption}

\usepackage{changes}
\definechangesauthor[name={Yu Chenhao}, color=blue]{Yu}

\hypersetup{colorlinks=true,linkcolor=blue,citecolor=red}

\numberwithin{equation}{section}
\newtheorem{theorem}{Theorem}
\newtheorem{lemma}{Lemma}[section]
\newtheorem{proposition}{Proposition}[section]

\newtheorem{definition}{Definition}
\newtheorem{remark}{Remark}
\newtheorem{Exa}{Example}

\newtheoremstyle{case}{}{}{}{}{\bf}{:}{ }{}
\theoremstyle{case}
\newtheorem{case}{Case}

\newtheoremstyle{case1}{}{}{}{}{\bf}{:}{ }{}
\theoremstyle{case1}
\newtheorem{case1}{Case}

\allowdisplaybreaks

\newcommand{\be}{\begin{equation}}
\newcommand{\ee}{\end{equation}}
\newcommand{\bea}{\begin{eqnarray*}}
\newcommand{\eea}{\end{eqnarray*}}

\newcommand{\mE}{\mathbb{E}}
\newcommand{\la}{\langle}
\newcommand{\ra}{\rangle}

\newcommand{\sumlimits}[2]{\sum\limits_{#1}^{#2}}
\newcommand{\lsmooth}[1]{\left(L_0+L_1\|\nabla f(x_{#1})\|\right)}
\newcommand{\nablaf}[1]{\|\nabla f(#1)\|}

\newcommand{\keywords}[1]{\textit{\textbf{Keywords}: #1}}

\newcommand{\tiletalow}{\sqrt{G_{t-1}^2+A\overline{\Delta}_t+\left(B+1\right)\|\overline{g}_t\|^2+C}}
\newcommand{\etalow}{G_t}

\newcommand{\Gconstant}{8\overline{\Delta}_1+8+32\log^2\frac{T}{\delta}+8\frac{\sqrt{2C}}{\sqrt{L_0T}}\log\frac{T}{\delta}+\frac{303L_0}{224L_1^2}}
\newcommand{\Gadaptive}{\frac{8}{5}\overline{\Delta}_1+\frac{8}{5}\log\frac{T}{\delta}+\frac{8}{5}+\frac{303L_0}{1120L_1^2}}

\newcommand{\functiongapboundconstant}{\overline{\Delta}_c}
\newcommand{\functiongapboundadaptive}{\overline{\Delta}_a}
\newcommand{\functiongapboundsmooth}{\overline{\Delta}_L}
\newcommand{\functiongapexpectation}{\overline{\Delta}_e}

\begin{document}

  \title{Convergence Analysis of 
  Stochastic Accelerated Gradient  Methods
  	 for Generalized Smooth  Optimizations}

  \author{Chenhao Yu\footnotemark[1] \qquad
          Yusu Hong\footnotemark[1]\qquad 
          Junhong Lin\footnotemark[2]}
  \renewcommand{\thefootnote}{\fnsymbol{footnote}}
  \footnotetext[1]{Center for Data Science and School of Mathematical Sciences, Zhejiang University.}
  \footnotetext[2]{Center for Data Science, Zhejiang University.}
 
\date{	\small Ver.~1: Feb.~2025; Ver.~2: Oct.~2026 (JOTA version, adds simulations and minor revisions) }
\maketitle \baselineskip 16pt

\begin{abstract}
    We investigate the Randomized Stochastic Accelerated Gradient (RSAG) method, utilizing either constant or adaptive step sizes, for stochastic optimization problems with generalized smooth objective functions. Under relaxed affine variance assumptions for the stochastic gradient noise, we establish high-probability convergence rates of order $\tilde{O}\left(\sqrt{\log(1/\delta)/T}\right)$ for function value gaps in the convex setting, and for the squared gradient norms in the non-convex setting. Furthermore, when the noise parameters are sufficiently small, the convergence rate improves to $\tilde{O}\left(\log(1/\delta)/T\right)$, where $T$ denotes the total number of iterations and $\delta$ is the probability margin.
    Our analysis is also applicable to SGD with both constant and adaptive step sizes.
  \end{abstract}

\section{Introduction} 
\renewcommand{\thefootnote}{\arabic{footnote}}
We consider the stochastic optimization problem that 
\begin{align}
  \label{problem}
  \min\limits_{x\in\mathbb{R}^d} f(x)=\mathbb{E}_{z \sim \mathcal{D}}[f_z(x;z)],
\end{align}
where the objective function $f(x)$ is possibly non-convex and $\mathcal{D}$ is a probability distribution from which the random vector $z$ is drawn.

One of the popular algorithms for \eqref{problem} is Stochastic Gradient Descent (SGD) \cite{robbins1951stochastic}.
The vanilla SGD performs the iterations of the form  
$$x_{t+1}=x_t-\eta_t g_t,$$
where $g_t$ is the stochastic gradient estimate of $f$ at $x_t$ and $\eta_t>0$ is the step size.
Extensive research has been conducted on the convergence of SGD for smooth objective functions.
\cite{nemirovski2009robust} analyzed strongly convex functions and the expected error in terms of the function value gap is of order $\mathcal{O}\left(1/T\right)$, where $T$ is the total number of iterations. 
In the convex setting, \cite{lan2012optimal} provided an expected upper bound of order $\mathcal{O}(1/T + \sigma/\sqrt{T})$ in terms of the function value gap,
while for the non-convex case, \cite{ghadimi2013stochastic} obtained the same-order expected upper bound  for the average of squared gradient norms, both considering constant and well-tuned step sizes, under the bounded variance noise assumption.
 Here,
 $\sigma$ is the noise variance parameter.
While the upper bound for the non-convex case is in general unimprovable as it matches the lower bound in \cite{arjevani2023lower},
 classic convergence results for SGD with constant step sizes often require the information on problem parameters to tune the step sizes, e.g. $\eta_t=\eta\leq 1/L$
 with $L$ denoting the smoothness parameter of the objective function  in \cite{ghadimi2013stochastic}, which may be somewhat limited since the smoothness constant could  be  hard to know.
 Alternatively, one could consider SGD with adaptive step sizes, namely, 
 AdaGrad-Norm \cite{streeter2010less,duchi2011adaptive,ward2020adagrad,faw2022power}, which could thus potentially avoid this drawback.
 AdaGrad-Norm can been seen as a scalar version of AdaGrad \cite{streeter2010less,duchi2011adaptive}.
 Recall that the step sizes of AdaGrad-Norm are
\begin{align}
  \label{adaptive step size}
  \eta_t=\eta/\left(G_0^2+\sum\limits_{k=1}^t\|g_k\|^2\right)^{\frac{1}{2}},
\end{align}
where $\eta, G_0$ are some positive constants.
\cite{streeter2010less} and \cite{duchi2011adaptive} independently
provided the regret bound of order $\mathcal{O}\left(\sqrt{T}\right)$
 for AdaGrad-Norm in online convex optimization, assuming only bounded stochastic gradients.
 For the non-convex smooth stochastic optimization settings, \cite{ward2020adagrad} proved that AdaGrad-Norm has the expected convergence rate of $\tilde{\mathcal{O}}\left(1/\sqrt{T}\right)$ to find a stationary point
 without pre-tuning step sizes, assuming that both full gradients and noise variance are bounded.

Recently, increasing attention has been directed towards the so-called affine variance noise, in contrast to the previously mentioned bounded variance noise. In this case, for some constants \(B, C > 0\), the stochastic gradients \(g(x), \forall x \in \mathbb{R}^d\), satisfy
\begin{align}
\label{affine}
\mathbb{E}\left[\|g(x)-\nabla f(x)\|^2\right]\leq B \nablaf{x}^2+C.
\end{align} 
As noted in \cite{faw2022power}, this noise model has many practical applications, such as machine learning with feature noise \cite{khani2020feature} and robust linear regression \cite{xu2008robust}. 
While the analysis for vanilla SGD \cite{bottou2018optimization} is similar for both bounded variance noise and affine variance noise, the analysis for AdaGrad-Norm is more challenging in the latter case \cite{faw2022power, wang2023convergence, attia2023sgd}. \cite{faw2022power, wang2023convergence} provided the expected convergence rates of AdaGrad-Norm in the non-convex setting under \eqref{affine}, while \cite{attia2023sgd} established high-probability convergence rates in both convex and non-convex settings, assuming the almost-sure version of \eqref{affine},
\begin{align}
\label{asaffine}
\|g(x)-\nabla f(x)\|^2\leq B \nablaf{x}^2+C,
\end{align} 
or its corresponding sub-Gaussian version.

 In this paper, we will focus on a more general noise model, the so-called relaxed affine variance noise model\footnote{Our high-probability results also hold for its corresponding sub-Gaussian version.}, which satisfies that for some constants $A,B,C\geq 0$,
\begin{align}
  \label{affine variance}
  \|g(x)-\nabla f(x)\|^2\leq A\left(f(x)-f^*\right)+B\nablaf{x}^2+C,
\end{align}
where $f(x)\geq f^*$ for all $x\in\mathbb{R}^d$.
This relaxed affine variance noise assumption was proposed by \cite{khaledbetter} in the expected form where they observed that numerous practical stochastic gradient
settings, including
commonly used perturbation, sub-sampling and compression, satisfy this relaxed affine variance noise model but not \eqref{affine}. 
This expected noise assumption can be also seen as a combination of \eqref{affine} and the following expected smoothness condition \cite{nguyen2018sgd,gower2019sgd,grimmer2019convergence,wang2023convergence2},
\begin{align}
  \label{expected smooth}
\mathbb{E}\left[ \|g(x)\|^2\right]\leq A\left(f(x)-f^*\right)+C.
\end{align}
Moreover, \eqref{affine variance} is the general noise model which covers the corresponding almost-sure versions of the aforementioned
 noise assumptions, such as
  bounded variance noise, affine variance noise,
  expected smoothness condition, and also maximal strong growth condition \cite{schmidt2013fast}. 
Under the expected version of \eqref{affine variance}, \cite{khaledbetter}  obtained the expected convergence rate of SGD with constant step sizes. 
Recently, \cite{hongrevisiting} obtained high-probability convergence results for AdaGrad, assuming \eqref{affine variance} or its corresponding sub-Gaussian version.
Both \cite{khaledbetter} and \cite{hongrevisiting} focused on non-convex optimization.

Although the analysis in previous research \cite{ghadimi2013stochastic,ghadimi2016accelerated,levy2018online,ward2020adagrad,attia2023sgd,maddenhighprobability2024} often relies on the commonly used \(L\)-smoothness condition,  
\begin{align}
  \label{definition L smoothness}
  \|\nabla f(x)-\nabla f(y)\| \leq L\|x-y\|, \quad \forall x,y \in \mathbb{R}^d, 
\end{align}
this assumption can be somewhat restrictive in modern machine learning, particularly for neural network models \cite{zhang2019gradient}.
Indeed, \cite{zhang2019gradient}
empirically observed that in large language models, the following  generalized smoothness condition,
\begin{align*}
	\|\nabla^2 f(x)\|\leq L_0+L_1\|\nabla f(x)\|,
\end{align*}
could be more realistic to characterize the behavior of the objective function.
This assumption is further relaxed in \cite{zhang2020improved} with the following first-order formulation, not requiring the
second differentiability of the objective function.
\begin{definition}[$\left(L_0,L_1\right)$-smoothness]
  \label{definition 1}
  Let $f(x)$ be a differentiable function. For some $L_0,L_1>0$,
   $f(x)$ is $(L_0,L_1)$-smooth iff for any $x,y\in\mathbb{R}^d$ with $\|x-y\|\leq 1/L_1$,
  \begin{align}
    \label{def generalized smooth}
  \|\nabla f(x)-\nabla f(y)\|\leq \left(L_0+L_1\nablaf{x}\right)\|x-y\|.
  \end{align}
\end{definition}

There are a substantial amount of research on the convergence analysis of gradient based algorithms for generalized smooth objective functions.  
Indeed, \cite{zhang2019gradient} obtained the convergence rates of order $\mathcal{O}\left(1/T\right)$ for Gradient Descent (GD) and $\mathcal{O}\left(\left(1+\sigma\right)/\sqrt{T}\right)$ for SGD, where $\sigma$ is the bounded noise parameter, both involving extra gradient clipping steps, under the setting of non-convex objective functions.
\cite{zhang2020improved} improved the convergence rates for clipped SGD on the dependencies for all problem-dependent parameters, especially $L_1$ by order, under the same assumptions as \cite{zhang2019gradient}.
In the convex case, \cite{gorbunov2024methods} obtained the convergence rate of $\mathcal{O}\left(1/T\right)$ for clipping GD, and they referred to this algorithm as $(L_0,L_1)$-GD.
\cite{faw2023beyond} derived the expected convergence bound of order $\tilde{\mathcal{O}}\left(1/\sqrt{T}\right)$ for AdaGrad-Norm in the non-convex case under \eqref{affine} and \eqref{def generalized smooth} while \cite{wang2023convergence} provided alternative simpler analysis under the same setup.
Furthermore, \cite{wang2023convergence} gave a counter-example showing that the necessity of prior knowledge on parameters for the step-size parameter in the generalized smoothness case.

In this paper, we study the convergence analysis on RSAG proposed by \cite{ghadimi2016accelerated}, which is a stochastic accelerated gradient descent (AGD) template (see Algorithm \ref{algorithm1}).
This framework could recover some well-known algorithms, such as Nesterov's Accelerated Gradient Descent (NAG) \cite{nesterov1983method,nesterov2013introductor},
AdaGrad-Norm, and AcceleGrad \cite{levy2018online},
depending on parameter choices.
While NAG \cite{nesterov1983method} was originally designed for optimization problems with convex smooth objective functions achieving accelerated convergence over GD, \cite{ghadimi2016accelerated} 
showed that for non-convex smooth stochastic optimizations,
an expected convergence rate with respect to the stationary-point measure matching the lower bound in \cite{arjevani2023lower}, could be also derived for RSAG with constant step sizes, assuming bounded variance noise.
\cite{kavis2022high} recently studied adaptive RSAG and derived similar high-probability convergence results with respect to the stationary-point measure for the algorithm, assuming that the objective function is smooth and the noise is almost surely bounded.
In our work, we focus on both convex and non-convex objective functions and derive several convergence results under the assumption of relaxed affine variance noise or its almost-sure version.
 Also, all the convergence results presented in this paper hold true for vanilla SGD or AdaGrad-Norm.
We  summarize our main contributions as follows.
\begin{enumerate}[(a)]
  \item \label{contribution a} 
  We analyze SGD and RSAG with the constant step size under the assumptions of \eqref{affine variance} and \eqref{def generalized smooth}, and we obtain high-probability convergence rates. We show that
  after $T$ iterations,  with probability at least $1-\delta$,
  \begin{enumerate}[(i)]
    \item for non-convex objective functions, we have
    \begin{align}
      \label{contribution 1}
    \frac{1}{T}\sumlimits{t=1}{T}\|\nabla f(\overline{x}_t)\|^2\leq \mathcal{O}\left(\frac{\text{poly}\left(\log\frac{T}{\delta}\right)}{T}+\sqrt{\frac{\left(A+C\right)\text{poly}\left(\log\frac{T}{\delta}\right)}{T}}\right);
    \end{align}
    \item for convex objective functions, we have
    \begin{align}
      \label{contribution 2}
      f\left(\frac{1}{T}\sumlimits{t=1}{T}\overline{x}_t\right)-f^*\leq \mathcal{O}\left(\frac{\text{poly}\left(\log\frac{T}{\delta}\right)}{T}+\sqrt{\frac{\left(A+C\right)\text{poly}\left(\log\frac{T}{\delta}\right)}{T}}\right).
    \end{align}   
  \end{enumerate}
  Here, $\{\overline{x}_t\}_{t\in[T]}$ is the output sequence generated by Algorithm \ref{algorithm1}.
  \item Under the assumptions of \eqref{affine variance} and \eqref{def generalized smooth}, we show that AdaGrad-Norm and RSAG
   with the adaptive step size \eqref{adaptive step size} can have the same high-probability convergence rate as \eqref{contribution 1} for non-convex objective functions and the following high-probability
   convergence rate for convex objective functions,
  \begin{align}
    \label{contribution 3}
    f\left(\frac{1}{T}\sumlimits{t=1}{T}\overline{x}_t\right)-f^*\leq \mathcal{O}\left(\frac{\text{poly}\left(\log\frac{T}{\delta}\right)}{T}+\sqrt{\frac{C\text{poly}\left(\log\frac{T}{\delta}\right)}{T}}\right).
  \end{align}
  We believe that the difference between \eqref{contribution 2} and \eqref{contribution 3} is due to the setting of step sizes.
  \item We also apply our analysis about adaptive algorithms to $L$-smooth objective functions and obtain the convergence rates of the same order as \eqref{contribution 1} for non-convex functions and \eqref{contribution 3} for convex functions, not requiring the prior knowledge on problem parameters, such as $A$, $B$, $C$ and $L$.
  \item We obtain the following expected convergence rates of SGD and RSAG with the constant step size for $L$-smooth objective functions, 
  \begin{enumerate}[(i)]
    \item  for non-convex objective functions, we have
    \begin{align*}
      \mE\left[\frac{1}{T}\sumlimits{t=1}{T}\|\nabla f(\overline{x}_t)\|^2\right]\leq \mathcal{O}\left(\frac{1}{T}+\sqrt{\frac{A+C}{T}}\right);
    \end{align*} 
    \item for convex objective functions, we have
    \begin{align*}
      \mE\left[f\left(\frac{1}{T}\sumlimits{t=1}{T}\overline{x}_t\right)-f^*\right]\leq \mathcal{O}\left(\frac{1}{T}+\sqrt{\frac{C}{T}}\right).
    \end{align*}
  \end{enumerate}
\end{enumerate}

The rest of the paper are organized as follows. 
We first introduce some extra related works in the next section. 
Section \ref{preliminary} introduces the assumptions and the algorithms. 
In Section \ref{main result}, we summarize our main theorems.
Then, we give the proofs for generalized smooth objective functions with constant step sizes in Section \ref{section 4} and with adaptive step sizes in Section \ref{section 5}. 
In Section \ref{section 6}, we present our analysis for smooth objective functions with adaptive step sizes.
We also conduct experiments on the phase retrieval problem in Section \ref{experiment} to compare the practical behavior of RSAG (either with constant or adaptive step sizes), SGD and AdaGrad-Norm, showing the certain advantage of RSAG over SGD (and AdaGrad-Norm).
In the appendix, we provide the proofs that are omitted in the main text.

\paragraph{Notations}
In this paper, we denote the set $\{1,\cdots,T\}$ as $[T]$. 
We use $\mathbb{E}_t[\cdot]\triangleq\mathbb{E}[\cdot|z_1,\cdots,z_{t-1}]$ to represent the conditional expectation, where $z_i$ is the random sample in the $i$-th gradient oracle. 
The notation $a\sim \mathcal{O}(b)$ and $a\leq \mathcal{O}(b)$ refer to $c_1 b\leq a\leq c_2 b$ and $a\leq c_3 b$ with $c_1, c_2, c_3$ being positive constants, respectively. 
Also, we write $\tilde{\mathcal{O}}(b)$ as shorthand for $\mathcal{O}(b \cdot \text{poly}(\log b))$.
\section{Discussions}
\label{related work}
Due to page and knowledge limitations, we only discuss the most directly related results, focusing primarily on GD and SGD (without clipping or other advanced techniques) and their accelerated variants as well as using a cumulative gradient-norm based adaptive step sizes.
\paragraph{Accelerated Gradient Descent}
AGD, such as \cite{ghadimi2016accelerated,vaswani2019fast,nesterov2012efficiency,gupta2024nesterov,hermant2025gradient}, is one of the variants of GD and was first proposed by \cite{nesterov1983method} as NAG for convex smooth optimization, achieving the accelerated convergence rate of $\mathcal{O}(1/T^2)$ compared to the rate of $\mathcal{O}\left(1/T\right)$ for (non-accelerated) GD method.
\cite{nemirovskij1983problem} proved that the bound $\mathcal{O}\left(1/T^2\right)$ is not improvable for smooth convex optimization.
Later, numerous research has been conducted on the convergence behavior of NAG and its variants \cite{nesterov2005smooth,beck2009fast,su2016differential}.
\cite{ghadimi2016accelerated} generalized the classic AGD method to solve non-convex and possibly stochastic optimization problem. They obtained the convergence rate of $\mathcal{O}\left(1/T\right)$ for non-convex smooth objective functions when the accurate gradient is accessible.
Also for stochastic optimization, the expected convergence rate is up to $\mathcal{O}\left(1/T+\sqrt{C/T}\right)$ in the non-convex situation and $\mathcal{O}\left(1/T^2+\sqrt{C/T}\right)$ in the convex situation.
Under this generalized AGD formulation, \cite{kavis2022high} proved the convergence rate of $\tilde{\mathcal{O}}\left(1/\sqrt{T}\right)$ with high probability for non-convex objective functions without the knowledge of the smoothness parameters and noise level.
They assumed that the full gradient is bounded and also the stochastic gradient is almost surely bounded.
\cite{li2024convex} studied classic NAG under the generalized smoothness condition and they obtained the convergence rate of order $\mathcal{O}\left(1/T^2\right)$ for convex objective functions in the deterministic case.
Furthermore, \cite{hermant2025gradient} analyzed ACDM, which is a variant of AGD proposed by \cite{nesterov2012efficiency}, and showed the expected convergence rate of $\mathcal{O}\left((B+1)/T^2\right)$ and almost-sure convergence rate of $o\left((B+1)/T^2\right)$ in the convex case, assuming \eqref{affine} with $C=0$.
They also provided fast convergence rates for ACDM in the strongly convex case.
\paragraph{Generalized Smoothness and Relaxed Affine Variance Noise}
As mentioned above, global smoothness is limited especially when processing language models. 
Therefore, \cite{zhang2019gradient} provided the definition of generalized smoothness and established the convergence rate of $\mathcal{O}\left(1/\sqrt{T}\right)$ for SGD with bounded noise in the non-convex case.
Later, \cite{li2024convex} obtained the same-order convergence result for SGD under the bounded variance noise assumption.
In the field of adaptive methods for non-convex objective functions, \cite{faw2023beyond,wang2023convergence} established the expected convergence rate of order $\tilde{\mathcal{O}}\left(1/\sqrt{T}\right)$ for AdaGrad-Norm considering \eqref{affine} and \cite{hongrevisiting} proved that AdaGrad converges at the same order in high probability with relaxed affine variance noise in \eqref{affine variance}.

We list the studies most relevant to ours and the main results in this paper in Table \ref{table 1}.
\begin{table}[h!]
\renewcommand \arraystretch{1.8} 
\caption{Comparisons of the previous results and ours.}
\label{table 1}
\vspace{.1 in}
\scriptsize
\begin{center}
\begin{threeparttable}
  \begin{tabular}{|c|c|c|c|c|c|c|c|}
\hline
  &\textbf{Algorithm} & \textbf{Convexity} & \textbf{Noise} & \textbf{Smoothness} & \makecell{\textbf{Conv.} \\ \textbf{type}} & \makecell{\textbf{Conv.} \\ \textbf{rate}} & \makecell{\textbf{Extra cond.} \\ \textbf{for gradient}} \\
\hline
\cite{zhang2019gradient} & SGD & non-convex & \makecell{bounded \\(a.s.)} & generalized & $\mathbb{E}$ & $\frac{1+C}{\sqrt{T}}$& \checkmark\\
\hline
\cite{li2024convex} & SGD & non-convex & \makecell{bounded \\ variance}  & generalized & $\mathbb{E}$ &$\frac{1+\sqrt{C}}{\sqrt{T}}$ & \\
\hline
\cite{khaledbetter} & SGD & non-convex & relaxed affine & Lipschitz & $\mathbb{E}$ & $\frac{1}{T}+\frac{\sqrt{A}+\sqrt{C}}{\sqrt{T}}$&\\
\hline
\multirow{2}{*}{\cite{kavis2022high}} & \multirow{2}{*}{AdaGrad-Norm} & \multirow{2}{*}{non-convex} & \makecell{bounded \\ variance} & \multirow{2}{*}{Lipschitz} & \multirow{2}{*}{\textbf{w.h.p}} & $\frac{1+C}{\sqrt{T}}$ & \checkmark \\ \cline{4-4} \cline{7-8}
 &  &  & sub-Gaussian &  &  & $\frac{1}{T}+\sqrt{\frac{C}{T}}$& \checkmark \\
\hline
\multirow{2}{*}{\cite{attia2023sgd}} & \multirow{2}{*}{AdaGrad-Norm} & non-convex & \multirow{2}{*}{\makecell{affine \\ (a.s.)}} & \multirow{2}{*}{Lipschitz} & \multirow{2}{*}{\textbf{w.h.p}} & \multirow{2}{*}{$\frac{1}{T}+\sqrt{\frac{C}{T}}$} & \\ \cline{3-3}  \cline{8-8}
&& convex &&&&& \\
\hline
\cite{faw2023beyond} & AdaGrad-Norm & non-convex & \makecell{affine \\ $\left(B<1\right)$} & generalized & $\mathbb{E}$ & $\frac{1}{T}+\frac{\sqrt{C}+B}{\sqrt{T}}$ & \\  
\hline
\cite{wang2023convergence} & AdaGrad-Norm & non-convex & affine & generalized & $\mathbb{E}$ & $\frac{1}{T}+\sqrt{\frac{C}{T}}$ & \\                      
\hline
\cite{nesterov1983method} & NAG & convex & - & Lipschitz & - & $\frac{1}{T^2}$& \\
\hline
\cite{li2024convex} & NAG & convex & -  & generalized & - & $\frac{1}{T^2}$& \\
\hline
\multirow{2}{*}{\cite{ghadimi2016accelerated}} & \multirow{2}{*}{RSAG} & non-convex & \multirow{2}{*}{\makecell{bounded \\ variance}} & \multirow{2}{*}{Lipschitz} & \multirow{2}{*}{$\mathbb{E}$} & $\frac{1}{T}+\sqrt{\frac{C}{T}}$ & \\ \cline{3-3} \cline{7-8}
 &  & convex &  & & & $\frac{1}{T^2}+\sqrt{\frac{C}{T}}$& \\
\hline
\cite{kavis2022high} & RSAG & non-convex & \makecell{bounded \\ variance} &Lipschitz &  \textbf{w.h.p} & $\frac{1+\sqrt{C}}{\sqrt{T}}$ & \checkmark \\
\hline
Thm \ref{theorem 1}& \multirow{2}{*}{\makecell{SGD or\\ RSAG with \\ constant step sizes}} & non-convex & \multirow{2}{*}{\makecell{relaxed affine \\ (a.s.)}}& \multirow{2}{*}{generalized} & \multirow{2}{*}{\textbf{w.h.p}}&\multirow{2}{*}{$\frac{1}{T}+\frac{\sqrt{A}+\sqrt{C}}{\sqrt{T}}$}&\\ \cline{1-1} \cline{3-3} \cline{8-8}
Thm \ref{theorem 2}&   & convex &  &   & & &\\
\hline
Thm \ref{theorem 3}& \multirow{2}{*}{\makecell{AdaGrad-Norm \\ or RSAG with \\ adaptive step sizes}} & non-convex & \multirow{2}{*}{\makecell{relaxed affine\\ (a.s.)}}& \multirow{2}{*}{generalized} &\multirow{2}{*}{\textbf{w.h.p}}&$\frac{1}{T}+\frac{\sqrt{A}+\sqrt{C}}{\sqrt{T}}$&\\\cline{1-1} \cline{3-3} \cline{7-8}
Thm \ref{Theorem 4}&  & convex & & & &$\frac{1}{T}+\sqrt{\frac{C}{T}}$&\\
\hline
Thm \ref{theorem 5}& \multirow{2}{*}{\makecell{AdaGrad-Norm \\ or RSAG with \\ adaptive step sizes}} & non-convex & \multirow{2}{*}{\makecell{relaxed affine\\ (a.s.)}}& \multirow{2}{*}{Lipschitz} &\multirow{2}{*}{\textbf{w.h.p}}&$\frac{1}{T}+\frac{\sqrt{A}+\sqrt{C}}{\sqrt{T}}$&\\\cline{1-1} \cline{3-3} \cline{7-8}
Thm \ref{theorem 6}&  & convex &  & & &$\frac{1}{T}+\sqrt{\frac{C}{T}}$&\\
\hline
Thm \ref{theorem 7}& \multirow{2}{*}{\makecell{SGD or\\ RSAG with \\ constant step sizes}} & non-convex & \multirow{2}{*}{relaxed affine} & \multirow{2}{*}{Lipschitz} & \multirow{2}{*}{$\mathbb{E}$} &$\frac{1}{T}+\frac{\sqrt{A}+\sqrt{C}}{\sqrt{T}}$&\\\cline{1-1} \cline{3-3} \cline{7-8}
Thm \ref{theorem 8}&  & convex &  &  &  & $\frac{1}{T}+\sqrt{\frac{C}{T}}$&\\
\hline
\end{tabular}
\begin{tablenotes}
\footnotesize
\item[1] Indeed, \cite{li2024convex,faw2023beyond,wang2023convergence} provided the probabilistic results for stochastic optimization while the dependence of the probability margin is the polynomial of $1/\delta$. In order to distinguish them from other high-probability results with dependence of $\log\frac{T}{\delta}$, we consider them as the expected results.
\item[2] ``Con'' and ``cond'' are the shorthand of the words ``convergence'' and ``condition''. 
\item[3] We ignore the dependence on the noise parameters order and the logarithm factors of the horizon $T$ in this table.
\end{tablenotes}
\end{threeparttable}
\end{center}
\end{table}
\section{Preliminaries}
\label{preliminary}
We consider Problem \eqref{problem} over the Euclidean space $\mathbb{R}^d$ with the $l_2$ norm, denoted as $\|\cdot\|$.

\paragraph{Assumptions}
Below are the assumptions required in this paper.
\begin{enumerate}
  \item \label{assumption 1} \textbf{Below bounded:} The objective function is bounded from below, i.e.,  $f^*:=\inf_{x\in\mathbb{R}^d}f(x)>-\infty.$
  \item \label{assumption 2}\textbf{Unbiased estimator:} The gradient oracle returns an unbiased estimator of $\nabla f(x)$, i.e., for all $x\in\mathbb{R}^d$, 
  $$\mathbb{E}_z\left[\nabla f_z(x;z)\right]=\nabla f(x).$$
  \item \label{assumption 3}\textbf{Relaxed affine variance (almost-sure version):} The gradient oracle satisfies that for some constants $A,B,C>0$, $$\|\nabla f_z(x;z)-\nabla f(x)\|^2\leq A\left(f(x)-f^*\right)+B\nablaf{x}^2+C \quad a.s., \forall x\in\mathbb{R}^d.$$
  \item \label{assumption 4}\textbf{Relaxed affine variance (expected version):} The gradient oracle satisfies that for some constants $A,B,C>0$, $$\mE_z\left[\|\nabla f_z(x;z)-\nabla f(x)\|^2\right]\leq A\left(f(x)-f^*\right)+B\nablaf{x}^2+C, \quad \forall x\in\mathbb{R}^d.$$
\end{enumerate}

Assumptions \ref{assumption 1} and \ref{assumption 2} are standard in the field of stochastic optimization \cite{ghadimi2013stochastic,ghadimi2016accelerated,ward2020adagrad,attia2023sgd,li2024convex}.
We refer interested readers to see \cite{khaledbetter} for more examples satisfying the relaxed affine variance assumption.
\begin{remark}
  \label{remark 4.1}
  Assumption \ref{assumption 3} can be replaced by its sub-Gaussian form \cite{nemirovski2009robust,harvey2019tight} where 
  $$\mathbb{E}_z\left[\exp\left(\frac{\|\nabla f_z(x;z)-\nabla f(x)\|^2}{A\left(f(x)-f^*\right)+B\nablaf{x}^2+C}\right)\right]\leq \mathrm{e}.$$
  All our high-probability results still hold true under the noise assumption of sub-Gaussian form. See the appendix for the proofs.
\end{remark}
\begin{algorithm}[H]
  \caption{Randomized Stochastic Accelerated Gradient (RSAG)}
  \label{algorithm1}
  \begin{algorithmic}[1]
  \Require Horizon $T$, $\tilde{x}_1=x_1\in \mathbb{R}^d$, $\alpha_t \in \left(\left.0,1\right.\right]$, step sizes $\{\theta_t\}_{t\in [T]}$, $\{\gamma_t\}_{t\in [T]}$.
  \For{$t=1,\cdots, T$}
  \State $\overline{x}_t=\alpha_tx_t+\left(1-\alpha_t\right)\tilde{x}_t$;
  \State \textbf{Set} $g_t=\nabla f_z(\overline{x}_t;z_t)$;
  \State $x_{t+1}=x_t-\theta_t g_t$;
  \State $\tilde{x}_{t+1}=\overline{x}_t-\gamma_t g_t$.
  \EndFor
  \end{algorithmic}
\end{algorithm}

 Algorithm \ref{algorithm1} can be specialized into several well-known algorithms if we choose different parameter settings.

\begin{Exa}
  \label{example 1}
  Let the step size $\left\{\theta_t\right\}_{t\in[T]}$ be a constant. 
  If $\gamma_t=\theta_t, \forall t\in[T]$, Algorithm \ref{algorithm1} becomes classic SGD.
  Moreover, when $\gamma_t=\alpha_t\theta_t, \forall t\in[T]$, Algorithm \ref{algorithm1} becomes stochastic Nesterov's AGD with the step size $\gamma_t$ and the momentum parameter $\frac{\left(1-\alpha_t\right)\alpha_{t+1}}{\alpha_t}$. 
  We refer interested readers to see \cite{hong2025high} for details. 
\end{Exa}
\begin{Exa}
  \label{example 2}
  Let the step size $\left\{\theta_t\right\}_{t\in[T]}$ be of the adaptive form and $\left\{\eta_t\right\}_{t\in[T]}$ be defined in \eqref{adaptive step size}. 
  Algorithm \ref{algorithm1} becomes AdaGrad-Norm when $\gamma_t=\theta_t=\eta_t, \forall t\in[T]$, and  AdaGrad-Norm with averaging when $\gamma_t=0, \theta_t=\alpha_t\eta_t$ with $ \alpha_t=\frac{2}{t+1}$ or $\alpha_t=\frac{1}{t}, \forall t\in[T]$ \cite{kavis2022high}.
\end{Exa}

\section{Main Results}
\label{main result}
The main theorems of our paper are presented in this section.
We divide this section into four parts according to the different step size settings and assumptions, and in each subsection the theorems for both non-convex and convex cases are provided.
Concretely, we establish convergence bounds for RSAG with $\left\{\theta_t\right\}_{t\in[T]}$ being a constant (consider SGD as a special case) in Section \ref{subsection 4.1}, \ref{subsection 4.4} and for RSAG with $\left\{\theta_t\right\}_{t\in[T]}$ being adaptive (consider AdaGrad-Norm as a special case) in Section \ref{subsection 4.2}, \ref{subsection 4.3}. 
We refer interested readers to see the complete proof in the following part of this paper.
For simplicity, we denote $\overline{\Delta}_t=f(\overline{x}_t)-f^*$ and $\xi_t=g_t-\overline{g}_t$ where $\overline{g}_t=\nabla f(\overline{x}_t)$.
\subsection{Results of RSAG with Constant Step Sizes}
\label{subsection 4.1}
In this subsection, we analyzed Algorithm \ref{algorithm1} with constant step sizes under the generalized smoothness condition and the relaxed affine variance noise assumption of its almost-sure version.
\subsubsection*{Non-convex Optimization}
We first provide the theorem for non-convex objective functions.
\begin{theorem}
  \label{theorem 1}
  Let  $T>0$ and $\delta\in(0,1)$. Suppose that $f$ is an $\left(L_0,L_1\right)$-smooth function, $\{\overline{x}_t\}_{x\in[T]}$ is a sequence generated by Algorithm \ref{algorithm1} with $\alpha_t=\frac{2}{t+1},\gamma_t=\left(1+\alpha_t\right)\theta_t$ or $\gamma_t=\theta_t$ (SGD), $\forall t\in[T]$, and the step size $\theta_t$ satisfies
  \begin{align}
    \label{eta constant non-convex}
    \theta_t=\eta = \min\left\{\frac{1}{\mathcal{A}\sqrt{T}},\frac{1}{\mathcal{B}},\frac{1}{\mathcal{C}\sqrt{T}}, \frac{1}{\mathcal{G}},\frac{\functiongapboundconstant^{-1/2}}{3\left(A+4BL_0\right)}, \frac{\functiongapboundconstant^{-3/2}}{12BL_1^2}\right\},
  \end{align}
    where $\mathcal{A}, \mathcal{B}, \mathcal{C}, \mathcal{G}$ are defined as
  \begin{align}
    \label{theorem 1 eta}
    \mathcal{A}=4\sqrt{\left(L_0+L_1\mathcal{M}_c\right)A}, \qquad &\mathcal{B}=8\left(L_0+L_1\mathcal{M}_c\right)\left(B+1\right), \notag\\
    \mathcal{C}=2\sqrt{\left(L_0+L_1\mathcal{M}_c\right)C},\qquad &\mathcal{G}=8L_1 \left(\sqrt{A\functiongapboundconstant}+2\left(\sqrt{B}+1\right)\sqrt{L_0 \functiongapboundconstant+L_1^2 \functiongapboundconstant^2}+\sqrt{C}\right),  
  \end{align}
     and $\functiongapboundconstant$, $\mathcal{M}_c$ are given in the following order\footnote{The explicit expressions of $\functiongapboundconstant$ and $\mathcal{M}_c$ are in \eqref{Gconstant} and \eqref{nabla fx_t constant step size}, respectively.},
  \begin{align*} 
  \functiongapboundconstant \sim \mathcal{O}\left(\overline{\Delta}_1+\log^2\frac{T}{\delta}+\frac{\sqrt{C}}{\sqrt{L_0T}}\log\frac{T}{\delta}+\frac{L_0}{L_1^2 }\right), \qquad  \mathcal{M}_c \sim \mathcal{O}\left(L_1\functiongapboundconstant\right).
  \end{align*}
  Under Assumptions \ref{assumption 1}, \ref{assumption 2} and \ref{assumption 3}, with probability at least $1-\delta$, we have\footnote{The detail convergence rate is stated in \eqref{5.34}.}
  \begin{align}
    \label{4.1}
    \frac{1}{T}\sumlimits{t=1}{T}\|\nabla f(\overline{x}_t)\|^2\leq & \mathcal{O}\left(\frac{\functiongapboundconstant\left(\functiongapboundconstant B\log \frac{T}{\delta}+\sqrt{B+1}\right)}{T}+\frac{\sqrt{\functiongapboundconstant}\left(\functiongapboundconstant\sqrt{A}+\sqrt{C}\right)}{\sqrt{T}}\right).
  \end{align}
\end{theorem}

 As we can see from the above theorem, the step size $\theta_t$ in \eqref{eta constant non-convex} is of order $\tilde{\mathcal{O}}\left(1/\sqrt{T}\right)$, matching the one in classic SGD \cite{ghadimi2013stochastic} and RSAG \cite{ghadimi2016accelerated} up to the logarithm factor, where the smoothness condition and bounded variance noise assumption are required.
Moreover, the probabilistic convergence result in \eqref{4.1} is of order $\tilde{\mathcal{O}}\left(1/T+\sqrt{\left(A+C\right)/T}\right)$, and it adapts to the noise level since the rate could accelerate to $\tilde{\mathcal{O}}\left(1/T\right)$ when the noise level $A$ and $C$ are sufficiently low.
 This result matches the expected rates of vanilla SGD with fixed step sizes in \cite{ghadimi2013stochastic} and \cite{bottou2018optimization} as well as the high-probability rate in \cite{liu2023high}.
  Note that \cite{ghadimi2013stochastic,bottou2018optimization,liu2023high} required the smoothness condition.
\subsubsection*{Convex Optimization}
Next, we show the convergence behavior in terms of the function value gap for the convex objective functions under the same assumptions.
\begin{theorem}
  \label{theorem 2}
  Let $T>0$ and $\delta\in(0,1)$. Suppose that $f$ is $\left(L_0,L_1\right)$-smooth and convex, $\{\overline{x}_t\}_{x\in[T]}$ is the sequence generated by Algorithm \ref{algorithm1} with $\alpha_t=\frac{2}{t+1},\gamma_t=\left(1+\alpha_t\right)\theta_t$ or $\gamma_t=\theta_t$ (SGD), $\forall t\in[T]$, and the step size $\theta_t$ satisfies
  \begin{align}
    \label{eta constant convex}
  \theta_t=\eta =  \min\left\{\frac{1}{\mathcal{A}\sqrt{T}},\frac{1}{\mathcal{B}},\frac{1}{\mathcal{C}\sqrt{T}}, \frac{1}{\mathcal{G}}, \frac{1}{\mathcal{L}_c},\frac{\functiongapboundconstant^{-1/2}}{3\left(A+4BL_0\right)}, \frac{\functiongapboundconstant^{-3/2}}{12BL_1^2}\right\},
  \end{align}
  where $\mathcal{A}, \mathcal{B}, \mathcal{C}, \mathcal{G}$ are defined in \eqref{theorem 1 eta}, $\functiongapboundconstant$ is in \eqref{Gconstant} and
  \begin{align*}
    \mathcal{L}_c=16\left(A+4\left(B+1\right)\left(L_0+L_1^2\functiongapboundconstant\right)\right).
  \end{align*}
  Under Assumptions \ref{assumption 1}, \ref{assumption 2} and \ref{assumption 3}, with probability at least $1-3\delta$, we have\footnote{The explicit expressions of the convergence rate and $D_c^2$ are in \eqref{5.66} and \eqref{D_c}, respectively.}
  \begin{align}
    \label{4.7}
    f\left(\frac{1}{T}\sumlimits{t=1}{T}\overline{x}_t\right)-f^*\leq & \mathcal{O}\left[ \frac{\functiongapboundconstant \left(D_c^2 B\log\frac{1}{\delta}   + \sqrt{B+1}  \right)  }{T}       + \frac{\sqrt{\functiongapboundconstant \left(A+C\right)}+D_c\sqrt{C\log\frac{1}{\delta}}}{\sqrt{T}}   \right],
  \end{align}
  where 
  \begin{align*}
    D_c^2 \sim \mathcal{O}\left(\|x_1-x^*\|^2+ \frac{1}{L_1}+\log^2\frac{T}{\delta}\right).
  \end{align*}
\end{theorem}

Similar to the step size in Theorem \ref{theorem 1}, the step size $\theta_t$ here is of order $\tilde{\mathcal{O}}\left(1/\sqrt{T}\right)$, matching the one in the convex case in \cite{ghadimi2013stochastic} up to the logarithm factor.
Also, the convergence rate in \eqref{4.7} is of order $\tilde{\mathcal{O}}\left(1/T+\sqrt{\left(A+C\right)/T}\right)$ and it matches the expected rate in \cite{ghadimi2013stochastic} ignoring the logarithm factor.
However, the order of \eqref{4.7} is a bit worse comparing with the rate of $\mathcal{O}\left(1/T^2+\sqrt{C/T}\right)$ for solving smooth convex stochastic optimization problems in \cite{ghadimi2016accelerated}, although they required the smoothness condition and the bounded variance noise assumption considering different parameters scheme.

	\begin{remark}
	(1). In the two theorems above, we show that both RSAG and SGD achieve the same convergence rates for generalized-smooth stochastic optimization problems under relaxed affine variance noise assumptions, in both convex and non-convex settings. For the non-convex case, the rates for both methods are optimal and, ignoring logarithmic factors, unimprovable without additional assumptions. In the convex case, the rate for RSAG is slightly suboptimal compared to the rate established in \cite[Corollary 3, b)]{ghadimi2016accelerated} for RSAG under standard smoothness and bounded variance noise assumptions, albeit with a different parameter scheme. Improving this rate remains an interesting direction for future work.\\
		 (2). 	The RSAG we study is based on \cite[Corollary 3, a)]{ghadimi2016accelerated} and is closely related to the momentum technique. Momentum plays a key role in achieving accelerated convergence in convex optimization \cite{nesterov1983method}, and also enhances performance and noise stability in neural network training \cite{sutskever2013importance,liu2020improved,jelassi2022towards,kavis2022high}. In Section \ref{experiment}, we evaluate RSAG and SGD on a phase retrieval problem with generalized smooth objective functions and observe that RSAG outperforms SGD.	 
	\end{remark}

\subsection{Results of RSAG with Adaptive Step Sizes}
\label{subsection 4.2}
We then present the high-probability convergence bounds for RSAG with the adaptive step size and AdaGrad-Norm under the generalized smoothness condition and the relaxed affine variance noise assumption of its almost-sure version.
Theorem \ref{theorem 3} provides the convergence bound with respect to the average of the squared gradient norms in the non-convex case and Theorem \ref{Theorem 4} establishes the convergence bound in terms of the function value gap for convex objective functions.
\subsubsection*{Non-convex Optimization}
\begin{theorem}
  \label{theorem 3}
  Let $T>0$ and $\delta\in(0,1)$. 
  Suppose that $f$ is an $\left(L_0,L_1\right)$-smooth function, $\{\overline{x}_t\}_{x\in[T]}$ is the sequence generated by Algorithm \ref{algorithm1} with $\alpha_t=\frac{2}{t+1},\gamma_t=\left(1+\alpha_t\right)\theta_t$ or $\gamma_t=\theta_t$ (AdaGrad-Norm), $\forall t\in[T]$, where $\theta_t=\eta_t$ is defined in \eqref{adaptive step size}, and the constants $\eta, G_0$ satisfy 
  \begin{align}
    \label{theorem 3 eta}
    G_0>0, \qquad 0<\eta\leq \min\left\{\frac{1}{2\sqrt{\left(L_0+L_1\mathcal{M}_a\right)\mathcal{H}}}, \frac{1}{8P_a \mathcal{H}}, \frac{1}{3P_a},\frac{1}{8L_1} \right\},
  \end{align} 
  where
  \begin{align}
    \label{P_a}
    P_a =\sqrt{A\functiongapboundadaptive+4B\left(L_0\functiongapboundadaptive+L_1^2\functiongapboundadaptive^2\right)+C},
  \end{align}
  \begin{align}
    \label{mathcal H}
    \mathcal{H}=\log\left(1+\frac{2T\left(A\functiongapboundadaptive+4\left(B+1\right)\left(L_0\functiongapboundadaptive+L_1^2\functiongapboundadaptive^2\right)+C\right)}{G_0^2}\right),
  \end{align}
  and $\functiongapboundadaptive, \mathcal{M}_a$ are of order\footnote{The explicit expressions of $\functiongapboundadaptive, \mathcal{M}_a$ are in \eqref{functiongapboundadaptive}, \eqref{nabla fx_t adaptive step size}.}
  \begin{align*}
    \functiongapboundadaptive \sim \mathcal{O}\left(\overline{\Delta}_1+\log\frac{T}{\delta}+\frac{L_0}{L_1^2}\right), \quad  \mathcal{M}_a\sim \mathcal{O}\left(L_1\functiongapboundadaptive\right).
  \end{align*}
  Under Assumptions  \ref{assumption 1}, \ref{assumption 2} and \ref{assumption 3} , with probability at least $1-\delta$, we have\footnote{Refer \eqref{result of theorem 3} for the detail convergence rate.}
  \begin{align}
    \label{4.13}
    \frac{1}{T}\sumlimits{t=1}{T}\|\nabla f(\overline{x}_t)\|^2\leq \mathcal{O}\left[\functiongapboundadaptive^2\log T \left(\frac{\left(B+1\right)\functiongapboundadaptive^2\log T}{T}+\frac{\sqrt{A\functiongapboundadaptive+C}}{\sqrt{T}}\right)\right].
  \end{align}
  \end{theorem}
  
The order of the constant $\eta$ is lower than $\mathcal{O}\left(\frac{1}{\mathrm{poly}\left(\log T\right)}\right)$ while the upper bound of $\eta$ in \cite{wang2023convergence}, where they analyzed AdaGrad-Norm under the generalized smoothness and affine variance noise assumption, is a constant related to the smoothness and noise parameters and independent of the horizon $T$.
  Further, the convergence rate in \eqref{4.13} is of order $\tilde{\mathcal{O}}\left(1/T+\sqrt{\left(A+C\right)/T}\right)$ and it matches the high-probability rates of AdaGrad with momentum in \cite{hongrevisiting} and RSAG with adaptive step sizes in \cite{kavis2022high}. 
  Note that \cite{kavis2022high} not only assumes the conditions of smoothness and bounded variance noise, but also requires the gradients and stochastic gradients to be bounded.
\subsubsection*{Convex Optimization}
  We then analyze RSAG with adaptive step sizes and AdaGrad-Norm under the same assumptions of Theorem \ref{theorem 3} in the convex case.
  \begin{theorem}
    \label{Theorem 4}
    Let $T>0$ and $\delta\in(0,1)$. 
    Suppose that $f$ is $\left(L_0,L_1\right)$-smooth and convex, $\{\overline{x}_t\}_{x\in[T]}$ is the sequence generated by Algorithm \ref{algorithm1} with $\alpha_t=\frac{2}{t+1},\gamma_t=\left(1+\alpha_t\right)\theta_t$ or $\gamma_t=\theta_t$ (AdaGrad-Norm), $\forall t\in[T]$, where $\theta_t=\eta_t$ is defined in \eqref{adaptive step size}, and the parameters $\eta, G_0$ satisfy \eqref{theorem 3 eta}.
     Under Assumptions \ref{assumption 1}, \ref{assumption 2} and \ref{assumption 3}, with probability at least $1-4\delta$, we have
    \begin{align}
      \label{inequality 5.32}
      f\left(\frac{1}{T}\sumlimits{t=1}{T}\overline{x}_t\right)-f^*\leq & \mathcal{O}  \left[D_a^2\functiongapboundadaptive \log T\left(\frac{D_a^2\functiongapboundadaptive\log T\left(A+\functiongapboundadaptive\left(B+1\right)\right)}{T}       +    \frac{\sqrt{C}}{\sqrt{T}} \right)\right],
    \end{align} 
  where\footnote{The explicit expressions of the convergence rate and $R_a$, $D_a$ are stated in \eqref{6.77}, \eqref{R}, \eqref{D_a} respectively. }
    \begin{align*}
      R_a\sim \mathcal{O}\left( \log\left(\frac{T\functiongapboundadaptive^2}{G_0^2}\right)+\functiongapboundadaptive^2\log\frac{T}{\delta}  \right), \qquad D_a^2\sim \mathcal{O}\left(\|x_1-x^*\|^2+R_a\log\frac{T}{\delta}+R_a^2\right).
    \end{align*}
  \end{theorem}
  
  The convergence rate in \eqref{inequality 5.32} is of order $\tilde{\mathcal{O}}\left(1/T+\sqrt{C/T}\right)$, which differs from \eqref{4.13} slightly in the dependence of $A$ under the same constraints of $\eta$. 
    We believe that this difference is due to the different measures in non-convex and convex settings, i.e., the average of squared gradient norms for the non-convex case and the function value gap for the convex one.
    Although \eqref{inequality 5.32} matches the high-probability rate for AdaGrad-Norm in \cite{attia2023sgd} where smoothness and affine variance noise are assumed, it is  slightly suboptimal compared to the the bound of order $\mathcal{O}\left(1/T^2+\sqrt{C/T}\right)$ in \cite{ghadimi2016accelerated} where they assumed smoothness and bounded variance noise, considering different parameters scheme.

   \subsection{Results of RSAG with Adaptive Step Sizes under the Smoothness Condition}
  \label{subsection 4.3}
   Theorem \ref{theorem 3} and Theorem \ref{Theorem 4} require prior knowledge on problem parameters for step sizes, which is in general unavoidable as \cite{wang2023convergence} gave a counter-example showing the necessity of prior knowledge on problem parameters for step sizes in the generalized smoothness case.
   To eliminate this reliance on prior knowledge on problem parameters, we apply our analysis to the standard smoothness case and we obtain the high-probability convergence rates for RSAG with adaptive step sizes in this subsection.
  \subsubsection*{Non-convex Optimization}
   \begin{theorem}
    \label{theorem 5}
   Let $f(x)$ be an $L$-smooth function and $\left\{\overline{x}_t\right\}_{t\in[T]}$ be the sequence generated by Algorithm \ref{algorithm1} with $\alpha_t=\frac{2}{t+1},\gamma_t=\left(1+\alpha_t\right)\theta_t$ or $\gamma_t=\theta_t$ (AdaGrad-Norm), $\forall t\in[T]$, where $\theta_t=\eta_t$ is defined in \eqref{adaptive step size}. Suppose that Assumptions \ref{assumption 1}, \ref{assumption 2} and \ref{assumption 3} hold.
    For any $\eta>0$, with probability at least $1-\delta$, we have\footnote{Refer \eqref{7.37} and \eqref{functiongapboundsmooth} for explicit expressions of the convergence rate and $\functiongapboundsmooth$.}
    \begin{align}
      \label{4.16}
      \frac{1}{T}\sumlimits{t=1}{T}\|\nabla f(\overline{x}_t)\|^2\leq & \mathcal{O}\left[\functiongapboundsmooth\left( \frac{\functiongapboundsmooth\left(B+1\right)}{T}      +\frac{\sqrt{A\functiongapboundsmooth+C}}{\sqrt{T}}        \right)\right],
    \end{align}
    where $$\functiongapboundsmooth \sim \mathcal{O}\left(\overline{\Delta}_1+\log^2\frac{T}{G_0}+ \log^2\frac{T}{\delta} \right).$$
  \end{theorem}

 The convergence rate in Theorem \ref{theorem 5} is similar to \eqref{4.13} but the prior knowledge about the problem parameters is not required to tune the step size.
  Also, \eqref{4.16} aligns in order with the rates of AdaGrad-Norm in \cite{wang2023convergence}, AdaGrad with momentum in \cite{hongrevisiting} and RSAG with adaptive step sizes in \cite{kavis2022high}.
  As mentioned before, \cite{kavis2022high} required the assumptions of bounded (stochastic) gradients.
\subsubsection*{Convex Optimization}
  Under the same assumption of Theorem \ref{theorem 5}, Theorem \ref{theorem 6} provides the high-probability convergence result in the convex case.
  \begin{theorem}
    \label{theorem 6}
    Let $f(x)$ be $L$-smooth and convex, and  $\left\{\overline{x}_t\right\}_{t\in[T]}$ be the sequence generated by Algorithm \ref{algorithm1} with $\alpha_t=\frac{2}{t+1},\gamma_t=\left(1+\alpha_t\right)\theta_t$ or $\gamma_t=\theta_t$ (AdaGrad-Norm), $\forall t\in[T]$, where $\theta_t=\eta_t$ is defined in \eqref{adaptive step size}. Suppose that Assumptions \ref{assumption 1}, \ref{assumption 2} and \ref{assumption 3} hold.
    For any $\eta>0$, with probability at least $1-4\delta$, we have
    \begin{align} 
      \label{inequality theorem 6}
      f\left(\frac{1}{T}\sumlimits{t=1}{T}\overline{x}_t\right)-f^*\leq & \mathcal{O} \left[ D_L^2\left(  \frac{ D_L^2\left(A+\left(B+1\right)\right)}{T}       +  \frac{\sqrt{C}}{\sqrt{T}} \right) \right], 
    \end{align}
    where\footnote{We state the explicit expressions of the convergence rate and $R_L$, $D_L$  in \eqref{7.62}, \eqref{6.42}, \eqref{D_L}.}
    \begin{align*} 
      R_L=\mathcal{O}\left(\log\left(\frac{T\functiongapboundsmooth}{G_0^2}\right)+\functiongapboundsmooth\log\frac{T}{\delta}\right)  ,\qquad D_L^2 \sim \mathcal{O}\left(\|x_1-x^*\|^2+R_L\log\frac{T}{\delta}+R_L^2   \right).
    \end{align*}
  \end{theorem}
    
  This theorem relaxes the noise assumption in \cite{attia2023sgd} where they analyzed AdaGrad-Norm with $A=0$, and obtain the high-probability bound of the same order.
  Further, Theorem \ref{theorem 6} reveals that the convergence rate could accelerate to $\tilde{\mathcal{O}}\left(1/T\right)$ when the noise parameter $C$ is small enough, regardless of $A$ and $B$.

\subsection{Expected Results of Algorithms with Constant Step Sizes under the Smoothness Condition}
\label{subsection 4.4}
As a supplement, we also provide the expected convergence results of SGD and RSAG with the constant step size under the $L$-smoothness condition.
Correspondingly, we assume that the noise model satisfies the relaxed affine variance noise assumption of its expected version i.e., 
$$ \mE_z\left[\|\nabla f_z(x;z)-\nabla f(x)\|^2\right]\leq A\left(f(x)-f^*\right)+B\nablaf{x}^2+C, \quad \forall x\in\mathbb{R}^d,$$
which is weaker than its almost-sure version in Assumption \ref{assumption 3}.

Several prior work focused on generalized smooth and non-convex objective functions with expected versions of noise assumptions.
\cite{li2024convex} studied vanilla SGD and obtained a high-probability convergence rate with a $1/\delta$ dependence, where $\delta$ is the probability margin, assuming bounded variance noises.
\cite{faw2023beyond,wang2023convergence} analyzed AdaGrad-Norm under the expected affine variance noise assumption.
Essential challenge arises and more involved techniques are used, either relying on the ``stopping time" technique \cite{li2024convex} or the benign natures on adaptive step sizes \cite{faw2023beyond,wang2023convergence},
 in these papers
 when considering generalized smoothness objective functions.
In this subsection, we rather focus on the class of standard smooth functions.

\subsubsection*{Non-convex Optimization}
We present the result for non-convex objective functions first.
\begin{theorem}
  \label{theorem 7}
  Let $T>0$ and $f(x)$ be an $L$-smooth function. Suppose that $\left\{\overline{x}_t\right\}_{t\in[T]}$ is a sequence generated by Algorithm \ref{algorithm1} with $\alpha_t=\frac{2}{t+1},\gamma_t=\left(1+\alpha_t\right)\theta_t$ or $\gamma_t=\theta_t$ (SGD), $\forall t\in[T]$, and $\theta_t$ satisfies
  \begin{align}
    \label{expectation eta constraint}
    \theta_t=\eta=\min\left\{\frac{1}{\sqrt{3LAT}}, \sqrt{\frac{2}{3LCT}},\sqrt{\frac{2}{5L\left(A\functiongapexpectation+2L\left(B+1\right)\functiongapexpectation+C\right)}}, \frac{1}{4L\left(B+1\right)}\right\},
  \end{align}
where $\functiongapexpectation=4\overline{\Delta}_1+8$. Under Assumptions \ref{assumption 1}, \ref{assumption 2} and \ref{assumption 4}, we have that\footnote{The explicit convergence rate is stated in \eqref{expectation non convex}.}
\begin{align}
  \label{expected nonconvex}
  \mE\left[\frac{1}{T}\sumlimits{t=1}{T}\left\|\nabla f(\overline{x}_t)\right\|^2\right]\leq \mathcal{O}\left(\frac{B+1}{T}+\sqrt{\frac{A+C}{T}}\right).
\end{align}
\end{theorem}

Here, $\functiongapexpectation$ is the upper bound of $\mathbb{E}\left[f(x_t)-f^*\right]$ for all $t\in[T]$ and it is obvious that $\functiongapexpectation$ is independent of the term $\log \frac{T}{\delta}$, which appears in above probabilistic results. We provide the corresponding proof in the appendix.
\subsubsection*{Convex Optimization}
Based on some propositions obtained in the analysis of Theorem \ref{theorem 7}, we could provide the following theorem for convex objective functions under the same assumptions.
\begin{theorem}
  \label{theorem 8}
  Let $T>0$,  $f(x)$ be $L$-smooth and convex. Suppose that $\left\{\overline{x}_t\right\}_{t\in[T]}$ is a sequence generated by Algorithm \ref{algorithm1} with $\alpha_t=\frac{2}{t+1},\gamma_t=\left(1+\alpha_t\right)\theta_t$ or $\gamma_t=\theta_t$ (SGD), $\forall t\in[T]$, and $\theta_t$ satisfies
  \begin{align}
    \label{expectation eta constraint convex}
    \theta_t=\eta=\min\left\{\frac{1}{4\left(A+2L\left(B+1\right)\right)}, \frac{1}{\sqrt{C T}}\right\}.
  \end{align}
  Under Assumptions \ref{assumption 1}, \ref{assumption 2} and \ref{assumption 4}, we have that\footnote{The explicit expression is shown in \eqref{expectation convex}.}
  \begin{align}
    \label{expected convex}
    \mE\left[f\left(\frac{1}{T}\sumlimits{t=1}{T}\overline{x}_t\right)-f^*\right]\leq \mathcal{O}\left(\frac{A+B+1}{T}+\sqrt{\frac{C}{T}}\right).
  \end{align}
\end{theorem}

  Note that the step size $\eta$ in \eqref{expectation eta constraint} and \eqref{expectation eta constraint convex} is of order $\mathcal{O}\left(1/\sqrt{T}\right)$, matching the ones in classic SGD \cite{ghadimi2013stochastic} and RSAG \cite{ghadimi2016accelerated} where they assumed bounded variance noise assumption.
  Moreover, both \eqref{expected nonconvex} and \eqref{expected convex} adapt to the noise level but they have different dependence on the noise parameter $A$.
  From \cite{arjevani2023lower}, we know that the result in Theorem \ref{theorem 7} reaches the optimal order for non-convex objective functions. 
  However, the rate in Theorem \ref{theorem 8} is a bit worse compared to $\mathcal{O}\left(1/T^2+\sqrt{C/T}\right)$ in \cite{ghadimi2016accelerated} although they required bounded variance noise assumption.
\section{Analysis of Algorithms with Constant Step Sizes}
\label{section 4}
In this section, we provide analysis of the convergence rate for SGD and RSAG with the constant step size under the generalized smoothness condition. 
Concretely, we offer the proof for non-convex optimization in Section \ref{constant non-convex} and then derive the result for convex optimization in Section \ref{constant convex}.
We borrow some of the proof idea from \cite{ghadimi2016accelerated,kavis2022high,junchitwo,hubler2024parameter,attia2023sgd,hongrevisiting,hong2023high}. Refer the appendix for the omitted lemmas and proofs.
\subsection{Non-convex Optimization}
\label{constant non-convex}
The main challenge here is to bound $f(\overline{x}_t)-f^*$ for all $t\in[T]$ and we solve it using the induction argument in Proposition \ref{constant bound}.
The analysis of Proposition \ref{constant bound} relies on the descent lemma for the generalized smooth functions derived in Lemma \ref{descent lemma}.
By telescoping and then applying Lemma \ref{high probability sum non-convex constant} to the summation of the martingale difference sequence, we could bound $f(x_{l+1})-f^*$ in \eqref{f x+1-f x_1 constant step size}. 
To finish the induction, we establish the relation between $f(\overline{x}_{l+1})$ and $f(x_{l+1})$ using the descent lemma again in \eqref{gap between f overline x_t+1 and f x_t+1 constant step size}.
Then, we could obtain the probabilistic bound of $\overline{\Delta}_{l+1}$ and finish the induction argument.
Furthermore, dividing $\eta$ on both sides of \eqref{inequality base} and summing over $t\in[T]$, we could obtain \eqref{constant gt}.
Applying Lemma \ref{high probability sum non-convex constant} again to the summation of the martingale difference sequence appeared on the right side of \eqref{constant gt} with the constraints of $\eta$ and the probabilistic upper bound for the function value gap obtained in Proposition \ref{constant bound}, we reach the final result.

In what follows, we will present our complete analysis. The following two propositions follow from the analysis in \cite{ghadimi2016accelerated,kavis2022high}.
 \begin{proposition}[Proposition 5.2 in \cite{kavis2022high}]
  \label{proposition 1}
  Denote $\alpha_t=\frac{2}{t+1}$ and $\varGamma_t=\left(1-\alpha_t\right)\varGamma_{t-1}$ with $\varGamma_1=1$. We have that for all $t\in[T],$
  \begin{align}
    \label{proposition 1.1}
   \varGamma_{t}\sumlimits{k=1}{t}\frac{\alpha_k}{\varGamma_k}=1,
  \end{align}
  and 
  \begin{align}
    \label{proposition 1.2}
    \left[\sumlimits{k=t}{T}\left(1-\alpha_k\right)\varGamma_k\right]\frac{\alpha_t}{\varGamma_t}\leq 2.
  \end{align}
\end{proposition}

The following proposition bounds the gap of the two sequences specific to Algorithm \ref{algorithm1}.
\begin{proposition}
  \label{proposition overline x_t - x_t}
  Let $\{x_t\}_{t\in[T]}$ and $\{\overline{x}_t\}_{t\in[T]}$ be generated by Algorithm \ref{algorithm1}. We have
  \begin{align}
    \label{overline x_t-x_t}
    \overline{x}_t-x_t =\left(1-\alpha_t\right)\varGamma_{t-1}\sumlimits{k=1}{t-1}\frac{\alpha_k}{\varGamma_k}\frac{\left(\theta_k-\gamma_k\right)}{\alpha_k}g_k,
  \end{align}
  and
  \begin{align*}
  \|\overline{x}_t-x_t\|^2\leq \left(1-\alpha_t\right)\varGamma_t\sumlimits{k=1}{t-1}\frac{\alpha_k}{\varGamma_k}\frac{\left(\theta_k-\gamma_k\right)^2}{\alpha_k^2}\|g_k\|^2.
\end{align*}
\end{proposition}
\begin{lemma}
  \label{high probability sum non-convex constant}
  Given $T\geq 1 $ and $\delta\in \left(0,1\right)$, if Assumptions \ref{assumption 2} and \ref{assumption 3} hold, then with probability at least $1-\delta$, $\forall l\in[T]$,
  \begin{align}
    \label{inequality high probability non-convex constant}
    \sumlimits{t=1}{l}-\la\overline{g}_t,\xi_t\ra\leq \frac{1}{4}\sumlimits{t=1}{l}\frac{P_t^2}{P_c^2}\|\overline{g}_t\|^2+3P_c^2\log\frac{T}{\delta},
  \end{align}
  where 
  \begin{align}
    \label{P_t}
    P_t=\sqrt{A\overline{\Delta}_t+4B\left(L_0\overline{\Delta}_t+L_1^2\overline{\Delta}_t^2\right)+C},
  \end{align}
and 
\begin{align}
  \label{P}
  P_c=\sqrt{A\functiongapboundconstant+4B\left(L_0\functiongapboundconstant+L_1^2\functiongapboundconstant^2\right)+C}.
\end{align}
\end{lemma}

 We apply an induction argument to obtain the upper bound of the function value gap in the following proposition.
\begin{proposition}
  \label{constant bound}
  Under the conditions of Theorem \ref{theorem 1}, $\overline{\Delta}_t\leq \functiongapboundconstant,P_t\leq P_c, \forall t\in[T]$, hold with probability at least $1-\delta$, where  $P_t$, $P_c$ are in \eqref{P_t}, \eqref{P} and $\functiongapboundconstant$ is defined as
  \begin{align}
    \label{Gconstant}
    \functiongapboundconstant=\Gconstant.
  \end{align} 
\end{proposition}
\begin{proof}
  We assume that \eqref{inequality high probability non-convex constant} always happens and then deduce $\overline{\Delta}_t\leq \functiongapboundconstant$, $P_t\leq P_c$ for all $t\in[T]$. 
  Since \eqref{inequality high probability non-convex constant} happens with probability at least $1-\delta$, $\overline{\Delta}_t\leq \functiongapboundconstant$ and $P_t\leq P_c$, $\forall t\in[T]$,  hold with probability at least $1-\delta$. 
  It is obvious that $f(\overline{x}_1)-f^*\leq \functiongapboundconstant$. Therefore, $P_1\leq P_c$.
  Suppose that for some $l\in[T]$, $$f(\overline{x}_t)-f^*\leq \functiongapboundconstant,\quad\forall t\in[l], \quad \text{thus}, \quad P_t\leq P_c, \quad \forall t\in[l].$$ 
  By the triangle inequality of the norm function,
  \begin{align*}
    \|g_t\|\leq  \|g_t-\overline{g}_t\|+\|\overline{g}_t\|
    \leq  \sqrt{A\overline{\Delta}_t+B\|\overline{g}_t\|^2+C}+\|\overline{g}_t\|
    \leq  \sqrt{A\overline{\Delta}_t}+\left(\sqrt{B}+1\right)\|\overline{g}_t\|+\sqrt{C},
  \end{align*}
  where the second inequality follows from Assumption \ref{assumption 3} and the last inequality follows from Lemma \ref{sqrt sum}.
  Combining with Lemma \ref{lemma 6.2} and the assumption that $\overline{\Delta}_t\leq \functiongapboundconstant, \forall t\in[l]$, we have
  \begin{align} 
    \label{g_t}
    \|g_t\| \leq  \sqrt{A\functiongapboundconstant}+2\left(\sqrt{B}+1\right)\sqrt{L_0 \functiongapboundconstant+L_1^2 \functiongapboundconstant^2}+\sqrt{C}, \quad \forall t\in[l].
  \end{align}
  For simplicity, let 
  \begin{align}
    \label{Y_c}
    Y_c=\eta \left(\sqrt{A\functiongapboundconstant}+2\left(\sqrt{B}+1\right)\sqrt{L_0 \functiongapboundconstant+L_1^2 \functiongapboundconstant^2}+\sqrt{C}\right).
  \end{align}
    By the iteration step of Algorithm \ref{algorithm1}, we have
  \begin{align*}
  \|x_{t+1}-x_t\|=\theta_t\|g_t\|=\eta \|g_t\|\leq Y_c\leq 1/8L_1,
  \end{align*}
  where the last inequality follows from the restriction of $\eta$.
 Applying Lemma \ref{descent lemma}, we have 
\begin{align*}
  & f(x_{t+1})-f(x_t) \notag\\
  \leq & \la \nabla f(x_t),x_{t+1}-x_t\ra+\frac{L_0+L_1\nablaf{x_t}}{2}\|x_{t+1}-x_t\|^2\notag\\
   = & -\theta_t\la \nabla f(x_t),g_t\ra+\frac{L_0+L_1\nablaf{x_t}}{2}\theta_t^2\|g_t\|^2\notag\\
   = & -\theta_t\la \overline{g}_t+\nabla f(x_t)-\overline{g}_t,\overline{g}_t+\xi_t\ra + \frac{L_0+L_1\nablaf{x_t}}{2}\theta_t^2\|g_t\|^2\notag\\
   = & -\theta_t \|\overline{g}_t\|^2-\theta_t\la \overline{g}_t,\xi_t\ra -\theta_t \la \nabla f(x_t)-\overline{g}_t,g_t\ra+\frac{L_0+L_1\nablaf{x_t}}{2}\theta_t^2\|g_t\|^2\notag\\
  \leq & -\theta_t \|\overline{g}_t\|^2-\theta_t\la \overline{g}_t,\xi_t\ra +\theta_t \| \nabla f(x_t)-\overline{g}_t\|\|g_t\|+\frac{L_0+L_1\nablaf{x_t}}{2}\theta_t^2\|g_t\|^2,
\end{align*} 
where the first equation follows from the update rule in Algorithm \ref{algorithm1} and the last line follows from Cauchy-Schwarz inequality. 
Recall the constraints of the step size for RSAG with the constant step size and SGD, we have 
\begin{align}
  \label{constant step size constraint}
  0\leq \frac{\gamma_t-\theta_t}{\alpha_t}\leq \theta_t=\eta, \quad \forall t\in[T].
\end{align}
Combining \eqref{overline x_t-x_t} with \eqref{proposition 1.1}, \eqref{g_t}, \eqref{constant step size constraint}, we have
\begin{align}
  \label{bound overline x_t-x_t}
  \|\overline{x}_t-x_t\| \leq & \eta\varGamma_{t-1}\sumlimits{k=1}{t-1}\frac{\alpha_k}{\varGamma_k}\|g_k\|\notag\\
  \leq & \eta\left(\sqrt{A\functiongapboundconstant}+2\left(\sqrt{B}+1\right)\sqrt{L_0 \functiongapboundconstant+L_1^2 \functiongapboundconstant^2}+\sqrt{C}\right)\varGamma_{t-1}\sumlimits{k=1}{t-1}\frac{\alpha_k}{\varGamma_k}\notag\\
  = & Y_c \leq 1/8L_1.
\end{align}
Note that \eqref{bound overline x_t-x_t} holds for all $t\in[l+1]$. 
Thus, applying Definition \ref{definition 1}, we have that for all $t\in[l]$,
\begin{align}  
  \label{inequality base}
  & f(x_{t+1})-f(x_t) \notag\\
  \leq & -\theta_t \|\overline{g}_t\|^2-\theta_t\la \overline{g}_t,\xi_t\ra +\theta_t\left(L_0+L_1\nablaf{x_t}\right)\|\overline{x}_t-x_t\|\|g_t\| +\frac{L_0+L_1\nablaf{x_t}}{2}\theta_t^2\|g_t\|^2\notag\\
  \leq & -\theta_t\|\overline{g}_t\|^2-\theta_t\la\overline{g}_t,\xi_t\ra+\frac{1}{2}\left(L_0+L_1\nablaf{x_t}\right)\|\overline{x}_t-x_t\|^2+\left(L_0+L_1\nablaf{x_t}\right)\theta_t^2\|g_t\|^2\notag\\
  \leq & -\theta_t\|\overline{g}_t\|^2-\theta_t\la\overline{g}_t,\xi_t\ra+\frac{1}{2}\left(L_0+L_1\nablaf{x_t}\right)\left(1-\alpha_t\right)\varGamma_t\sumlimits{k=1}{t}\frac{\alpha_k}{\varGamma_k}\theta_k^2\|g_k\|^2\notag\\
  & + \left(L_0+L_1\nablaf{x_t}\right)\theta_t^2\|g_t\|^2,
\end{align}
where we use Young's inequality in the second inequality and apply Proposition \ref{proposition overline x_t - x_t} with \eqref{constant step size constraint} in the last inequality. 
Summing over $t\in [l]$ with $\theta_t=\eta$, we have
\begin{align}
\label{telescope}
f(x_{l+1})-f(x_1)  \leq& \frac{1}{2}\sum\limits^{l}_{t=1}\left[\lsmooth{t}\left(1-\alpha_t\right)\varGamma_t\sum\limits_{k=1}^t\frac{\alpha_k}{\varGamma_k}\eta^2\|g_k\|^2\right]\notag\\
&+\eta^2\sumlimits{t=1}{l}\lsmooth{t}\|g_t\|^2-\eta\sumlimits{t=1}{l}\|\overline{g}_t\|^2-\eta\sumlimits{t=1}{l}\la \overline{g}_t,\xi_t\ra\notag\\
=& \frac{1}{2}\eta^2 \sumlimits{t=1}{l}\left[\sumlimits{k=t}{l}\lsmooth{k}\left(1-\alpha_k\right)\varGamma_k\right]\frac{\alpha_t}{\varGamma_t}\|g_t\|^2\notag\\
&+\eta^2\sumlimits{t=1}{l}\lsmooth{t}\|g_t\|^2-\eta\sumlimits{t=1}{l}\|\overline{g}_t\|^2-\eta\sumlimits{t=1}{l}\la \overline{g}_t,\xi_t\ra.
\end{align}
Applying the triangle inequality of $\|\cdot\|$, we have that for all $t\in[l]$,
\begin{align}
  \label{nabla fx_t}
  \nablaf{x_t} \leq  \|\overline{g}_t-\nabla f(x_t)\|+\|\overline{g}_t\| 
  \leq  \left(L_0+L_1\|\overline{g}_t\|\right)\|\overline{x}_t-x_t\|+\|\overline{g}_t\|,
\end{align}
where the second inequality holds since $\|\overline{x}_t-x_t\|\leq 1/L_1$. Combining \eqref{bound overline x_t-x_t} and \eqref{nabla fx_t}, for all $t\in[l]$,
\begin{align}
  \label{nabla fx_t constant step size}
   \nablaf{x_t}
  \leq & \frac{1}{8L_1}\left(L_0+2L_1\sqrt{L_0\functiongapboundconstant+L_1^2\functiongapboundconstant^2}\right)+2\sqrt{L_0 \functiongapboundconstant+L_1^2 \functiongapboundconstant^2}\notag\\
  = & \frac{L_0}{8L_1}+\frac{9}{4} \sqrt{L_0 \functiongapboundconstant+L_1^2 \functiongapboundconstant^2}=\mathcal{M}_c,
\end{align}
where we apply Lemma \ref{lemma 6.2} to the inequality. 
Applying the triangle inequality again, we have
\begin{align}
  \label{triangle square}
  \|g_t\|^2\leq 2\|g_t-\overline{g}_t\|^2+2\|\overline{g}_t\|^2\leq 2\left(A \overline{\Delta}_t+\left(B+1\right)\|\overline{g}_t\|^2+C\right).
\end{align}
Combining  \eqref{proposition 1.2}, \eqref{inequality high probability non-convex constant}, \eqref{telescope} and \eqref{triangle square}, together with the bound of $\nablaf{x_t}$ in \eqref{nabla fx_t constant step size}, we have 
\begin{align}
  \label{4.26}
f(x_{l+1})-f(x_1)  \leq & 2\left(L_0+L_1\mathcal{M}_c\right)\eta^2\sumlimits{t=1}{l}\|g_t\|^2-\eta\sumlimits{t=1}{l}\|\overline{g}_t\|^2-\eta\sumlimits{t=1}{l}\la\overline{g}_t,\xi_t\ra\notag\\
\leq & 4\left(L_0+L_1\mathcal{M}_c\right)\eta^2\sumlimits{t=1}{l}\left(A\overline{\Delta}_t+\left(B+1\right)\|\overline{g}_t\|^2+C\right)-\eta\sumlimits{t=1}{l}\|\overline{g}_t\|^2\notag\\
&+\frac{1}{4}\eta\sumlimits{t=1}{l}\frac{P_t^2}{P_c^2}\|\overline{g}_t\|^2+3\eta P_c^2\log\frac{T}{\delta}\notag\\
\leq & 4\left(L_0+L_1\mathcal{M}_c\right)\eta^2 lA\functiongapboundconstant+4\left(L_0+L_1\mathcal{M}_c\right)\eta^2lC \notag\\
&+\left(4\left(L_0+L_1\mathcal{M}_c\right)\eta\left(B+1\right)+\frac{1}{4}-1\right)\eta\sumlimits{t=1}{l}\|\overline{g}_t\|^2+3\eta P_c^2\log\frac{T}{\delta},
\end{align}
where the last inequality holds since the assumption that $P_t\leq P_c, \forall t\in[l]$. 
Applying the restrictions of $\eta$ in \eqref{eta constant non-convex}, it is apparent that 
\begin{align*}
  f(x_{l+1})\leq f(x_1)+\frac{1}{4}\functiongapboundconstant+1+3\eta P_c^2\log\frac{T}{\delta}.
\end{align*}
Since $P_c^2 = A\functiongapboundconstant+4B\left(L_0\functiongapboundconstant+L_1^2\functiongapboundconstant^2\right)+C= \left(A+4BL_0\right)\functiongapboundconstant+4BL_1^2\functiongapboundconstant^2+C$,
\begin{align}
  \label{f x+1-f x_1 constant step size}
  f(x_{l+1}) \leq & f(x_1)+\frac{1}{4}\functiongapboundconstant+1+3\eta \left(A+4BL_0\right)\functiongapboundconstant\log\frac{T}{\delta}+12\eta BL_1^2\functiongapboundconstant^2\log\frac{T}{\delta}+3\eta C \log\frac{T}{\delta}\notag\\
  \leq & f(x_1)+\frac{1}{4}\functiongapboundconstant+1+2\sqrt{\functiongapboundconstant}\log\frac{T}{\delta}+3\eta C\log\frac{T}{\delta}     \notag\\
  \leq & f(x_1)+\frac{1}{2}\functiongapboundconstant+1+ 4\log^2\frac{T}{\delta}+3\eta C \log\frac{T}{\delta},
\end{align}
where the second inequality is due to the restrictions of $\eta$ in \eqref{eta constant non-convex} and the last inequality follows from Young's inequality. Next we will bound the gap between $f(x_{l+1})$ and $f(\overline{x}_{l+1})$. By \eqref{bound overline x_t-x_t}, we have $\|\overline{x}_{l+1}-x_{l+1}\|<1/L_1$. Using Lemma \ref{descent lemma} again,
\begin{align}
  \label{gap between f overline x_t+1 and f x_t+1 constant step size}
  f(\overline{x}_{l+1}) \leq & f(x_{l+1})+\la \overline{g}_{l+1}, \overline{x}_{l+1}-x_{l+1}\ra+\frac{L_0+L_1\|\overline{g}_{l+1}\|}{2}\|\overline{x}_{l+1}-x_{l+1}\|^2\notag\\
  \leq & f(x_{l+1})+\|\overline{g}_{l+1}\|\|\overline{x}_{l+1}-x_{l+1}\|+\frac{L_0+L_1\|\overline{g}_{l+1}\|}{2}\|\overline{x}_{l+1}-x_{l+1}\|^2,
\end{align}
where the second inequality follows from Cauchy-Schwarz inequality. Substituting \eqref{bound overline x_t-x_t} into \eqref{gap between f overline x_t+1 and f x_t+1 constant step size}, we have
\begin{align}
  \label{descent lemma 2}
  f(\overline{x}_{l+1}) \leq & f(x_{l+1})+\frac{1}{8L_1}\|\overline{g}_{l+1}\|+\frac{1}{128L_1^2}\left(L_0+L_1\|\overline{g}_{l+1}\|\right)\notag\\
  = & f(x_{l+1})+\frac{17}{128L_1}\|\overline{g}_{l+1}\|+\frac{L_0}{128L_1^2}.
\end{align}
Combining \eqref{f x+1-f x_1 constant step size} and \eqref{descent lemma 2}, and subtracting $f^*$ from both sides, we have
\begin{align}
  \overline{\Delta}_{l+1}\leq & \overline{\Delta}_1+\frac{1}{2}\functiongapboundconstant+1+4\log^2\frac{T}{\delta}+3\eta C \log\frac{T}{\delta} +\frac{17}{128L_1}\|\overline{g}_{l+1}\|+\frac{L_0}{128L_1^2}\notag\\
  \leq & \overline{\Delta}_1+\frac{1}{2}\functiongapboundconstant+1+4\log^2\frac{T}{\delta}+3\eta C \log\frac{T}{\delta} + \frac{17}{64L_1}\left(\sqrt{L_0\overline{\Delta}_{l+1}}+L_1\overline{\Delta}_{l+1}\right)+  \frac{L_0}{128L_1^2}\notag\\
  \leq &   \overline{\Delta}_1+\frac{1}{2}\functiongapboundconstant+1+4\log^2\frac{T}{\delta}+\frac{\sqrt{2C}}{\sqrt{L_0T}}\log\frac{T}{\delta} + \frac{7}{64}\overline{\Delta}_{l+1}+\frac{289L_0}{1792L_1^2}+\frac{17}{64}\overline{\Delta}_{l+1}+\frac{L_0}{128L_1^2}.
\end{align}
The second inequality follows from Lemma \ref{lemma 6.2} and Lemma \ref{sqrt sum}. The third inequality follows from Young's inequality and the constraint of $\eta$ that $\eta\leq \frac{1}{2\sqrt{\left(L_0+L_1\mathcal{M}_c\right)CT}}$, where $\mathcal{M}_c\geq \frac{L_0}{8L_1}$.
Re-arranging the above inequality, we have 
\begin{align*}
  \frac{5}{8}\overline{\Delta}_{l+1}\leq \overline{\Delta}_1+\frac{1}{2}\functiongapboundconstant+1+4\log^2\frac{T}{\delta}+\frac{\sqrt{2C}}{\sqrt{L_0T}}+\frac{303L_0}{1792L_1^2}\leq \frac{5}{8}\functiongapboundconstant.
\end{align*}
The second inequality holds since 
$\functiongapboundconstant=\Gconstant$.
Therefore, $P_{l+1}\leq P_c$ also holds. 
So the induction is complete and we obtain the desired result. 
\end{proof}
 Using the above results, we are able to prove the convergence rate of RSAG with the constant step size and SGD in the non-convex case. 
\begin{proof}[Proof of Theorem \ref{theorem 1}]
We assume \eqref{inequality high probability non-convex constant} always holds then we deduce \eqref{4.1}. Note that \eqref{inequality high probability non-convex constant} holds with probability at least $1-\delta$. Therefore, \eqref{4.1} holds with probability at least $1-\delta$.
Recalling \eqref{inequality base} and dividing $\theta_t$ on both sides, we have  
\begin{align*}
  \|\overline{g}_t\|^2\leq &\frac{f(x_t)-f(x_{t+1})}{\theta_t}+\frac{1}{2\theta_t}\left(L_0+L_1\nablaf{x_t}\right)\left(1-\alpha_t\right)\varGamma_t\sumlimits{k=1}{t}\frac{\alpha_k}{\varGamma_k}\theta_k^2\|g_k\|^2\notag\\
&+\left(L_0+L_1\nablaf{x_t}\right)\theta_t\|g_t\|^2-\la\overline{g}_t,\xi_t\ra.
\end{align*}
Summing over $t\in[T]$ with $\theta_t=\eta$,
\begin{align}
  \label{constant gt}
  \sumlimits{t=1}{T}\|\overline{g}_t\|^2\leq & \frac{1}{\eta}\sumlimits{t=1}{T}\left(f(x_t)-f(x_{t+1})\right)+\frac{\eta}{2}\sumlimits{t=1}{T}\left[\left(L_0+L_1\nablaf{x_t}\right)\left(1-\alpha_t\right)\varGamma_t\sumlimits{k=1}{t}\frac{\alpha_k}{\varGamma_k}\|g_k\|^2\right]\notag\\
  &+\eta\sumlimits{t=1}{T}\left(L_0+L_1\nablaf{x_t}\right)\|g_t\|^2-\sumlimits{t=1}{T}\la \overline{g}_t,\xi_t\ra\notag\\
  \leq &\frac{\overline{\Delta}_1}{\eta}+\frac{\eta}{2} \sumlimits{t=1}{T}\left[\sumlimits{k=t}{T}\lsmooth{k}\left(1-\alpha_k\right)\varGamma_k\right]\frac{\alpha_t}{\varGamma_t}\|g_t\|^2\notag\\
  & +\eta\sumlimits{t=1}{T}\left(L_0+L_1\nablaf{x_t}\right)\|g_t\|^2-\sumlimits{t=1}{T}\la \overline{g}_t,\xi_t\ra.
\end{align}
Applying Lemma \ref{high probability sum non-convex constant},
\begin{align*}
  \sumlimits{t=1}{T}\|\overline{g}_t\|^2\leq &\frac{\overline{\Delta}_1}{\eta}+\frac{\eta}{2} \sumlimits{t=1}{T}\left[\sumlimits{k=t}{T}\lsmooth{k}\left(1-\alpha_k\right)\varGamma_k\right]\frac{\alpha_t}{\varGamma_t}\|g_t\|^2\notag\\
  & +\eta\sumlimits{t=1}{T}\left(L_0+L_1\nablaf{x_t}\right)\|g_t\|^2+\frac{1}{4P_c^2}\sumlimits{t=1}{T}\|\overline{g}_t\|^2 P_t^2+3P_c^2\log\frac{T}{\delta}.
\end{align*}
Combining with \eqref{proposition 1.2}, \eqref{nabla fx_t constant step size}, \eqref{triangle square} and Proposition \ref{constant bound}, we have
\begin{align*}
  \sumlimits{t=1}{T}\|\overline{g}_t\|^2 \leq & \frac{\overline{\Delta}_1}{\eta}+2\eta\sumlimits{t=1}{T}\left(L_0+L_1\mathcal{M}_c\right)\|g_t\|^2+\frac{1}{4}\sumlimits{t=1}{T}\|\overline{g}_t\|^2+3P_c^2\log\frac{T}{\delta}\notag\\
  \leq & \frac{\overline{\Delta}_1}{\eta}+4\eta\left(L_0+L_1\mathcal{M}_c\right) T \left(A\functiongapboundconstant+C\right)+4\left(L_0+L_1\mathcal{M}_c\right)\left(B+1\right)\eta\sumlimits{t=1}{T}\|\overline{g}_t\|^2\notag\\
  & + \frac{1}{4}\sumlimits{t=1}{T}\|\overline{g}_t\|^2+3\left(A\functiongapboundconstant + 4B\left(L_0\functiongapboundconstant+L_1^2\functiongapboundconstant^2\right)+C\right)\log\frac{T}{\delta}.
\end{align*}
Re-arranging the above inequality with the constraint of $\eta$ in \eqref{eta constant non-convex}, we obtain 
\begin{align*}
  \frac{1}{T}\sumlimits{t=1}{T}\|\overline{g}_t\|^2\leq & \frac{4\overline{\Delta}_1}{\eta T}+ \frac{4}{\sqrt{T}}\sqrt{L_0+L_1\mathcal{M}_c}\left(\sqrt{A\functiongapboundconstant^2}+2\sqrt{C}\right)\notag\\
  &+\frac{12}{T}\left(A\functiongapboundconstant + 4B\left(L_0\functiongapboundconstant+L_1^2\functiongapboundconstant^2\right)+C\right)\log\frac{T}{\delta}.
\end{align*}
Combining with the restrictions of $\eta$ in \eqref{eta constant non-convex} again, we get the final result,
\begin{align}
  \label{5.34}
  \frac{1}{T}\sumlimits{t=1}{T}\|\overline{g}_t\|^2\leq & \frac{4}{\sqrt{T}}\sqrt{L_0+L_1\mathcal{M}_c}\left(\sqrt{A}\left(\functiongapboundconstant+4\overline{\Delta}_1\right)+2\sqrt{C}\left(1+\overline{\Delta}_1\right)\right)\notag\\
  &+ \frac{32\overline{\Delta}_1}{T}\left(L_0+L_1\mathcal{M}_c\right)\left(B+1\right) +  \frac{12\overline{\Delta}_1}{T}\left(\left(A+4BL_0\right)\functiongapboundconstant^{1/2}+4BL_1^2\functiongapboundconstant^{3/2}\right) \notag\\
  &+ \frac{32\overline{\Delta}_1}{T}L_1 \left(\sqrt{A\functiongapboundconstant}+2\left(\sqrt{B}+1\right)\sqrt{L_0 \functiongapboundconstant+L_1^2 \functiongapboundconstant^2}+\sqrt{C}\right) \notag\\
  &+\frac{12}{T}\left(A\functiongapboundconstant + 4B\left(L_0\functiongapboundconstant+L_1^2\functiongapboundconstant^2\right)+C\right)\log\frac{T}{\delta}.
\end{align}
\end{proof}

\subsection{Convex Optimization}
\label{constant convex}
In this section, we derive the convergence rate in the convex case under the generalized smoothness assumption with the constant step size, using the function value gap obtained in Proposition \ref{constant bound}. 
Since the analysis is a little bit complex, we provide the proof sketch first.
The core idea is to bound $\|x_t-x^*\|^2, \forall t\in[T]$ by induction argument.
We start by using the iteration step and obtain \eqref{x_{t+1}-x_t}.
By telescoping and applying the probabilistic inequality in Lemma \ref{high probability sum convex constant}, we could finish the induction argument.
Then, relying on \eqref{x_{t+1}-x_t}, we obtain the upper bound of $\sum_{t=1}^{T}\la g_t, \overline{x}_t-x^*\ra$ in \eqref{constant convex sum}.
Using the noise model and applying Lemma \ref{high probability sum convex constant convergence}, we have the probabilistic bound of $\sum_{t=1}^T-\la \xi_t,\overline{x}_t-x^*\ra$ in \eqref{inequality 4.39}. 
Combining the bounds of the two summations, we are able to obtain the final convergence result.

Next, we will start the mathematical proof.
\begin{lemma}
  \label{high probability sum convex constant}
  Given $T\geq 1 $ and $\delta\in \left(0,1\right)$, if Assumptions \ref{assumption 2} and \ref{assumption 3} hold, then with probability at least $1-\delta$, $\forall l\in[T]$,
  \begin{align}
    \label{inequality high probability convex constant}
    \sumlimits{t=1}{l}-\eta \la\xi_t,x_t-x^*\ra\leq \frac{\eta^2}{2D_c}\sumlimits{t=1}{l}\left(A\overline{\Delta}_t+B\|\overline{g}_t\|^2+C\right)\|x_t-x^*\|^2+\frac{3D_c}{2}\log\frac{T}{\delta}.
  \end{align}
\end{lemma}

\begin{lemma}
  \label{high probability sum convex constant convergence}
  Given $T\geq 1 $ and $\delta\in \left(0,1\right)$, for any fixed $\lambda>0$, if Assumptions \ref{assumption 2} and \ref{assumption 3} hold, then with probability at least $1-\delta$,
  \begin{align}
    \label{inequality high probability convex constant convergence}
    \sumlimits{t=1}{T}-\la \xi_t,\overline{x}_t-x^*\ra\leq \frac{3\lambda}{4}\sumlimits{t=1}{T}\left(A\overline{\Delta}_t+B\|\overline{g}_t\|^2+C\right)\|\overline{x}_t-x^*\|^2+\frac{1}{\lambda}\log\frac{1}{\delta}.
  \end{align}
\end{lemma}

Based on the above probabilistic lemmas, we could bound the term $\|x_t-x^*\|^2$, $\forall t\in[T]$, by induction.
\begin{proposition}
  \label{4.3}
 Under the conditions of Theorem \ref{theorem 2}, with probability at least $1-2\delta$, we have
  \begin{align}
   \label{theorem 2 (1)}
   \|x_t-x^*\|^2\leq D_c^2, \quad \forall t\in[T],
  \end{align}
  where 
  \begin{align}
  \label{D_c}
    D_c^2=&2\|x_1-x^*\|^2+14\mathcal{I}+2\mathcal{I}^2+18\log^2\frac{T}{\delta},
  \end{align}
  and 
  \begin{align*}
    \mathcal{I}=\frac{1}{L_0+L_1\mathcal{M}_c}\left(\frac{1}{2}\overline{\triangle}_1+\frac{3}{16}\functiongapboundconstant+\frac{3}{4}+\frac{3}{16L_1}\sqrt{A\functiongapboundconstant+4B\left(L_0\functiongapboundconstant+L_1^2\functiongapboundconstant^2\right)+C}\log\frac{T}{\delta}\right).
  \end{align*}
\end{proposition}
\begin{proof}
  We assume that \eqref{inequality high probability non-convex constant} and \eqref{inequality high probability convex constant} always happen and then we deduce \eqref{theorem 2 (1)}. 
  Since \eqref{inequality high probability non-convex constant} and \eqref{inequality high probability convex constant} each hold with probability at least $1-\delta$, it follows that \eqref{theorem 2 (1)} holds with probability at least $1-2\delta$.
  It is easy to verify that $\|x_1-x^*\|^2\leq D_c^2$. Suppose that for some $l\in [T]$,
  \begin{align}
    \label{assumption constant convex}
    \|x_t-x^*\|^2\leq D_c^2, \quad \forall t\in [l].
  \end{align}
  Using the iteration step of Algorithm \ref{algorithm1}, we have
  \begin{align}
    \label{x_{t+1}-x_t}
    \|x_{t+1}-x^*\|^2=\|x_{t}-x^*\|^2-2\theta_t \la g_t, x_t-x^*\ra+\theta_t^2\|g_t\|^2. 
  \end{align}
  Summing over $t\in[l]$ with $\theta_t=\eta$, we have
  \begin{align*}
    \|x_{l+1}-x^*\|^2=&\|x_{1}-x^*\|^2-2\eta\sumlimits{t=1}{l} \la g_t, x_t-x^*\ra+\eta^2\sumlimits{t=1}{l}\|g_t\|^2.
  \end{align*}
Also, by decomposing the middle term on the right side, we have 
\begin{align*}
  -2\eta\sumlimits{t=1}{l} \la g_t, x_t-x^*\ra=&-2\eta\sumlimits{t=1}{l}\la \overline{g}_t, x_t-x^*\ra+2\eta\sumlimits{t=1}{l}\la \overline{g}_t-g_t,x_t-x^*\ra\notag\\
  =&-2\eta\sumlimits{t=1}{l}\la \overline{g}_t, \overline{x}_t-x^*\ra-2\eta\sumlimits{t=1}{l}\la \overline{g}_t, x_t-\overline{x}_t\ra +2\eta\sumlimits{t=1}{l}\la \overline{g}_t-g_t,x_t-x^*\ra\notag\\
  \leq & -2\eta\sumlimits{t=1}{l}\la \overline{g}_t, x_t-\overline{x}_t\ra +2\eta\sumlimits{t=1}{l}\la \overline{g}_t-g_t,x_t-x^*\ra,
\end{align*}
where the above inequality follows from Lemma \ref{convexity}. Therefore,
\begin{align}
  \label{4.50}
  \|x_{l+1}-x^*\|^2\leq & \|x_{1}-x^*\|^2-2\eta\sumlimits{t=1}{l}\la \overline{g}_t, x_t-\overline{x}_t\ra+ 2\eta\sumlimits{t=1}{l}\la \overline{g}_t-g_t,x_t-x^*\ra+\eta^2\sumlimits{t=1}{l}\|g_t\|^2.
\end{align}
Applying Lemma \ref{high probability sum convex constant}, Cauchy-Schwarz inequality and Young's inequality to \eqref{4.50}, 
\begin{align*}
  \|x_{l+1}-x^*\|^2 \leq & \|x_{1}-x^*\|^2+\eta^2\sumlimits{t=1}{l}\|\overline{g}_t\|^2+\sumlimits{t=1}{l}\|x_t-\overline{x}_t\|^2\notag\\
  & +\frac{\eta^2}{D_c}\sumlimits{t=1}{l}\left(A\overline{\Delta}_t+B\|\overline{g}_t\|^2+C\right)\|x_t-x^*\|^2+3D_c\log\frac{T}{\delta}+\eta^2\sumlimits{t=1}{l}\|g_t\|^2\notag\\
  \leq & \|x_{1}-x^*\|^2+\eta^2\sumlimits{t=1}{l}\|\overline{g}_t\|^2+\sumlimits{t=1}{l}\left[\left(1-\alpha_t\right)\varGamma_t\sumlimits{k=1}{t}\frac{\alpha_k}{\varGamma_k}\eta^2\|g_k\|^2\right]\notag\\
  & +\frac{\eta^2}{D_c}\sumlimits{t=1}{l}\left(A\overline{\Delta}_t+B\|\overline{g}_t\|^2+C\right)\|x_t-x^*\|^2+3D_c\log\frac{T}{\delta}+\eta^2\sumlimits{t=1}{l}\|g_t\|^2\notag\\
  \leq & \|x_{1}-x^*\|^2+\eta^2\sumlimits{t=1}{l}\|\overline{g}_t\|^2+\eta^2\sumlimits{t=1}{l}\left[\sumlimits{k=t}{l}\left(1-\alpha_k\right)\varGamma_k\right]\frac{\alpha_t}{\varGamma_t}\|g_t\|^2\notag\\
  &+\eta^2D_c\sumlimits{t=1}{l}\left(A\overline{\Delta}_t+B\|\overline{g}_t\|^2+C\right)+3D_c\log\frac{T}{\delta}+\eta^2\sumlimits{t=1}{l}\|g_t\|^2,
\end{align*}
where the second inequality follows from Proposition \ref{proposition overline x_t - x_t} and \eqref{constant step size constraint}, and the third inequality follows from \eqref{assumption constant convex}.
By \eqref{proposition 1.2}, we know that $\left[\sumlimits{k=t}{l}\left(1-\alpha_k\right)\varGamma_k\right]\frac{\alpha_t}{\varGamma_t}\leq 2$ for all $t\leq l\leq T$. Therefore,  
\begin{align*}  
 \|x_{l+1}-x^*\|^2\leq & \|x_{1}-x^*\|^2+\eta^2\sumlimits{t=1}{l}\|\overline{g}_t\|^2+3\eta^2\sumlimits{t=1}{l}\|g_t\|^2\notag\\
  &+\eta^2D_c\sumlimits{t=1}{l}\left(A\overline{\Delta}_t+B\|\overline{g}_t\|^2+C\right)+3D_c\log\frac{T}{\delta}\notag\\
  \leq & \|x_{1}-x^*\|^2+\eta^2\sumlimits{t=1}{l}\|\overline{g}_t\|^2+6\eta^2\sumlimits{t=1}{l}\left(A\overline{\Delta}_t+\left(B+1\right)\|\overline{g}_t\|^2+C\right)\notag\\
  &+\eta^2D_c\sumlimits{t=1}{l}\left(A\overline{\Delta}_t+B\|\overline{g}_t\|^2+C\right)+3D_c\log\frac{T}{\delta}\notag\\
  \leq & \|x_{1}-x^*\|^2+6\eta^2T(A\functiongapboundconstant+C)  +   7\eta^2\left(B+1\right)\sumlimits{t=1}{l}\|\overline{g}_t\|^2\notag\\
  &+ \eta^2 D_c \left[ T\left(A\functiongapboundconstant+C\right)+B\sumlimits{t=1}{l}\|\overline{g}_t\|^2\right]+3D_c\log\frac{T}{\delta},
\end{align*}
where the second inequality follows from \eqref{triangle square} and the last inequality holds since Proposition \ref{constant bound}.
By the constraints of $\eta$, we have 
\begin{align*}
  &\eta^2T\left(A\functiongapboundconstant+C\right)+\eta^2\left(B+1\right)\sumlimits{t=1}{l}\|\overline{g}_t\|^2\notag\\
  \leq & \frac{1}{\left(L_0+L_1\mathcal{M}_c\right)}\left(\frac{1}{16}\functiongapboundconstant+\frac{1}{4}\right)+ \frac{\eta}{8\left(L_0+L_1\mathcal{M}_c\right)}\sumlimits{t=1}{l}\|\overline{g}_t\|^2\notag\\
  \leq & \frac{1}{L_0+L_1\mathcal{M}_c}\left(\frac{1}{2}\overline{\triangle}_1+\frac{3}{16}\functiongapboundconstant+\frac{3}{4}+\frac{3}{16L_1}\sqrt{A\functiongapboundconstant+4B\left(L_0\functiongapboundconstant+L_1^2\functiongapboundconstant^2\right)+C}\log\frac{T}{\delta}\right)=\mathcal{I},
\end{align*}
where the second inequality follows from the analysis of Theorem \ref{theorem 1}.
Combining the above two inequalities, we obtain that
\begin{align*}
  \|x_{l+1}-x^*\|^2  \leq & \|x_{1}-x^*\|^2+ 7\mathcal{I}+D_c \mathcal{I}+3D_c\log\frac{T}{\delta}\notag\\
  \leq & \|x_{1}-x^*\|^2+7\mathcal{I}+\frac{1}{4}D_c^2+\mathcal{I}^2+\frac{1}{4}D_c^2+9\log^2\frac{T}{\delta}\leq D_c^2,
\end{align*} 
where the second inequality follows from Young's inequality and the last equation holds since 
\begin{align*}
  D_c^2=2\|x_1-x^*\|^2+14\mathcal{I}+2\mathcal{I}^2+18\log^2\frac{T}{\delta}.
\end{align*}
So we show that the argument holds when $l+1$ and the induction is complete. 
\end{proof}
 Based on Proposition \ref{4.3}, we are able to obtain the convergence rate in the convex case. 
\begin{proof}[Proof of Theorem \ref{theorem 2}]
We assume that \eqref{inequality high probability non-convex constant}, \eqref{inequality high probability convex constant} and \eqref{inequality high probability convex constant convergence} always happen and then we deduce that the convergence rate always holds.
Since \eqref{inequality high probability non-convex constant}, \eqref{inequality high probability convex constant} and \eqref{inequality high probability convex constant convergence} hold with probability at least  $1-\delta$ separately, the convergence rate holds with probability at least $1-3\delta$.
Re-arranging \eqref{x_{t+1}-x_t} with $\theta_t=\eta, \forall t\in[T]$, we have
\begin{align*}
  \la g_t, x_t-x^*\ra=&\frac{1}{2\eta}\left(\|x_{t}-x^*\|^2-\|x_{t+1}-x^*\|^2\right)+\frac{1}{2}\eta\|g_t\|^2.
\end{align*}
Adding $\la g_t, \overline{x}_t-x_t\ra$ to both sides of the above equation, we have
\begin{align*}
  \la g_t, \overline{x}_t-x^*\ra =&\frac{1}{2\eta}\left(\|x_{t}-x^*\|^2-\|x_{t+1}-x^*\|^2\right)+\frac{1}{2}\eta\|g_t\|^2+\la g_t, \overline{x}_t-x_t\ra.  
\end{align*}
Summing over $t\in[T]$,
\begin{align}
  \label{constant convex sum}
  \sumlimits{t=1}{T}\la g_t, \overline{x}_t-x^*\ra \leq & \frac{1}{2\eta}\|x_1-x^*\|^2+\frac{\eta}{2}\sumlimits{t=1}{T}\|g_t\|^2+\sumlimits{t=1}{T}\la g_t, \overline{x}_t-x_t\ra\notag\\
  \leq & \frac{1}{2\eta}\|x_1-x^*\|^2+\frac{\eta}{2}\sumlimits{t=1}{T}\|g_t\|^2+\frac{1}{2}\sumlimits{t=1}{T}\left(\eta\|g_t\|^2+\frac{1}{\eta}\|\overline{x}_t-x_t\|^2\right)\notag\\
  \leq &  \frac{1}{2\eta}\|x_1-x^*\|^2+\eta\sumlimits{t=1}{T}\|g_t\|^2+\frac{1}{2\eta}\sumlimits{t=1}{T}\left[\left(1-\alpha_t\right)\varGamma_t\sumlimits{k=1}{t}\frac{\alpha_k}{\varGamma_k}\eta^2\|g_k\|^2\right]\notag\\
  = & \frac{1}{2\eta}\|x_1-x^*\|^2+\eta\sumlimits{t=1}{T}\|g_t\|^2+\frac{\eta}{2}\sumlimits{t=1}{T}\left[\sumlimits{k=t}{T}\left(1-\alpha_k\right)\varGamma_k\right]\frac{\alpha_t}{\varGamma_t}\|g_t\|^2\notag\\
  \leq & \frac{1}{2\eta}\|x_1-x^*\|^2+2\eta\sumlimits{t=1}{T}\|g_t\|^2,
\end{align}
where the second inequality follows from Cauchy-Schwarz inequality and Young's inequality. The third inequality follows from Proposition \ref{proposition overline x_t - x_t} and \eqref{constant step size constraint}, and the last inequality is due to \eqref{proposition 1.2}. 
Combining with \eqref{triangle square}, Lemma \ref{lemma 6.2} and Proposition \ref{constant bound},
\begin{align}
  \label{inequality 4.46}
  \sumlimits{t=1}{T}\la g_t, \overline{x}_t-x^*\ra \leq & \frac{1}{2\eta}\|x_1-x^*\|^2+4\eta\sumlimits{t=1}{T}\left(A\overline{\Delta}_t+\left(B+1\right)\|\overline{g}_t\|^2+C\right)\notag\\
  \leq & \frac{1}{2\eta}\|x_1-x^*\|^2+4\eta\sumlimits{t=1}{T}\left(A+4\left(B+1\right)\left(L_0+L_1^2\functiongapboundconstant\right)\right)\overline{\Delta}_t+4\eta TC\notag\\
  \leq & \frac{1}{2\eta}\|x_1-x^*\|^2+4\eta\sumlimits{t=1}{T}\left(A+4\left(B+1\right)\left(L_0+L_1^2\functiongapboundconstant\right)\right)\la \overline{g}_t,\overline{x}_t-x^*\ra+4\eta TC,
\end{align}
where the last inequality is due to Lemma \ref{convexity}. 
Applying Lemma \ref{high probability sum convex constant convergence} with
$$\lambda=1/\left(\left(D_c+Y_c\right)\sqrt{TC/\log\frac{1}{\delta}}+3\left(A+4B\left(L_0+L_1^2\functiongapboundconstant\right)\right)\left(D_c+Y_c\right)^2\right),$$ 
where $Y_c$ is defined in \eqref{Y_c}, we have
\begin{align}
  \label{inequality 4.39}
  \sumlimits{t=1}{T}-\la \xi_t,\overline{x}_t-x^*\ra \leq &  \frac{3\lambda}{4}\sumlimits{t=1}{T}\left(\left(A+4B\left(L_0+L_1^2\functiongapboundconstant\right)\right)\overline{\Delta}_t+C\right)\|\overline{x}_t-x^*\|^2+\frac{1}{\lambda}\log\frac{1}{\delta}\notag\\
  \leq & \frac{1}{4\left(D_c+Y_c\right)^2}\sumlimits{t=1}{T}\overline{\Delta}_t\|\overline{x}_t-x^*\|^2+\frac{3}{4\left(D_c+Y_c\right)\sqrt{T}}\sqrt{C\log\frac{1}{\delta}}\sumlimits{t=1}{T}\|\overline{x}_t-x^*\|^2\notag\\
  &+\left(D_c+Y_c\right)\sqrt{TC\log\frac{1}{\delta}}+3\left(A+4B\left(L_0+L_1^2\functiongapboundconstant\right)\right)\left(D_c+Y_c\right)^2\log\frac{1}{\delta}\notag\\
  \leq & \frac{1}{4}\sumlimits{t=1}{T}\overline{\Delta}_t+\frac{7}{4}\left(D_c+\frac{1}{8L_1}\right)\sqrt{TC\log\frac{1}{\delta}}\notag\\
  &+3\left(A+4B\left(L_0+L_1^2\functiongapboundconstant\right)\right)\left(D_c+\frac{1}{8L_1}\right)^2\log\frac{1}{\delta},
\end{align}
where  the first inequality is due to Lemma \ref{lemma 6.2} and Proposition \ref{constant bound}. The third inequality holds since $\|\overline{x}_t-x^*\|\leq \|x_t-x^*\|+\|\overline{x}_t-x_t\|\leq D_c+Y_c\leq D_c+\frac{1}{8L_1}$. Applying Lemma \ref{convexity} to \eqref{inequality 4.39},
\begin{align}
  \label{inequality 4.50}
  \sumlimits{t=1}{T}\la \overline{g}_t-g_t,\overline{x}_t-x^*\ra \leq & \frac{1}{4}\sumlimits{t=1}{T}\la \overline{g}_t,\overline{x}_t-x^*\ra+\frac{7}{4}\left(D_c+\frac{1}{8L_1}\right)\sqrt{TC\log\frac{1}{\delta}}\notag\\
  &+3\left(A+4B\left(L_0+L_1^2\functiongapboundconstant\right)\right)\left(D_c+\frac{1}{8L_1}\right)^2\log\frac{1}{\delta}.
\end{align}
Combining \eqref{inequality 4.46}, \eqref{inequality 4.50} with the constraints of $\eta$ in \eqref{eta constant convex}, we have 
\begin{align*}
  \sumlimits{t=1}{T}\la \overline{g}_t, \overline{x}_t-x^*\ra\leq & \frac{1}{\eta}\|x_1-x^*\|^2+8\eta T C +\frac{7}{2}\left(D_c+\frac{1}{8L_1}\right)\sqrt{TC\log\frac{1}{\delta}}\notag\\
  &+6\left(A+4B\left(L_0+L_1^2\functiongapboundconstant\right)\right)\left(D_c+\frac{1}{8L_1}\right)^2\log\frac{1}{\delta}.
\end{align*}
Using Jensen's inequality and Lemma \ref{convexity}, we have that
\begin{align*}
  f\left(\frac{1}{T}\sumlimits{t=1}{T}\overline{x}_t\right)-f^*\leq &\left(\frac{1}{\eta}\|x_1-x^*\|^2+6\left(A+4B\left(L_0+L_1^2\functiongapboundconstant\right)\right)\left(D_c+1/8L_1\right)^2\log\frac{1}{\delta}\right)\frac{1}{T}\notag\\
&+8\eta C+\frac{7\left(D_c+1/8L_1\right)\sqrt{C\log\frac{1}{\delta}}}{2\sqrt{T}}.
\end{align*}
Combining with the constraints of $\eta$ in \eqref{eta constant convex}, we obtain the final result,
\begin{align}
  \label{5.66}
  f\left(\frac{1}{T}\sumlimits{t=1}{T}\overline{x}_t\right)-f^*\leq &\left(6\left(A+4B\left(L_0+L_1^2\functiongapboundconstant\right)\right)\left(D_c+\frac{1}{8L_1}\right)^2\log\frac{1}{\delta}\right)\frac{1}{T}\notag\\
    &+\left(8\left(L_0+L_1\mathcal{M}_c\right)\left(B+1\right)+3\left(A+4BL_0\right)\functiongapboundconstant^{1/2}+12BL_1^2\functiongapboundconstant^{3/2}\right)\frac{\|x_1-x^*\|^2}{T}     \notag\\
    &+ 8L_1 \left(\sqrt{A\functiongapboundconstant}+2\left(\sqrt{B}+1\right)\sqrt{L_0 \functiongapboundconstant+L_1^2 \functiongapboundconstant^2}+\sqrt{C}\right) \frac{\|x_1-x^*\|^2}{T}     \notag\\
    &+16\left(A+4\left(B+1\right)\left(L_0+L_1^2\functiongapboundconstant\right)\right)  \frac{\|x_1-x^*\|^2}{T}     \notag\\
    &+ 2\left(2\sqrt{\left(L_0+L_1\mathcal{M}_c\right)A}+\sqrt{\left(L_0+L_1\mathcal{M}_c\right)C}\right)  \frac{\|x_1-x^*\|^2}{\sqrt{T}}\notag\\
    &+\left(\frac{4}{\sqrt{\left(L_0+L_1\mathcal{M}_c\right)}}+\frac{7}{2}\left(D_c+\frac{1}{8L_1}\right)\sqrt{\log\frac{1}{\delta}}\right)\frac{\sqrt{C}}{\sqrt{T}}.
\end{align}
\end{proof}

\section{Analysis of Algorithms with Adaptive Step Sizes}
\label{section 5}
In this section, we analyze the convergence rate of RSAG with the adaptive step size and AdaGrad-Norm  under the assumptions of generalized smoothness and relaxed affine variance noise.
\subsection{Non-convex Optimization}
\label{adaptive non-convex}
We apply a similar proof strategy as in Section \ref{constant non-convex}.
It is worth emphasizing that the main challenge here is the stochastic correlation between $\eta_t$ in \eqref{adaptive step size} and $g_t$.
To eliminate the correlated randomness, we introduce the following decorrelated step size,
  \begin{align}
    \label{tilde eta_t}
    \tilde{\eta}_t=\frac{\eta}{\sqrt{G_{t-1}^2+A\overline{\Delta}_t+\left(B+1\right)\|\overline{g}_t\|^2+C}},
  \end{align}
  where $G_t^2=G_0^2+\sumlimits{k=1}{t}\|g_k\|^2$.
  To obtain the bound of $f(\overline{x}_t)-f^*, \forall t\in[T]$, we apply the similar proof idea in Section \ref{constant non-convex} but use some extra technical tools specific to the adaptive step sizes.
  Note that \eqref{inequality base} and \eqref{descent lemma 2} still hold here. 
  By telescoping and applying Lemma \ref{sum log}, Lemma \ref{tilde eta} and the probabilistic inequality in Lemma \ref{lemma 6}, we obtain the bound of $f(x_{l+1})-f^*$ in \eqref{5.23}.
  Combining with \eqref{descent lemma 2}, we have \eqref{function gap adaptive} and finish the proof of bounding the function value gap.
  The method for proving the final convergence result here is much more ingenuous than that in Section \ref{constant non-convex}.
  We observe that the term $-\frac{1}{2}\sum_{t=1}^{T}\tilde{\eta}_t\|\overline{g}_t\|^2$ emerges on the right side of \eqref{function gap adaptive} and the convergence result can be obtained directly.
  
  To start with, we first present some useful lemmas. 
\begin{lemma}
  \label{agd term}
  Suppose that $\{x_t\}_{t\in[T]}$ and $\{\overline{x}_t\}_{t\in[T]}$ are generated by RSAG with the adaptive step size or AdaGrad-Norm. Then, we have
  $$\|\overline{x}_t-x_t\|\leq \eta, \quad \forall t\in[T].$$
\end{lemma}
\begin{proof}
For all $t\in[T]$,
\begin{align}
  \label{inequality 6.4}
  \eta_t\|g_t\|=\frac{\eta}{\sqrt{G_0^2+\sumlimits{k=1}{t}\|g_k\|^2}}\|g_t\|\leq \eta.
\end{align}
For both RSAG with the adaptive step size and AdaGrad-Norm, we have 
\begin{align}
  \label{adaptive step size constraint}
  0\leq \frac{\gamma_t-\theta_t}{\alpha_t}\leq \theta_t=\eta_t, \quad \forall t\in[T].
\end{align}
Combining \eqref{proposition 1.1}, \eqref{overline x_t-x_t} and \eqref{adaptive step size constraint}, we have 
\begin{align*}
  \|\overline{x}_t-x_t\|\leq  \varGamma_{t-1}\sumlimits{k=1}{t-1}\frac{\alpha_k}{\varGamma_k} \eta_k\|g_k\|
  \leq  \eta\varGamma_{t-1}\sumlimits{k=1}{t-1}\frac{\alpha_k}{\varGamma_k} = \eta.
\end{align*}
\end{proof}

Lemma \ref{agd term} provides a bound of the distance between $\overline{x}_t$ and $x_t$ for all $t\in[T]$, which ensures that we can apply the descent lemma for the objective functions in the following analysis.
Then, we present the probabilistic bound for the summation of the martingale difference sequence.

\begin{lemma}
  \label{lemma 6}
  Given $T\geq 1 $ and $\delta\in \left(0,1\right)$, if Assumptions \ref{assumption 2} and \ref{assumption 3} hold, then with probability at least $1-\delta$, $\forall l\in[T]$,
  \begin{align}
    \label{inequality high probability non-convex adaptive}
    \sumlimits{t=1}{l}-\tilde{\eta}_t \la \overline{g}_t,\xi_t\ra\leq  \frac{1}{4P_a}\sumlimits{t=1}{l}\tilde{\eta}_t\|\overline{g}_t\|^2P_t+3\eta P_a\log\frac{T}{\delta},
  \end{align}
 where $P_t$ and $P_a$ are defined in \eqref{P_t} and \eqref{P_a}, respectively. 
\end{lemma}

\begin{proposition}
  \label{5.2}
  Under the conditions of Theorem \ref{theorem 3}, $\overline{\Delta}_t\leq \functiongapboundadaptive, P_t\leq P_a, \forall t\in[T]$ hold with probability at least $1-\delta$, where $\functiongapboundadaptive$ is defined as
  \begin{align}
    \label{functiongapboundadaptive}
    \functiongapboundadaptive=\Gadaptive,
  \end{align}
  and $P_t$, $P_a$ are given in \eqref{P_t} and \eqref{P_a}, respectively.  
\end{proposition}

\begin{proof}
  Assume that \eqref{inequality high probability non-convex adaptive} always happens. From this, we deduce $\overline{\Delta}_t\leq \functiongapboundadaptive,\forall t\in[T]$. Thus $P_t\leq P_a, \forall t\in[T]$. 
  Since \eqref{inequality high probability non-convex adaptive} holds with probability at least $1-\delta$, it follows that $\overline{\Delta}_t\leq \functiongapboundadaptive$ and $P_t\leq P_a$  hold with probability at least $1-\delta$ for all $t\in[T]$.
  It is easy to verify that $f(\overline{x}_1)-f^*\leq \functiongapboundadaptive$. Therefore, $P_1\leq P_a$. Assume that for some $l\in[T]$, 
  $$f(\overline{x}_t)-f^*\leq \functiongapboundadaptive, \quad \forall t\in[l], \quad \text{thus}, \quad P_t\leq P_a, \quad \forall t\in[l].$$
Combining Lemma \ref{agd term} and the constraints of $\eta$ in \eqref{theorem 3 eta}, we have $\|\overline{x}_t-x_t\|\leq \eta \leq 1/8L_1, \forall t\in[T]$.  
Therefore, by \eqref{nabla fx_t} and Lemma \ref{lemma 6.2}, for all $t\in[l]$, 
  \begin{align}
  \label{nabla fx_t adaptive step size}
  \nablaf{x_t} \leq & \frac{1}{8L_1}\left(L_0+2L_1\sqrt{L_0\functiongapboundadaptive+L_1^2\functiongapboundadaptive^2}\right)+2\sqrt{L_0\functiongapboundadaptive+L_1^2\functiongapboundadaptive^2}\notag\\
  = & \frac{L_0}{8L_1}+\frac{9}{4}\sqrt{L_0\functiongapboundadaptive+L_1^2\functiongapboundadaptive^2}=\mathcal{M}_a.
\end{align}
Also, for all $t\in[T]$, $$\|x_{t+1}-x_t\|=\theta_t\|g_t\|=\eta_t\|g_t\|\leq \eta \leq 1/8L_1.$$
Then, summing up \eqref{inequality base} over $t\in[l]$ with $\theta_t=\eta_t$,
\begin{align*}
  f(x_{l+1})-f(x_1)\leq & \frac{1}{2}\sumlimits{t=1}{l}\left[\left(L_0+L_1\nablaf{x_t}\right)\left(1-\alpha_t\right)\varGamma_t\sumlimits{k=1}{t}\frac{\alpha_k}{\varGamma_k}\eta_k^2\|g_k\|^2\right]\notag\\
  &+\sumlimits{t=1}{l}\left(L_0+L_1\nablaf{x_t}\right)\eta_t^2\|g_t\|^2-\sumlimits{t=1}{l}\eta_t\|\overline{g}_t\|^2-\sumlimits{t=1}{l}\eta_t\la \overline{g}_t,\xi_t\ra\notag\\
  = & \frac{1}{2}\sumlimits{t=1}{l}\left[\sumlimits{k=t}{l}\left(L_0+L_1\nablaf{x_k}\right)\left(1-\alpha_k\right)\varGamma_k\right]\frac{\alpha_t}{\varGamma_t}\eta_t^2\|g_t\|^2\notag\\
  &+\sumlimits{t=1}{l}\left(L_0+L_1\nablaf{x_t}\right)\eta_t^2\|g_t\|^2-\sumlimits{t=1}{l}\eta_t\|\overline{g}_t\|^2-\sumlimits{t=1}{l}\eta_t\la \overline{g}_t,\xi_t\ra.
\end{align*}
Combining with  \eqref{proposition 1.2} and \eqref{nabla fx_t adaptive step size}, 
\begin{align}
  \label{delta}
  f(x_{l+1})-f(x_1) \leq & 2\left(L_0+L_1\mathcal{M}_a\right)\sumlimits{t=1}{l}\eta_t^2\|g_t\|^2-\sumlimits{t=1}{l}\eta_t\|\overline{g}_t\|^2-\sumlimits{t=1}{l}\eta_t\la \overline{g}_t,\xi_t\ra\notag\\
   = & \underbrace{2\left(L_0+L_1\mathcal{M}_a\right)\sumlimits{t=1}{l}\eta_t^2\|g_t\|^2}_{\text{(a)}}\underbrace{-\sumlimits{t=1}{l}\eta_t\la \overline{g}_t,g_t\ra}_{\text{(b)}}.
\end{align}
\paragraph{Term (a)} 
\begin{align}
  \label{sum square eta_t g_t}
  \sumlimits{t=1}{l}\frac{\|g_t\|^2}{G_t^2}  \leq &  \log\left(1+\sumlimits{t=1}{l}\frac{\|g_t\|^2}{G_0^2}\right)\notag\\
  \leq & \log\left(1+\frac{2\sum_{t=1}^{l}\left(A\overline{\Delta}_t+\left(B+1\right)\|\overline{g}_t\|^2+C\right)}{G_0^2}\right)\notag\\
  \leq &  \log\left(1+\frac{2T\left(A\functiongapboundadaptive+4\left(B+1\right)\left(L_0\functiongapboundadaptive+L_1^2\functiongapboundadaptive^2\right)+C\right)}{G_0^2}\right),
\end{align}
where the  first inequality follows from Lemma \ref{sum log}, the second inequality follows from \eqref{triangle square} and the last inequality follows Lemma \ref{lemma 6.2}. 
Therefore,
\begin{align}
  \label{Term A}
  2\left(L_0+L_1\mathcal{M}_a\right)\sumlimits{t=1}{l}\eta_t^2\|g_t\|^2
  =  2\left(L_0+L_1\mathcal{M}_a\right)\eta^2\sumlimits{t=1}{l}\frac{\|g_t\|^2}{G_t^2}
  \leq  2\left(L_0+L_1\mathcal{M}_a\right)\eta^2\mathcal{H},
\end{align}
where $\mathcal{H}$ is defined in \eqref{mathcal H}.
\paragraph{Term (b)}
\begin{align}
  \label{Term B}
  \underbrace{-\sumlimits{t=1}{l}\eta_t\la \overline{g}_t,g_t\ra}_{\text{(b)}} = & \underbrace{-\sumlimits{t=1}{l}\tilde{\eta}_t\|\overline{g}_t\|^2}_{\text{(b.1)}}+\underbrace{\sumlimits{t=1}{l}\tilde{\eta}_t\la\overline{g}_t,\overline{g}_t-g_t\ra}_{\text{(b.2)}}+\underbrace{\sumlimits{t=1}{l}\left(\tilde{\eta}_t-\eta_t\right)\la \overline{g}_t, g_t\ra}_{\text{(b.3)}}.
\end{align}
\paragraph{Term (b.2)}
Combining \eqref{inequality high probability non-convex adaptive} with the assumptions that $\overline{\Delta}_t\leq \functiongapboundadaptive$ and $P_t\leq P_a$, $\forall t\in[l]$, we have
\begin{align}
  \label{Term B.2}
  \sumlimits{t=1}{l}\tilde{\eta}_t \la \overline{g}_t,\overline{g}_t-g_t\ra \leq \frac{1}{4}\sumlimits{t=1}{l}\tilde{\eta}_t\|\overline{g}_t\|^2+3\eta P_a\log\frac{T}{\delta}.
\end{align}
\paragraph{Term (b.3)}
Applying Lemma \ref{tilde eta} and Cauchy-Schwarz inequality, we have 
\begin{align}
  \label{5.21}
  |\tilde{\eta}_t-\eta_t| \la \overline{g}_t, g_t \ra \leq & 2\tilde{\eta}_t\sqrt{A\overline{\Delta}_t+B\|\overline{g}_t\|^2+C} \cdot \frac{\|\overline{g}_t\|\|g_t\|}{G_t}\notag\\
  \leq & \frac{1}{4}\tilde{\eta}_t\|\overline{g}_t\|^2+4\tilde{\eta}_t \left(A\overline{\Delta}_t+B\|\overline{g}_t\|^2+C\right)\cdot\frac{\|g_t\|^2}{G_t^2}\notag\\
  \leq & \frac{1}{4}\tilde{\eta}_t\|\overline{g}_t\|^2+4\eta \sqrt{A\overline{\Delta}_t+B\|\overline{g}_t\|^2+C}\cdot\frac{\|g_t\|^2}{G_t^2},
\end{align}
where the second inequality holds since Young's inequality and the last inequality follows from the definition of $\tilde{\eta}_t$ in \eqref{tilde eta_t}.
Summing up \eqref{5.21} over $t\in[l]$ and combining \eqref{sum square eta_t g_t}, Lemma \ref{lemma 6.2} with the assumption that $\overline{\Delta}_t\leq \functiongapboundadaptive, \forall t\in[l]$, we have
\begin{align}
  \label{Term B.3}
  \sumlimits{t=1}{l}\left(\tilde{\eta}_t-\eta_t\right) \la \overline{g}_t, g_t \ra \leq & \sumlimits{t=1}{l}|\tilde{\eta}_t-\eta_t| \la \overline{g}_t, g_t \ra\notag\\
  \leq & \frac{1}{4}\sumlimits{t=1}{l}\tilde{\eta}_t\|\overline{g}_t\|^2+4\eta \sqrt{A\functiongapboundadaptive+4B\left(L_0\functiongapboundadaptive+L_1^2\functiongapboundadaptive^2\right)+C} \sumlimits{t=1}{l}\frac{\|g_t\|^2}{G_t^2}\notag\\
  \leq & \frac{1}{4}\sumlimits{t=1}{l}\tilde{\eta}_t\|\overline{g}_t\|^2+4\eta \sqrt{A\functiongapboundadaptive+4B\left(L_0\functiongapboundadaptive+L_1^2\functiongapboundadaptive^2\right)+C}\mathcal{H}.
\end{align}
Combining \eqref{delta}, \eqref{Term A}, \eqref{Term B}, \eqref{Term B.2} and \eqref{Term B.3}, we have
\begin{align}
  \label{5.23}
  f(x_{l+1})-f(x_1) \leq & -\frac{1}{2}\sumlimits{t=1}{l}\tilde{\eta}_t\|\overline{g}_t\|^2+2\left(L_0+L_1\mathcal{M}_a\right)\eta^2\mathcal{H}\notag\\
  &+3\eta\sqrt{A\functiongapboundadaptive+4B\left(L_0\functiongapboundadaptive+L_1^2\functiongapboundadaptive^2\right)+C}\log\frac{T}{\delta}\notag\\
  &+4\eta \sqrt{A\functiongapboundadaptive+4B\left(L_0\functiongapboundadaptive+L_1^2\functiongapboundadaptive^2\right)+C}\mathcal{H}\notag\\
  \leq & -\frac{1}{2}\sumlimits{t=1}{l}\tilde{\eta}_t\|\overline{g}_t\|^2+\log\frac{T}{\delta}+1,
 \end{align}
where the second inequality holds since the constraints of $\eta$ in \eqref{theorem 3 eta}.
Now we need to bound the gap between $f(\overline{x}_{l+1})$ and $f(x_{l+1})$. 
By Lemma \ref{agd term} and the restrictions of $\eta$, we know that $\|\overline{x}_{l+1}-x_{l+1}\|\leq 1/8L_1$. Therefore, \eqref{descent lemma 2} still holds here.
Combining \eqref{5.23} with \eqref{descent lemma 2}, and subtracting $f^*$ from both sides, we have
\begin{align*}
  \overline{\Delta}_{l+1} \leq & \overline{\Delta}_1-\frac{1}{2}\sumlimits{t=1}{l}\tilde{\eta}_t\|\overline{g}_t\|^2+\log\frac{T}{\delta}+1+\frac{L_0}{128L_1^2}+\frac{17}{128L_1}\|\overline{g}_{l+1}\|.
  \end{align*}
Applying Lemma \ref{lemma 6.2} and Lemma \ref{sqrt sum}, we have
\begin{align*}
   \overline{\Delta}_{l+1}\leq  &\overline{\Delta}_1-\frac{1}{2}\sumlimits{t=1}{l}\tilde{\eta}_t\|\overline{g}_t\|^2+\log\frac{T}{\delta}+1+\frac{L_0}{128L_1^2}+\frac{17}{64L_1}\sqrt{L_0\overline{\Delta}_{l+1}}+\frac{17}{64}\overline{\Delta}_{l+1}\notag\\
  \leq &  \overline{\Delta}_1-\frac{1}{2}\sumlimits{t=1}{l}\tilde{\eta}_t\|\overline{g}_t\|^2+\log\frac{T}{\delta}+1+ \frac{L_0}{128L_1^2}+\frac{7}{64}\overline{\Delta}_{l+1}+\frac{289L_0}{1792L_1^2}+\frac{17}{64}\overline{\Delta}_{l+1},
 \end{align*}
 where the second inequality follows from Young's inequality.
 Re-arranging the above inequality, we have
 \begin{align}
  \label{function gap adaptive}
  \frac{5}{8}\overline{\Delta}_{t+1}\leq \overline{\Delta}_1-\frac{1}{2}\sumlimits{t=1}{l}\tilde{\eta}_t\|\overline{g}_t\|^2+\log\frac{T}{\delta}+1+\frac{303L_0}{1792L_1^2} \leq  \frac{5}{8}\functiongapboundadaptive,
 \end{align}
where the second inequality holds since $\functiongapboundadaptive= \Gadaptive$.
Therefore, $P_{l+1}\leq P_a$ also holds.
Now we finish the induction and obtain the desired conclusion.
\end{proof}
 Based on Proposition \ref{5.2}, we are able to analyze the convergence of RSAG with the adaptive step size and AdaGrad-Norm in the non-convex case.
\begin{proof}[Proof of Theorem \ref{theorem 3}]
Since $f(\overline{x}_t)-f^*\leq \functiongapboundadaptive, \forall t\in[T]$ holds with probability at least $1-\delta$, the convergence rate also holds with probability at least $1-\delta$. By \eqref{function gap adaptive}, 
\begin{align}
  \label{inequality 5.25}
  \frac{1}{2}\sumlimits{t=1}{T}\tilde{\eta}_t\|\overline{g}_t\|^2 \leq \frac{5}{8}\functiongapboundadaptive. 
\end{align}
Also, we have that for all $t\in[T]$,
\begin{align}
  \label{inequality 5.26}
  \frac{\eta}{\tilde{\eta}_t}=&\sqrt{G_0^2 +\sumlimits{s=1}{t-1}\|g_s\|^2+A\overline{\Delta}_t+\left(B+1\right)\|\overline{g}_t\|^2+C}\notag\\
  \leq & \sqrt{G_0^2+2\sumlimits{s=1}{t-1}\left(A\overline{\Delta}_s+\left(B+1\right)\|\overline{g}_s\|^2+C\right)+A\overline{\Delta}_t+\left(B+1\right)\|\overline{g}_t\|^2+C}\notag\\
  \leq & \sqrt{G_0^2+2\left(B+1\right)\sumlimits{s=1}{T}\|\overline{g}_s\|^2+2T\left(A\functiongapboundadaptive+C\right)},
\end{align}
where the first inequality follows from \eqref{triangle square} and the second inequality is due to Proposition \ref{5.2}.
Combining \eqref{inequality 5.25} and \eqref{inequality 5.26}, we have
\begin{align*}
  \frac{1}{2}\sumlimits{t=1}{T}\|\overline{g}_t\|^2 \leq &\frac{5\functiongapboundadaptive}{8\eta}\sqrt{G_0^2+2\left(B+1\right)\sumlimits{t=1}{T}\|\overline{g}_t\|^2+2T\left(A\functiongapboundadaptive+C\right)}\notag\\
  \leq & \frac{5\functiongapboundadaptive}{8\eta}\left(G_0+\sqrt{2\left(B+1\right)\sumlimits{t=1}{T}\|\overline{g}_t\|^2}+\sqrt{2T\left(A\functiongapboundadaptive+C\right)}\right)\notag\\
  \leq &\frac{5\functiongapboundadaptive G_0}{8\eta} + \frac{1}{4}\sumlimits{t=1}{T}\|\overline{g}_t\|^2+\frac{25\functiongapboundadaptive^2}{32\eta^2}\left(B+1\right)+\frac{5\functiongapboundadaptive}{8\eta}\sqrt{2T\left(A\functiongapboundadaptive+C\right)},
\end{align*}
where the second inequality is due to Lemma \ref{sqrt sum} and the last inequality follows from Young's inequality.
Re-arranging the above inequality and dividing by $T$ on both sides, we have 
\begin{align}
  \label{result of theorem 3}
  \frac{1}{T}\sumlimits{t=1}{T}\|\overline{g}_t\|^2\leq & \left(\frac{5\functiongapboundadaptive G_0}{2\eta}+\frac{25\functiongapboundadaptive^2}{8\eta^2}\left(B+1\right)\right)\frac{1}{T}+\frac{5\functiongapboundadaptive}{2\eta\sqrt{T}}\sqrt{2\left(A\functiongapboundadaptive+C\right)}.
\end{align}
\end{proof}

\subsection{Convex Optimization}
\label{section adaptive convex}
In this section, we obtain the convergence rate for convex optimization with the constraints of $\eta$ in Theorem \ref{theorem 3}.
In the convex case, we use a decorrelated step size defined in \cite{attia2023sgd},
\begin{align}
  \label{definition of hat eta}
  \hat{\eta}_t=\frac{\eta}{\sqrt{G_{t-1}^2+\|\overline{g}_t\|^2}}.
\end{align}
In order to bound $\|x_t-x^*\|^2, \forall t\in[T]$, we apply the probabilistic inequality in Lemma \ref{Lemma 5.4} and the term $\sum_{t=1}^l\eta_{t-1}^2\|\overline{g}_t-g_t\|^2$ emerges.
We approach it by case discussion and prove that this term is of order $\mathcal{O}\left(\log\frac{T}{\delta}\right)$.
As to the convergence analysis, we adopt the similar proof strategy as that in Section \ref{constant convex}.

Let us begin the formal proof.
The following lemma is the relaxed affine variance noise version of Lemma 15 in \cite{attia2023sgd}.
\begin{lemma}
  \label{Lemma 5.4}
  Let $\mathcal{W}(x)$ be a $\mathbb{R} \rightarrow \mathbb{R}$ function defined as 
  \begin{align}
    \label{mathcal W}
    \mathcal{W}(x)=\frac{\eta\sqrt{Ax+C}}{G_0}+\eta\sqrt{B}.
  \end{align}
  Suppose that $\{x_t\}_{t\in[T]}$ is generated by AdaGrad-Norm or RSAG with the adaptive step size \eqref{adaptive step size}.
  Given $T\geq 1 $ and $\delta\in \left(0,1\right)$, if Assumptions \ref{assumption 2} and \ref{assumption 3} hold, then with probability at least $1-\delta$, $\forall l \in [T]$,
  \begin{align}
    \label{inequality 5.4}
    \sumlimits{t=1}{l}\hat{\eta}_t \la \overline{g}_t-g_t, x_t-x^*\ra \leq 2\overline{D}_l \sqrt{A_{T,\delta}\sumlimits{t=1}{l}\eta_{t-1}^2\|g_t-\overline{g}_t\|^2+\mathcal{W}_{\max}^2B_{T,\delta}},
  \end{align}  
where 
\begin{align*}
  D_l=\max_{t\leq l}\left\{\|x_t-x^*\|\right\}, \qquad \overline{D}_l=\max\left\{D_l,\eta\right\}, \notag\\
  \mathcal{W}_{\max}=\mathcal{W}(\overline{\Delta}_{\max}),\qquad \overline{\Delta}_{\max}=\max_{t\leq T} \left\{\overline{\Delta}_t \right\},
\end{align*}
and  
\begin{align}
  \label{A_T B_T}
A_{T,\delta}=16\log\left(\frac{60\log_2(4T)\log(6T)}{\delta}\right), \qquad B_{T,\delta}=16\log^2\left(\frac{60\log_2(4T)\log(6T)}{\delta}\right).
\end{align}
\end{lemma}

The next lemma is derived from Lemma \ref{lemma high probability}.

\begin{lemma}
  \label{Lemma 5.5}
  Given $T\geq 1 $ and $\delta\in \left(0,1\right)$, if Assumptions \ref{assumption 2} and \ref{assumption 3} hold, then with probability at least $1-\delta$, $\forall l\in[T]$,
  \begin{align}
    \label{inequality 5.5}
    -\sumlimits{t=1}{l}\la \overline{g}_t,\xi_t\ra \leq & \frac{1}{2P_a^2}\sumlimits{t=1}{l}\|\overline{g}_t\|^2P_t^2+\frac{3P_a^2}{2}\log\frac{T}{\delta},
  \end{align}  
  where $P_t$ and $P_a$ are defined in \eqref{P_t} and \eqref{P_a}, respectively.
\end{lemma}

 In the following sections, we define $\eta_0=\frac{\eta}{G_0}$ for simplicity.
 \begin{proposition}
  \label{proposition 5.2}
  Under the conditions of Theorem \ref{Theorem 4}, with probability at least $1-3\delta$, we have
  \begin{align}
    \label{5.38 1}
    \|x_t-x^*\|^2\leq D_a^2,\quad \forall t\in[T],
  \end{align}
  where 
  \begin{align}
    \label{D_a}
    D_a^2=2\|x_1-x^*\|^2+4R_a+10\eta^2\mathcal{H}+\frac{\eta^2}{2}+ 32\left(A_{T,\delta}R_a+\mathcal{W}_a^2B_{T,\delta}\right)  +\frac{8}{\eta^2}R_a^2,
  \end{align}
  \begin{align}
    \label{R}
    R_a=&2\eta^2\log\left(1+\frac{T\left(A\functiongapboundadaptive+4B\left(L_0\functiongapboundadaptive+L_1^2\functiongapboundadaptive^2\right)+C\right)}{2G_0^2}\right)\notag\\
    &+7\frac{\eta^2}{G_0^2}\left(A\functiongapboundadaptive+4B\left(L_0\functiongapboundadaptive+L_1^2\functiongapboundadaptive^2\right)+C\right)\log\frac{T}{\delta},
  \end{align}
   $\mathcal{W}_a=\mathcal{W}(\functiongapboundadaptive)$ with $\mathcal{W}(x)$ defined in \eqref{mathcal W}, and $\mathcal{H}$, $A_{T,\delta}$, $B_{T,\delta}$ are given in \eqref{mathcal H}, \eqref{A_T B_T}, respectively.
 \end{proposition}
\begin{proof}
We apply a method similar to that used in the proof of Proposition \ref{4.3}. Assuming that \eqref{inequality high probability non-convex adaptive}, \eqref{inequality 5.4} and  \eqref{inequality 5.5} always happen, we deduce that \eqref{5.38 1} always holds. 
Since \eqref{inequality high probability non-convex adaptive}, \eqref{inequality 5.4} and  \eqref{inequality 5.5} happen with probability at least $1-\delta$ separately, it follows that \eqref{5.38 1} holds with probability at least $1-3\delta$.
It is apparent that $\|x_1-x^*\|^2\leq D_a^2$. 
Suppose that for some $l\in[T]$, 
$$\|x_t-x^*\|^2\leq D_a^2, \quad \forall t\in[l].$$
By the iteration step of Algorithm \ref{algorithm1}, we have  
  \begin{align}
    \label{iterating step adaptive}
     \|x_{t+1}-x^*\|^2=\|x_{t}-x^*\|^2-2\theta_t \la g_t, x_t-x^*\ra+\theta_t^2\|g_t\|^2.
  \end{align}
 Summing over $t\in[l]$ with $\theta_t=\eta_t$, where $\eta_t$ are defined in \eqref{adaptive step size}, and applying \eqref{sum square eta_t g_t},
\begin{align}
  \label{5.38}
  \|x_{l+1}-x^*\|^2=&\|x_1-x^*\|^2-2\sumlimits{t=1}{l}\eta_t \la g_t,x_t-x^*\ra+\sumlimits{t=1}{l}\eta_t^2\|g_t\|^2\notag\\
  \leq & \|x_1-x^*\|^2-2\sumlimits{t=1}{l}\eta_t \la g_t,x_t-x^*\ra+\eta^2\mathcal{H}.
\end{align}
By decomposing the middle term on the right side, we have
\begin{align}
  \label{5.39}
  &-2\sumlimits{t=1}{l}\eta_t \la g_t, x_t-x^*\ra\notag\\
  =&-2\sumlimits{t=1}{l}\eta_t\la \overline{g}_t, x_t-x^*\ra+2\sumlimits{t=1}{l}\eta_t\la \overline{g}_t-g_t,x_t-x^*\ra\notag\\
  =&-2\sumlimits{t=1}{l}\eta_t\la \overline{g}_t, \overline{x}_t-x^*\ra-2\sumlimits{t=1}{l}\eta_t\la \overline{g}_t, x_t-\overline{x}_t\ra +2\sumlimits{t=1}{l}\eta_t\la \overline{g}_t-g_t,x_t-x^*\ra\notag\\
  \leq & -2\sumlimits{t=1}{l}\eta_t\la \overline{g}_t, x_t-\overline{x}_t\ra +2\sumlimits{t=1}{l}\eta_t\la \overline{g}_t-g_t,x_t-x^*\ra\notag\\
  = & \underbrace{-2\sumlimits{t=1}{l}\eta_t\la \overline{g}_t, x_t-\overline{x}_t\ra}_{\text{(i)}}+\underbrace{2\sumlimits{t=1}{l}\hat{\eta}_t\la \overline{g}_t-g_t,x_t-x^*\ra}_{\text{(ii)}}+\underbrace{2\sumlimits{t=1}{l}\left(\eta_t-\hat{\eta}_t\right)\la \overline{g}_t-g_t,x_t-x^*\ra}_{\text{(iii)}},
\end{align}
where the inequality follows from Lemma \ref{convexity}.
\paragraph{Term (i)}
Applying Cauchy-Schwarz inequality and Young's inequality,  
\begin{align}
  \label{6.56}
  -2\sumlimits{t=1}{l}\eta_t\la  \overline{g}_t, x_t-\overline{x}_t\ra\leq &\sumlimits{t=1}{l}\eta_t^2\|\overline{g}_t\|^2+\sumlimits{t=1}{l}\|\overline{x}_t-x_t\|^2\notag\\
  \leq & \sumlimits{t=1}{l}\eta_t^2\|\overline{g}_t\|^2+ \sumlimits{t=1}{l}\left[\left(1-\alpha_t\right)\varGamma_t\sumlimits{k=1}{t}\frac{\alpha_k}{\varGamma_k}\eta_k^2\|g_k\|^2\right]\notag\\
  = & \sumlimits{t=1}{l}\eta_t^2\|\overline{g}_t\|^2+\sumlimits{t=1}{l}\left[\sumlimits{k=t}{l}\left(1-\alpha_k\right)\varGamma_k\right]\frac{\alpha_t}{\varGamma_t}\eta_t^2\|g_t\|^2\notag\\
  \leq & \sumlimits{t=1}{l}\eta_t^2\|\overline{g}_t\|^2+2\sumlimits{t=1}{l}\eta_t^2\|g_t\|^2,
\end{align}
where the second inequality follows from Proposition \ref{proposition overline x_t - x_t} and \eqref{adaptive step size constraint}, and the last inequality holds since \eqref{proposition 1.2}.
By the triangle inequality, we have $\|\overline{g}_t\|^2\leq 2\|\overline{g}_t-g_t\|^2+2\|g_t\|^2$.
Therefore,
\begin{align}
  \label{Term i}
  -2\sumlimits{t=1}{l}\eta_t\la  \overline{g}_t, x_t-\overline{x}_t\ra\leq & 2\sumlimits{t=1}{l}\eta_t^2\|\overline{g}_t-g_t\|^2+4\sumlimits{t=1}{l}\eta_t^2\|g_t\|^2\notag\\
  \leq & 2\sumlimits{t=1}{l}\eta_{t-1}^2\|\overline{g}_t-g_t\|^2+4\sumlimits{t=1}{l}\eta_t^2\|g_t\|^2,
\end{align}
where the second inequality holds since $\eta_t\leq \eta_{t-1}, \forall t\in[T]$.
\paragraph{Term (ii)}
Applying Lemma \ref{Lemma 5.4}, we have 
\begin{align}
  \label{Term ii}
 2 \sumlimits{t=1}{l}\hat{\eta}_t \la \overline{g}_t-g_t,x_t-x^*\ra\leq & 4 \overline{D}_l \sqrt{A_{T,\delta}\sumlimits{t=1}{l}\eta_{t-1}^2\|g_t-\overline{g}_t\|^2+\mathcal{W}_a^2B_{T,\delta}}\notag\\
 \leq & \frac{\overline{D}_l^2}{4}+16\left(A_{T,\delta}\sumlimits{t=1}{l}\eta_{t-1}^2\|g_t-\overline{g}_t\|^2+\mathcal{W}_a^2B_{T,\delta}\right)\notag\\
 \leq & \frac{D_a^2}{4}+\frac{\eta^2}{4}+16\left(A_{T,\delta}\sumlimits{t=1}{l}\eta_{t-1}^2\|g_t-\overline{g}_t\|^2+\mathcal{W}_a^2B_{T,\delta}\right),
\end{align}
where the second inequality follows from Young's inequality and the last inequality holds since the definition of $\overline{D}_l$ and the assumption that $\|x_t-x^*\|^2\leq D_a^2, \forall t\in[l]$.
\paragraph{Term (iii)}
Applying Cauchy-Schwarz inequality and the assumption that for all $t\in[l]$, $\|x_t-x^*\|\leq D_a$, we have 
\begin{align}
  \label{Term iii}
  2\sumlimits{t=1}{l}\left(\eta_t-\hat{\eta}_t\right)\la \overline{g}_t-g_t,x_t-x^*\ra\leq & 2D_a\sumlimits{t=1}{l}|\eta_t-\hat{\eta}_t|\cdot\|\overline{g}_t-g_t\|\notag\\
  \leq & 2\eta D_a \sumlimits{t=1}{l} \frac{\|\overline{g}_t-g_t\|^2}{\sqrt{G_{t-1}^2+\|g_t\|^2}\sqrt{G_{t-1}^2+\|\overline{g}_t\|^2}}\notag\\
  \leq & 2\frac{D_a}{\eta}\sumlimits{t=1}{l}\eta_{t-1}^2\|\overline{g}_t-g_t\|^2,
\end{align}
where the second inequality holds since Lemma \ref{hat eta} and the third inequality holds since the definition of $\eta_{t}$.
Next we will bound $\sumlimits{t=1}{l}\eta_{t-1}^2\|\overline{g}_t-g_t\|^2$.
\begin{case}
\begin{align*}
  \sumlimits{i=1}{T}\|\overline{g}_i-g_i\|^2\leq6\left(A\functiongapboundadaptive+4B\left(L_0\functiongapboundadaptive+L_1^2\functiongapboundadaptive^2\right)+C\right)\log\frac{T}{\delta}.
\end{align*}
Then, we have 
\begin{align}
  \label{6}
  \sumlimits{t=1}{l}\eta_{t-1}^2\|\overline{g}_t-g_t\|^2\leq \frac{6\eta^2}{G_0^2}\left(A\functiongapboundadaptive+4B\left(L_0\functiongapboundadaptive+L_1^2\functiongapboundadaptive^2\right)+C\right)\log\frac{T}{\delta}.
\end{align}
\end{case}
\begin{case}
  \begin{align*}
    \sumlimits{i=1}{T}\|\overline{g}_i-g_i\|^2>6\left(A\functiongapboundadaptive+4B\left(L_0\functiongapboundadaptive+L_1^2\functiongapboundadaptive^2\right)+C\right)\log\frac{T}{\delta}.
  \end{align*}
  Let
  \begin{align}
    \label{k}
  k=\min\left\{t\in [T] : \sumlimits{i=1}{t}\|\overline{g}_i-g_i\|^2>6\left(A\functiongapboundadaptive+4B\left(L_0\functiongapboundadaptive+L_1^2\functiongapboundadaptive^2\right)+C\right)\log\frac{T}{\delta}\right\}.
  \end{align}
  By  \eqref{inequality 5.5}, with the assumption that for all $t\in [l]$, $\overline{\Delta}_t\leq \functiongapboundadaptive$ and $P_t\leq P_a$,
\begin{align}
  \label{lambda}
  \sumlimits{t=1}{l}\la \overline{g}_t, \overline{g}_t-g_t\ra \leq & \frac{1}{2}\sumlimits{t=1}{l}\|\overline{g}_t\|^2+\frac{3}{2}\left(A\functiongapboundadaptive+4B\left(L_0\functiongapboundadaptive+L_1^2\functiongapboundadaptive^2\right)+C\right)\log\frac{T}{\delta}.
\end{align}
If $l<k$,
\begin{align}
  \label{5.44}
  \sumlimits{t=1}{l}\eta_{t-1}^2\|\overline{g}_t-g_t\|^2\leq \frac{\eta^2}{G_0^2}\sumlimits{t=1}{l}\|\overline{g}_t-g_t\|^2\leq 6\frac{\eta^2}{G_0^2}\left(A\functiongapboundadaptive+4B\left(L_0\functiongapboundadaptive+L_1^2\functiongapboundadaptive^2\right)+C\right)\log\frac{T}{\delta}.
\end{align}
If $l\geq k$, $\forall k\leq t \leq l$,
\begin{align*}
  G_t^2=&G_0^2+\sumlimits{i=1}{t}\|\overline{g}_i\|^2+\sumlimits{i=1}{t}\|\overline{g}_i-g_i\|^2-2\sumlimits{i=1}{t}\la \overline{g}_i,\overline{g}_i-g_i\ra\notag\\
\geq & G_0^2+\sumlimits{i=1}{t} \|\overline{g}_i-g_i\|^2-3\left(A\functiongapboundadaptive+4B\left(L_0\functiongapboundadaptive+L_1^2\functiongapboundadaptive^2\right)+C\right)\log\frac{T}{\delta}\notag\\
\geq & G_0^2+\frac{1}{2}\sumlimits{i=1}{t}\|\overline{g}_i-g_i\|^2,
\end{align*}
where the first inequality follows from \eqref{lambda} and the last inequality is due to the definition of $k$ in \eqref{k}.
Hence,
\begin{align}
  \label{5.46}
  \sumlimits{t=k}{l}\eta_{t-1}^2\|\overline{g}_t-g_t\|^2=&\sumlimits{t=k}{l}\eta_t^2\|\overline{g}_t-g_t\|^2+\sumlimits{t=k}{l}\left(\eta_{t-1}^2-\eta_t^2\right)\|\overline{g}_t-g_t\|^2\notag\\
  \leq & \eta^2\sumlimits{t=k}{l}\left(\frac{\|\overline{g}_t-g_t\|^2}{G_0^2+\frac{1}{2}\sum_{i=1}^t\|\overline{g}_i-g_i\|^2}\right)+\frac{\eta^2}{G_0^2}\left(A\functiongapboundadaptive+4B\left(L_0\functiongapboundadaptive+L_1^2\functiongapboundadaptive^2\right)+C\right)\notag\\
  \leq & 2\eta^2\sumlimits{t=1}{T}\left(\frac{\|\overline{g}_t-g_t\|^2}{2G_0^2+\sum_{i=1}^t\|\overline{g}_i-g_i\|^2}\right)+\frac{\eta^2}{G_0^2}\left(A\functiongapboundadaptive+4B\left(L_0\functiongapboundadaptive+L_1^2\functiongapboundadaptive^2\right)+C\right)\notag\\
  \leq & 2\eta^2\log\left(1+\frac{T\left(A\functiongapboundadaptive+4B\left(L_0\functiongapboundadaptive+L_1^2\functiongapboundadaptive^2\right)+C\right)}{2G_0^2}\right)\notag\\
  &+\frac{\eta^2}{G_0^2}\left(A\functiongapboundadaptive+4B\left(L_0\functiongapboundadaptive+L_1^2\functiongapboundadaptive^2\right)+C\right),
\end{align}
where the first inequality holds since $\eta_t\leq \eta_{t-1}, \forall t\in[T]$, and  the last inequality follows from Lemma \ref{lemma 6.2}, Lemma \ref{sum log} and Proposition \ref{5.2}.
\end{case}
Combining \eqref{6}, \eqref{5.44} and \eqref{5.46}, we have
\begin{align}
  \label{5.47}
  \sumlimits{t=1}{l}\eta_{t-1}^2\|\overline{g}_t-g_t\|^2\leq &2\eta^2\log\left(1+\frac{T\left(A\functiongapboundadaptive+4B\left(L_0\functiongapboundadaptive+L_1^2\functiongapboundadaptive^2\right)+C\right)}{2G_0^2}\right)\notag\\
  &+ 7\frac{\eta^2}{G_0^2}\left(A\functiongapboundadaptive+4B\left(L_0\functiongapboundadaptive+L_1^2\functiongapboundadaptive^2\right)+C\right)\log\frac{T}{\delta}=R_a.
\end{align}
Therefore, combining \eqref{5.38}, \eqref{5.39}, \eqref{Term i}, \eqref{Term ii} and \eqref{Term iii}, we have
\begin{align*}
  &\|x_{l+1}-x^*\|^2\notag\\
  \leq & \|x_1-x^*\|^2+2\sumlimits{t=1}{l}\eta_{t-1}^2\|\overline{g}_t-g_t\|^2+4\sumlimits{t=1}{l}\eta_t^2\|g_t\|^2+\frac{D_a^2}{4}+\frac{\eta^2}{4}\notag\\
  &+16\left(A_{T,\delta}\sumlimits{t=1}{l}\eta_{t-1}^2\|g_t-\overline{g}_t\|^2+\mathcal{W}_a^2B_{T,\delta}\right)+2\frac{D_a}{\eta}\sumlimits{t=1}{l}\eta_{t-1}^2\|\overline{g}_t-g_t\|^2+\eta^2\mathcal{H}\notag\\
  \leq & \|x_1-x^*\|^2+2R_a+4\eta^2\mathcal{H}+\frac{D_a^2}{4}+ \frac{\eta^2}{4}+16\left(A_{T,\delta}R_a+\mathcal{W}_a^2B_{T,\delta}\right)+\frac{D_a^2}{4}+\frac{4}{\eta^2}R_a^2+\eta^2 \mathcal{H}\notag\\
  = & \|x_1-x^*\|^2+2R_a+5\eta^2\mathcal{H}+\frac{1}{2}D_a^2+\frac{\eta^2}{4}+ 16\left(A_{T,\delta}R_a+\mathcal{W}_a^2B_{T,\delta}\right)  +\frac{4}{\eta^2}R_a^2= D_a^2,
\end{align*}
where the second inequality holds since Young's inequality and  \eqref{sum square eta_t g_t}, \eqref{5.47}. 
The last equation holds since 
\begin{align*}
  D_a^2=2\|x_1-x^*\|^2+4R_a+10\eta^2\mathcal{H}+\frac{\eta^2}{2}+ 32\left(A_{T,\delta}R_a+\mathcal{W}_a^2B_{T,\delta}\right)  +\frac{8}{\eta^2}R_a^2.
\end{align*}
\end{proof}
 Next we will prove Theorem \ref{Theorem 4} based on Proposition \ref{proposition 5.2}. 
\begin{proof}[Proof of Theorem \ref{Theorem 4}]
  Assuming that \eqref{inequality high probability convex constant convergence}, \eqref{inequality high probability non-convex adaptive}, \eqref{inequality 5.4} and  \eqref{inequality 5.5} always hold, we deduce the convergence rate always holds. 
  Since \eqref{inequality high probability convex constant convergence}, \eqref{inequality high probability non-convex adaptive}, \eqref{inequality 5.4} and  \eqref{inequality 5.5} happen with probability at least  $1-\delta$ separately, the convergence rate holds with probability at least $1-4\delta$.
  Dividing $2\theta_t$ on both sides of \eqref{iterating step adaptive} and  summing over $t\in[T]$ with $\theta_t=\eta_t$,
\begin{align*}
  \sumlimits{t=1}{T}\la g_t, x_t-x^*\ra =& \sumlimits{t=1}{T}\frac{\|x_t-x^*\|^2-\|x_{t+1}-x^*\|^2}{2\eta_t}+\frac{1}{2}\sumlimits{t=1}{T}\eta_t\|g_t\|^2\notag\\
  \leq & \frac{G_0\|x_1-x^*\|^2}{2\eta}+\frac{1}{2}\sumlimits{t=1}{T}\left(\frac{1}{\eta_t}-\frac{1}{\eta_{t-1}}\right)\|x_t-x^*\|^2+\frac{1}{2}\sumlimits{t=1}{T}\eta_t\|g_t\|^2.
\end{align*}
Since  $\eta_t\leq \eta_{t-1}, \forall t\in[T]$,
\begin{align*}
\frac{1}{\eta_t}-\frac{1}{\eta_{t-1}}\leq \eta_t \left(\frac{1}{\eta_t^2}-\frac{1}{\eta_{t-1}^2}\right)=\frac{\eta_t}{\eta^2}\|g_t\|^2.
\end{align*}
Therefore,
\begin{align*}
\sumlimits{t=1}{T}\la g_t, x_t-x^*\ra \leq &\frac{G_0\|x_1-x^*\|^2}{2\eta}+\frac{1}{2}\sumlimits{t=1}{T}\eta_t\|g_t\|^2\left(1+\frac{\|x_t-x^*\|^2}{\eta^2}\right).
\end{align*}
Adding $\sum_{t=1}^T\la g_t,\overline{x}_t-x_t\ra$ to both sides, we have
\begin{align*}
  \sumlimits{t=1}{T}\la g_t,\overline{x}_t-x^*\ra\leq &\frac{G_0\|x_1-x^*\|^2}{2\eta}+\frac{1}{2}\sumlimits{t=1}{T}\eta_t\|g_t\|^2\left(1+\frac{\|x_t-x^*\|^2}{\eta^2}\right)+\sumlimits{t=1}{T}\la g_t,\overline{x}_t-x_t\ra\notag\\
  \leq & \frac{G_0\|x_1-x^*\|^2}{2\eta}+\frac{1}{2}\sumlimits{t=1}{T}\eta_t\|g_t\|^2\left(1+\frac{\|x_t-x^*\|^2}{\eta^2}\right)\notag\\
  &+\frac{1}{2}\sumlimits{t=1}{T}\eta_t\|g_t\|^2+\frac{1}{2}\sumlimits{t=1}{T}\frac{1}{\eta_t}\|\overline{x}_t-x_t\|^2\notag\\
  = & \frac{G_0\|x_1-x^*\|^2}{2\eta}+\frac{1}{2}\sumlimits{t=1}{T}\eta_t\|g_t\|^2\left(2+\frac{\|x_t-x^*\|^2}{\eta^2}\right)+\frac{1}{2}\sumlimits{t=1}{T}\frac{1}{\eta_t}\|\overline{x}_t-x_t\|^2,
\end{align*}
where the second inequality follows from Cauchy-Schwarz inequality and Young's inequality. 
Combining with \eqref{proposition 1.2}, \eqref{adaptive step size constraint} and Proposition \ref{proposition overline x_t - x_t}, we have
\begin{align}
  \label{convex sum adaptive}
  \sumlimits{t=1}{T}\la g_t,\overline{x}_t-x^*\ra \leq & \frac{G_0\|x_1-x^*\|^2}{2\eta}+\frac{1}{2}\sumlimits{t=1}{T}\eta_t\|g_t\|^2\left(2+\frac{\|x_t-x^*\|^2}{\eta^2}\right)\notag\\
  & + \frac{1}{2}\sumlimits{t=1}{T}\left[\frac{1}{\eta_t}\left(1-\alpha_t\right)\varGamma_t\sumlimits{k=1}{t}\frac{\alpha_k}{\varGamma_k}\eta_k^2\|g_k\|^2\right]\notag\\
  = & \frac{G_0\|x_1-x^*\|^2}{2\eta}+\frac{1}{2}\sumlimits{t=1}{T}\eta_t\|g_t\|^2\left(2+\frac{\|x_t-x^*\|^2}{\eta^2}\right)\notag\\
  &+\frac{1}{2}\sumlimits{t=1}{T}\left[\sumlimits{k=t}{T}\left(1-\alpha_k\right)\frac{\varGamma_k}{\eta_k}\right]\frac{\alpha_t}{\varGamma_t}\eta_t^2\|g_t\|^2\notag\\
  \leq & \frac{G_0\|x_1-x^*\|^2}{2\eta}+\frac{1}{2}\sumlimits{t=1}{T}\eta_t\|g_t\|^2\left(2+\frac{\|x_t-x^*\|^2}{\eta^2}\right)+\frac{1}{\eta_T}\sumlimits{t=1}{T}\eta_t^2\|g_t\|^2,
\end{align}
where the last inequality holds since $\frac{1}{\eta_t}\leq\frac{1}{\eta_T}, \forall t\in[T]$.
Applying \eqref{sum square eta_t g_t} and Lemma \ref{sum square} to \eqref{convex sum adaptive}, together with the conclusions that $\overline{\Delta}_t\leq \functiongapboundadaptive$ and $\|x_t-x^*\|^2 \leq D_a^2, \forall t\in[T]$, we have
\begin{align*}
  &\sumlimits{t=1}{T}\la g_t,\overline{x}_t-x^*\ra \notag\\
  \leq &  \frac{G_0\|x_1-x^*\|^2}{2\eta}+\left(2\eta+\frac{D_a^2}{\eta}\right)\sqrt{\sumlimits{t=1}{T}\|g_t\|^2} +\eta \mathcal{H} \sqrt{G_0^2+\sum_{t=1}^T\|g_t\|^2}   \notag\\
  \leq &  \frac{G_0\|x_1-x^*\|^2}{2\eta}+\left(2\eta+\frac{D_a^2}{\eta}\right)\sqrt{\sumlimits{t=1}{T}\|g_t\|^2} +\eta \mathcal{H} G_0+\eta \mathcal{H} \sqrt{\sumlimits{t=1}{T}\|g_t\|^2}\notag\\
  \leq &  \frac{G_0\|x_1-x^*\|^2}{2\eta}+\left(2\eta+\frac{D_a^2}{\eta}+\eta \mathcal{H}\right)\sqrt{2\sumlimits{t=1}{T}\left(A\overline{\Delta}_t+\left(B+1\right)\|\overline{g}_t\|^2+C\right)}+\eta \mathcal{H} G_0\notag\\
  \leq & \frac{G_0\|x_1-x^*\|^2}{2\eta}+\left(2\eta+\frac{D_a^2}{\eta}+\eta \mathcal{H}\right)\sqrt{2\sumlimits{t=1}{T}\left(A\overline{\Delta}_t+4\left(B+1\right)\left(L_0\overline{\Delta}_t+L_1^2\overline{\Delta}_t^2\right)+C\right)}\notag\\
  & + \eta \mathcal{H} G_0,
\end{align*} 
where the second inequality holds since Lemma \ref{sqrt sum}, the third inequality follows from \eqref{triangle square} and the last inequality follows from Lemma \ref{lemma 6.2}.
Applying Lemma \ref{sqrt sum} again, we have
\begin{align}
  \label{5.54}
  \sumlimits{t=1}{T}\la g_t,\overline{x}_t-x^*\ra \leq & \frac{G_0\|x_1-x^*\|^2}{2\eta}+\eta \mathcal{H} G_0\notag\\
  &+ \left(2\eta+\frac{D_a^2}{\eta}+\eta \mathcal{H}\right)\left(\sqrt{2\left(A+4\left(B+1\right)\left(L_0+L_1^2\functiongapboundadaptive\right)\right)\sumlimits{t=1}{T}\overline{\Delta}_t}+\sqrt{2TC}\right)\notag\\
  \leq & \frac{G_0\|x_1-x^*\|^2}{2\eta}+\left(2\eta+\frac{D_a^2}{\eta}+\eta \mathcal{H}\right)\sqrt{2TC} +\eta \mathcal{H}G_0\notag\\
  &+\left(2\eta+\frac{D_a^2}{\eta}+\eta \mathcal{H}\right)^2\left(A+4\left(B+1\right)\left(L_0+L_1^2\functiongapboundadaptive\right)\right)+\frac{1}{2}\sumlimits{t=1}{T}\overline{\Delta}_t\notag\\
  \leq & \frac{G_0\|x_1-x^*\|^2}{2\eta}+\left(2\eta+\frac{D_a^2}{\eta}+\eta \mathcal{H}\right)\sqrt{2TC} +\eta \mathcal{H}G_0\notag\\
  &+\left(2\eta+\frac{D_a^2}{\eta}+\eta \mathcal{H}\right)^2\left(A+4\left(B+1\right)\left(L_0+L_1^2\functiongapboundadaptive\right)\right)+\frac{1}{2}\sumlimits{t=1}{T}\la \overline{g}_t,\overline{x}_t-x^*\ra,
\end{align}
where the second inequality holds since Young's inequality and the last inequality follows from Lemma \ref{convexity}.
Applying Lemma \ref{high probability sum convex constant convergence} and letting 
$$\lambda=1/\left(\left(D_a+1/8L_1\right)\sqrt{TC/\log\frac{1}{\delta}}+3\left(A+4B\left(L_0+L_1^2\functiongapboundadaptive\right)\right)\left(D_a+1/8L_1\right)^2\right),$$ we have
\begin{align}
  \label{inequality 5.57}
  & \sumlimits{t=1}{T}\la \overline{g}_t-g_t,\overline{x}_t-x^*\ra \notag\\
  \leq &  \frac{3\lambda}{4}\sumlimits{t=1}{T}\left(\left(A+4B\left(L_0+L_1^2\functiongapboundadaptive\right)\right)\overline{\Delta}_t+C\right)\|\overline{x}_t-x^*\|^2+\frac{1}{\lambda}\log\frac{1}{\delta}\notag\\
  \leq & \frac{1}{4\left(D_a+1/8L_1\right)^2}\sumlimits{t=1}{T}\overline{\Delta}_t\|\overline{x}_t-x^*\|^2+\frac{3}{4\left(D_a+1/8L_1\right)\sqrt{T}}\sqrt{C\log\frac{1}{\delta}}\sumlimits{t=1}{T}\|\overline{x}_t-x^*\|^2\notag\\
  &+\left(D_a+1/8L_1\right)\sqrt{TC\log\frac{1}{\delta}}+3\left(A+4B\left(L_0+L_1^2\functiongapboundadaptive\right)\right)\left(D_a+1/8L_1\right)^2\log\frac{1}{\delta}\notag\\
  \leq & \frac{1}{4}\sumlimits{t=1}{T}\overline{\Delta}_t+\frac{7}{4}\left(D_a+\frac{1}{8L_1}\right)\sqrt{TC\log\frac{1}{\delta}}+3\left(A+4B\left(L_0+L_1^2\functiongapboundadaptive\right)\right)\left(D_a+\frac{1}{8L_1}\right)^2\log\frac{1}{\delta},
 \end{align}
 where the first inequality is due to Lemma \ref{lemma 6.2} and the assumption that $\overline{\Delta}_t\leq \functiongapboundadaptive, \forall t\in[T]$. The third inequality holds since $\|\overline{x}_t-x^*\|\leq \|x_t-x^*\|+\|\overline{x}_t-x_t\|\leq  D_a+\frac{1}{8L_1}$.
 Applying Lemma \ref{convexity} to \eqref{inequality 5.57},
 \begin{align}
  \label{5.56}
  \sumlimits{t=1}{T}\la \overline{g}_t-g_t,\overline{x}_t-x^*\ra \leq & \frac{1}{4}\sumlimits{t=1}{T}\la \overline{g}_t,\overline{x}_t-x^*\ra +\frac{7}{4}\left(D_a+\frac{1}{8L_1}\right)\sqrt{TC\log\frac{1}{\delta}}\notag\\
  &+3\left(A+4B\left(L_0+L_1^2\functiongapboundadaptive\right)\right)\left(D_a+\frac{1}{8L_1}\right)^2\log\frac{1}{\delta}.
 \end{align}
Combining \eqref{5.54} and \eqref{5.56}, we have
\begin{align*}
  \sumlimits{t=1}{T}\la \overline{g}_t,\overline{x}_t-x^*\ra\leq &\frac{1}{4}\sumlimits{t=1}{T}\la \overline{g}_t,\overline{x}_t-x^*\ra +\frac{7}{4}\left(D_a+\frac{1}{8L_1}\right)\sqrt{TC\log\frac{1}{\delta}}\notag\\
  &+3\left(A+4B\left(L_0+L_1^2\functiongapboundadaptive\right)\right)\left(D_a+\frac{1}{8L_1}\right)^2\log\frac{1}{\delta}\notag\\
  &+ \frac{G_0\|x_1-x^*\|^2}{2\eta}+\left(2\eta+\frac{D_a^2}{\eta}+\eta \mathcal{H}\right)\sqrt{2TC} +\eta \mathcal{H}G_0\notag\\
  &+\left(2\eta+\frac{D_a^2}{\eta}+\eta \mathcal{H}\right)^2\left(A+4\left(B+1\right)\left(L_0+L_1^2\functiongapboundadaptive\right)\right)+\frac{1}{2}\sumlimits{t=1}{T}\la \overline{g}_t,\overline{x}_t-x^*\ra.
\end{align*}
Re-arranging the above inequality,
\begin{align*}
  \sumlimits{t=1}{T}\la \overline{g}_t,\overline{x}_t-x^*\ra\leq & 7\left(D_a+\frac{1}{8L_1}\right)\sqrt{TC\log\frac{1}{\delta}}+12\left(A+4B\left(L_0+L_1^2\functiongapboundadaptive\right)\right)\left(D_a+\frac{1}{8L_1}\right)^2\log\frac{1}{\delta}\notag\\
  &+ \frac{2G_0\|x_1-x^*\|^2}{\eta}+4\left(2\eta+\frac{D_a^2}{\eta}+\eta \mathcal{H}\right)\sqrt{2TC} +4\eta \mathcal{H}G_0\notag\\
  &+  4\left(2\eta+\frac{D_a^2}{\eta}+\eta \mathcal{H}\right)^2\left(A+4\left(B+1\right)\left(L_0+L_1^2\functiongapboundadaptive\right)\right).
\end{align*} 
Using Jensen's inequality and Lemma \ref{convexity}, we have
\begin{align}
  \label{6.77}
  f\left(\frac{1}{T}\sumlimits{t=1}{T}\overline{x}_t\right)-f^*\leq & \left(12\left(A+4B\left(L_0+L_1^2\functiongapboundadaptive\right)\right)\left(D_a+\frac{1}{8L_1}\right)^2\log\frac{1}{\delta}+\frac{2G_0\|x_1-x^*\|^2}{\eta}\right)\frac{1}{T}\notag\\
  &+\left(4\left(2\eta+\frac{D_a^2}{\eta}+\eta \mathcal{H}\right)^2\left(A+4\left(B+1\right)\left(L_0+L_1^2\functiongapboundadaptive\right)\right)+4\eta \mathcal{H}G_0\right)\frac{1}{T}\notag\\
  & + \frac{7\left(D_a+\frac{1}{8L_1}\right)\sqrt{C\log\frac{1}{\delta}}+4\left(2\eta+\frac{D_a^2}{\eta}+\eta \mathcal{H}\right)\sqrt{2C}}{\sqrt{T}}.
\end{align}

\end{proof}
\section{Analysis of Algorithms with Adaptive Step Sizes under the Smoothness Condition}
\label{section 6}
In this section, we present the convergence rate of RSAG with the adaptive step size and AdaGrad-Norm under the smoothness condition, and demonstrate that prior knowledge of the problem parameters is not required to tune the step size in this case.
\subsection{Non-convex Optimization}
The proof strategy in this section has many similarities to the proof in Section \ref{adaptive non-convex}, with a subtle distinction highlighted in Lemma \ref{sum function value gap} and Lemma \ref{sum log smooth}.
Concretely, we start with $\|\overline{x}_t-\overline{x}_{t-1}\|\leq 3\eta$ obtained in \eqref{6.7} and prove that the term $\sum_{t=1}^T\overline{\Delta}_t$ is of order $\mathcal{O}\left(T^3\right)$, which is independent of the bound for the function value gap.
Thus, the problem parameters are not required to tune the step size.

Let us begin the formal analysis.
\begin{lemma}
  \label{lemma 6.1}
  Given $T\geq 1 $ and $\delta\in \left(0,1\right)$, if $f(x)$ is an $L$-smooth function and Assumptions \ref{assumption 2}, \ref{assumption 3} hold, then with probability at least $1-\delta$, $\forall l\in[T]$,
  \begin{align}
    \label{inequality high probability non-convex adaptive smooth}
    \sumlimits{t=1}{l}-\tilde{\eta}_t \la \overline{g}_t,\xi_t\ra\leq  \frac{1}{4Q}\sumlimits{t=1}{l}\tilde{\eta}_t\|\overline{g}_t\|^2Q_t+3\eta Q\log\frac{T}{\delta},
  \end{align}
 where 
 \begin{align}
  \label{6.3}
  Q_t=\sqrt{A\overline{\Delta}_t+2BL\overline{\Delta}_t+C},
 \end{align}
\begin{align}
  \label{definition 6.4}
Q=\sqrt{A\functiongapboundsmooth+2BL\functiongapboundsmooth+C},
\end{align}
and $\functiongapboundsmooth$ is defined in \eqref{functiongapboundsmooth}.
\end{lemma}

\begin{lemma}
  \label{sum function value gap}
  Suppose that  $f(x)$ is an $L$-smooth function. Then, for RSAG with the adaptive step size and AdaGrad-Norm, we have
  \begin{align*}
    \sumlimits{t=1}{T}\overline{\Delta}_t\leq \overline{\Delta}_1 T+\left(3\eta\|\overline{g}_1\|+\frac{9L\eta^2}{2}\right)T^2+\frac{9L\eta^2}{2}T^3.
  \end{align*}
\end{lemma}
\begin{proof}
  Applying the triangle inequality and the definition of $L$-smoothness in \eqref{definition L smoothness}, we have 
  \begin{align}
    \label{6.6}
    \|\overline{g}_t\|\leq \|\overline{g}_{t-1}\|+\|\overline{g}_t-\overline{g}_{t-1}\|\leq \|\overline{g}_{t-1}\|+L\|\overline{x}_t-\overline{x}_{t-1}\|.
  \end{align}
  Using the triangle inequality again, we have that for all $2\leq t\leq T, t\in \mathbb{N}$,
  \begin{align}
    \label{6.7}
    \|\overline{x}_t-\overline{x}_{t-1}\| = & \|\left(\overline{x}_t-x_t\right)-\left(\overline{x}_{t-1}-x_{t-1}\right)+\left(x_t-x_{t-1}\right)\|\notag\\
    \leq & \|\overline{x}_t-x_t\|+\|\overline{x}_{t-1}-x_{t-1}\|+\|x_t-x_{t-1}\|\leq 3\eta,
  \end{align}
  where the last inequality follows from Lemma \ref{agd term} and the iteration step.
Combining \eqref{6.6} and \eqref{6.7}, we have that 
\begin{align}
  \label{6.8}
  \|\overline{g}_t\|\leq \|\overline{g}_{t-1}\|+3L\eta \leq \|\overline{g}_1\|+3L\eta\left(t-1\right).
\end{align}
By Lemma \ref{descent lemma} with $L_0=L$ and $L_1=0$, we have that for all $s\in[T]$,
\begin{align}
  \label{7.8}
  f(\overline{x}_{s+1})\leq & f(\overline{x}_s)+\la \overline{g}_s, \overline{x}_{s+1}-\overline{x}_s\ra+\frac{L}{2}\|\overline{x}_{s+1}-\overline{x}_s\|^2\notag\\
  \leq & f(\overline{x}_s)+\|\overline{g}_s\|\|\overline{x}_{s+1}-\overline{x}_s\|+\frac{L}{2}\|\overline{x}_{s+1}-\overline{x}_s\|^2\notag\\
  \leq & f(\overline{x}_s)+3\eta\left(\|\overline{g}_1\|+3L\eta\left(s-1\right)\right)+\frac{9L\eta^2}{2},
\end{align}
where the second inequality follows from Cauchy-Schwarz inequality and the last inequality holds since \eqref{6.7} and \eqref{6.8}.
Subtracting $f^*$ from both sides of \eqref{7.8} and summing over $s\in[t]$, we have
\begin{align*}
  \overline{\Delta}_{t+1}\leq \overline{\Delta}_1+\left(3\eta\|\overline{g}_1\|+\frac{9L\eta^2}{2}\right)t+9L\eta^2\frac{t\left(t-1\right)}{2}.
\end{align*} 
Summing over $t\in[T]$, we have
\begin{align*}
  \sumlimits{t=1}{T}\overline{\Delta}_t\leq & \sumlimits{t=1}{T}\overline{\Delta}_1  + \left(3\eta\|\overline{g}_1\|+\frac{9L\eta^2}{2}\right)\sumlimits{t=1}{T}\left(t-1\right)+\frac{9L\eta^2}{2}\sumlimits{t=1}{T}\left(t-1\right)\left(t-2\right)\notag\\
  \leq & \overline{\Delta}_1 T+\left(3\eta\|\overline{g}_1\|+\frac{9L\eta^2}{2}\right)T^2+\frac{9L\eta^2}{2}T^3.
\end{align*}
\end{proof}

\begin{lemma}
  \label{sum log smooth}
  Under the same conditions of Lemma \ref{sum function value gap}, we have 
  \begin{align*}
    \sumlimits{t=1}{T}\frac{\|g_t\|^2}{G_t^2}\leq \mathcal{F},
  \end{align*}
  where 
  \begin{align}
    \label{f}
    \mathcal{F}=\log\left(1+\left(\left(X\overline{\Delta}_1+2C\right)T+X\left(3\eta\|\overline{g}_1\|+\frac{9L\eta^2}{2}\right)T^2+\frac{9L\eta^2}{2}XT^3\right)/G_0^2\right),
  \end{align}
and 
\begin{align}
  \label{x}
  X=2A+4LB+4L.
\end{align}
\end{lemma}
\begin{proof}
  Summing up \eqref{triangle square} over $t\in[T]$ and applying Lemma \ref{smooth function value gap}, we have
\begin{align}
  \label{sum stochastic gradient 1}
  \sumlimits{t=1}{T} \|g_t\|^2 \leq & 2\sumlimits{t=1}{T}\left(A\overline{\Delta}_t+\left(B+1\right)\|\overline{g}_t\|^2+C\right)\notag
  \leq  2\sumlimits{t=1}{T}\left(\left(A+2LB+2L\right)\overline{\Delta}_t+C\right)\notag\\
  = & X\sumlimits{t=1}{T}\overline{\Delta}_t+2CT.
\end{align}
Combining \eqref{sum stochastic gradient 1} and Lemma \ref{sum function value gap}, 
\begin{align}
  \label{sum g_t}
  \sumlimits{t=1}{T} \|g_t\|^2 \leq \left(X\overline{\Delta}_1+2C\right)T+X\left(3\eta\|\overline{g}_1\|+\frac{9L\eta^2}{2}\right)T^2+\frac{9L\eta^2}{2}XT^3.
\end{align}
Then, substituting \eqref{sum g_t}   into Lemma \ref{sum log}, we have
\begin{align*}
  \sumlimits{t=1}{T}\frac{\|g_t\|^2}{G_t^2} & \leq \log\left(1+\frac{ \sum_{t=1}^{T}\|g_t\|^2  }{G_0^2}\right) \leq  \mathcal{F}.
\end{align*}
\end{proof}

\begin{proposition}
  \label{proposition 6.1}
  Under the same conditions of Theorem \ref{theorem 5}, we have that with probability at least $1-\delta$,
  \begin{align}
    \label{7.20}
    \overline{\Delta}_t\leq \functiongapboundsmooth, \quad\forall t\in[T],
  \end{align}
where 
\begin{align}
  \label{functiongapboundsmooth}
  \functiongapboundsmooth=&4\overline{\Delta}_1+8L\eta^2\mathcal{F}+36\eta^2\left(A+2BL\right)\log^2\frac{T}{\delta}+12\eta\sqrt{C}\log\frac{T}{\delta}\notag\\
  &+64\eta^2\left(A+2BL\right)\mathcal{F}^2+16\eta\sqrt{C}\mathcal{F}+10L\eta^2.
\end{align}
\end{proposition}

\begin{proof}
  We will apply an induction argument in the proof. 
  Suppose that \eqref{inequality high probability non-convex adaptive smooth} always happens and then we deduce $\overline{\Delta}_t\leq \functiongapboundsmooth, \forall t\in[T]$, always holds. 
  Since \eqref{inequality high probability non-convex adaptive smooth} happens with probability at least $1-\delta$, it follows that $\overline{\Delta}_t\leq \functiongapboundsmooth, \forall t\in[T]$, holds with probability at least $1-\delta$.
  It is apparent that $f(\overline{x}_1)-f^*\leq \functiongapboundsmooth$. 
  Assume that for some $l\in[T]$, 
  $$f(\overline{x}_t)-f^*\leq \functiongapboundsmooth, \quad \forall t\in[l].$$
  We view $L$-smoothness as a special case of  $(L_0,L_1)$-generalized smoothness, i.e., $L_0=L$ and $L_1=0$. Therefore, by \eqref{delta},
  \begin{align}
    \label{6.4}
    f(x_{l+1})-f(x_1)\leq \underbrace{2L\sumlimits{t=1}{l}\eta_t^2\|g_t\|^2}_{\text{(1)}} \underbrace{-\sumlimits{t=1}{l}\eta_t\la \overline{g}_t,g_t\ra}_{\text{(2)}}.
  \end{align}
 \paragraph{Term (1)}
Applying Lemma \ref{sum log smooth}, we have that 
\begin{align}
  \label{Term A smooth}
  2L\sumlimits{t=1}{l}\eta_t^2\|g_t\|^2 \leq 2L\eta^2 \mathcal{F},
\end{align}
where $\mathcal{F}$ is defined in \eqref{f}.
\paragraph{Term (2)}
Recall the definition of $\tilde{\eta}_t$ in \eqref{tilde eta_t}.
\begin{align}
  \label{Term B smooth}
  \underbrace{-\sumlimits{t=1}{l}\eta_t\la \overline{g}_t,g_t\ra}_{(2)} = & \underbrace{-\sumlimits{t=1}{l}\tilde{\eta}_t\|\overline{g}_t\|^2}_{(2.1)}+\underbrace{\sumlimits{t=1}{l}\tilde{\eta}_t\la\overline{g}_t,\overline{g}_t-g_t\ra}_{(2.2)}+\underbrace{\sumlimits{t=1}{l}\left(\tilde{\eta}_t-\eta_t\right)\la \overline{g}_t, g_t\ra}_{(2.3)}.
\end{align}
\paragraph{Term (2.2)}
By \eqref{inequality high probability non-convex adaptive smooth}, we have 
\begin{align*}
  \sumlimits{t=1}{l}\tilde{\eta}_t \la \overline{g}_t,\overline{g}_t-g_t\ra \leq \frac{1}{4Q}\sumlimits{t=1}{l}\tilde{\eta}_t\|\overline{g}_t\|^2Q_t+3\eta Q\log\frac{T}{\delta}.
\end{align*}
Since the assumption that $\overline{\Delta}_t\leq \functiongapboundsmooth, \forall t\in[l]$, we have $Q_t\leq Q, \forall t\in[l]$, and 
\begin{align*}
  \sumlimits{t=1}{l}\tilde{\eta}_t \la \overline{g}_t,\overline{g}_t-g_t\ra \leq & \frac{1}{4}\sumlimits{t=1}{l}\tilde{\eta}_t\|\overline{g}_t\|^2+3\eta Q\log\frac{T}{\delta}.
\end{align*}
Note that $Q=\sqrt{\left(A+2BL\right)\functiongapboundsmooth+C}\leq \sqrt{\left(A+2BL\right)\functiongapboundsmooth}+\sqrt{C}$. Using Young's inequality,
\begin{align}
  \label{Term B.2 smooth}
  \sumlimits{t=1}{l}\tilde{\eta}_t \la \overline{g}_t,\overline{g}_t-g_t\ra \leq & \frac{1}{4}\sumlimits{t=1}{l}\tilde{\eta}_t\|\overline{g}_t\|^2+3\eta\sqrt{\left(A+2BL\right)\functiongapboundsmooth}\log\frac{T}{\delta}+3\eta\sqrt{C}\log\frac{T}{\delta}\notag\\
  \leq & \frac{1}{4}\sumlimits{t=1}{l}\tilde{\eta}_t\|\overline{g}_t\|^2+\frac{\functiongapboundsmooth}{4}+9\eta^2\left(A+2BL\right)\log^2\frac{T}{\delta}+3\eta\sqrt{C}\log\frac{T}{\delta}.
\end{align}
\paragraph{Term (2.3)}
The estimation here is similar to  \eqref{Term B.3} except for the bound of $\|\overline{g}_t\|^2$. 
Concretely,
\begin{align*}
  \sumlimits{t=1}{l}\left(\tilde{\eta}_t-\eta_t\right)\la \overline{g}_t, g_t\ra\leq \sumlimits{t=1}{l}|\tilde{\eta}_t-\eta_t| \la \overline{g}_t, g_t \ra \leq & \frac{1}{4}\sumlimits{t=1}{l}\tilde{\eta}_t\|\overline{g}_t\|^2+4\eta \sqrt{\left(A+2BL\right)\functiongapboundsmooth+C} \sumlimits{t=1}{l}\frac{\|g_t\|^2}{G_t^2}.
\end{align*}
Applying Lemma \ref{sum log smooth},
\begin{align}
  \label{Term B.3 smooth}
  \sumlimits{t=1}{l}\left(\tilde{\eta}_t-\eta_t\right) \la \overline{g}_t, g_t \ra \leq & \frac{1}{4}\sumlimits{t=1}{l}\tilde{\eta}_t\|\overline{g}_t\|^2+4\eta \sqrt{\left(A+2BL\right)\functiongapboundsmooth+C}\mathcal{F}\notag\\
  \leq & \frac{1}{4}\sumlimits{t=1}{l}\tilde{\eta}_t\|\overline{g}_t\|^2+4\eta\left(\sqrt{\left(A+2BL\right)\functiongapboundsmooth}+\sqrt{C}\right)\mathcal{F}\notag\\
  \leq & \frac{1}{4}\sumlimits{t=1}{l}\tilde{\eta}_t\|\overline{g}_t\|^2+\frac{\functiongapboundsmooth}{4}+16\eta^2\left(A+2BL\right)\mathcal{F}^2+4\eta\sqrt{C}\mathcal{F},
\end{align}
where the second inequality follows from Lemma \ref{sqrt sum} and the last inequality holds since Young's inequality.
Combining \eqref{6.4}, \eqref{Term A smooth}, \eqref{Term B smooth}, \eqref{Term B.2 smooth} and \eqref{Term B.3 smooth}, we have that 
\begin{align}
  \label{6.11}
  f(x_{l+1})-f(x_1)\leq &  -\frac{1}{2}\sumlimits{t=1}{l}\tilde{\eta}_t\|\overline{g}_t\|^2+2L\eta^2\mathcal{F}+ \frac{\functiongapboundsmooth}{2}+ 9\eta^2\left(A+2BL\right)\log^2\frac{T}{\delta}+3\eta\sqrt{C}\log\frac{T}{\delta}\notag\\
  & +16\eta^2\left(A+2BL\right)\mathcal{F}^2+4\eta\sqrt{C}\mathcal{F}.
\end{align}
Next we will bound the gap between $f(x_{l+1})$ and $f(\overline{x}_{l+1})$.
Applying Lemma \ref{descent lemma} with $L_0=L$ and $L_1=0$ again,
\begin{align}
  \label{l smooth descent lemma}
  f(\overline{x}_{l+1}) \leq & f(x_{l+1})+\la \overline{g}_{l+1}, \overline{x}_{l+1}-x_{l+1}\ra+\frac{L}{2}\|\overline{x}_{l+1}-x_{l+1}\|^2\notag\\
  \leq & f(x_{l+1})+\|\overline{g}_{l+1}\|\|\overline{x}_{l+1}-x_{l+1}\|+\frac{L}{2}\|\overline{x}_{l+1}-x_{l+1}\|^2,
\end{align}
where the second inequality follows from Cauchy-Schwarz inequality.
Combining with Lemma \ref{agd term} and Lemma \ref{smooth function value gap}, we have
\begin{align}
  \label{6.27}
  f(\overline{x}_{l+1}) \leq & f(x_{l+1})+\eta\sqrt{2L\overline{\Delta}_{l+1}}+\frac{L\eta^2}{2}
  \leq  f(x_{l+1})+\frac{\overline{\Delta}_{l+1}}{4}+\frac{5L\eta^2}{2}.
\end{align}
Then, substituting \eqref{6.11} into \eqref{6.27} and subtracting $f^*$ from both sides, we have
\begin{align}
  \label{6.28}
  \frac{3}{4}\overline{\Delta}_{l+1}\leq & \overline{\Delta}_1  -\frac{1}{2}\sumlimits{t=1}{l}\tilde{\eta}_t\|\overline{g}_t\|^2+2L\eta^2\mathcal{F}+ \frac{\functiongapboundsmooth}{2}+ 9\eta^2\left(A+2BL\right)\log^2\frac{T}{\delta}+3\eta\sqrt{C}\log\frac{T}{\delta}\notag\\
  & +16\eta^2\left(A+2BL\right)\mathcal{F}^2+4\eta\sqrt{C}\mathcal{F}+\frac{5L\eta^2}{2}\leq \frac{3}{4}\functiongapboundsmooth,
\end{align}
where the last inequality holds since 
\begin{align*}
  \functiongapboundsmooth=&4\overline{\Delta}_1+8L\eta^2\mathcal{F}+36\eta^2\left(A+2BL\right)\log^2\frac{T}{\delta}+12\eta\sqrt{C}\log\frac{T}{\delta}\notag\\
  &+64\eta^2\left(A+2BL\right)\mathcal{F}^2+16\eta\sqrt{C}\mathcal{F}+10L\eta^2.
\end{align*}
  Therefore, $\overline{\Delta}_{l+1}$ satisfies the assumption and the induction is complete. 
\end{proof}
 Based on Proposition \ref{proposition 6.1}, we are able to obtain the convergence rate.
\begin{proof}[Proof of Theorem \ref{theorem 5}]
Since \eqref{7.20} holds with probability at least $1-\delta$, the convergence rate also holds with probability at least $1-\delta$. By \eqref{6.28}, 
\begin{align}
  \label{6.29}
  \frac{1}{2}\sumlimits{t=1}{T}\tilde{\eta}_t\|\overline{g}_t\|^2\leq \frac{3}{4}\functiongapboundsmooth.
\end{align}
Recalling the definition of $\tilde{\eta}_t$ in \eqref{tilde eta_t}, we have that for all $t\in[T]$,
\begin{align}
  \label{6.30}
  \frac{\eta}{\tilde{\eta}_t}=&\sqrt{G_0^2+\sumlimits{s=1}{t-1}\|g_s\|^2+A\overline{\Delta}_t+\left(B+1\right)\|\overline{g}_t\|^2+C}\notag\\
  \leq & \sqrt{G_0^2+2\sumlimits{s=1}{t-1}\left(A\overline{\Delta}_s+\left(B+1\right)\|\overline{g}_s\|^2+C\right)+A\overline{\Delta}_t+\left(B+1\right)\|\overline{g}_t\|^2+C}\notag\\
  \leq & \sqrt{G_0^2+2\left(B+1\right)\sumlimits{s=1}{T}\|\overline{g}_s\|^2+2T\left(A\functiongapboundsmooth+C\right)},
\end{align}
where the first inequality follows from \eqref{triangle square} and the second inequality is due to Lemma \ref{smooth function value gap}.
Combining \eqref{6.29} and \eqref{6.30},
\begin{align*}
  \frac{1}{2}\sumlimits{t=1}{T}\|\overline{g}_t\|^2\leq &\frac{3}{4\eta}\functiongapboundsmooth\sqrt{G_0^2+2\left(B+1\right)\sumlimits{t=1}{T}\|\overline{g}_t\|^2+2T\left(A\functiongapboundsmooth+C\right)}\notag\\
 \leq & \frac{3\functiongapboundsmooth}{4\eta}\left(G_0+\sqrt{2\left(B+1\right)\sumlimits{t=1}{T}\|\overline{g}_t\|^2}+\sqrt{2T\left(A\functiongapboundsmooth+C\right)}\right)\notag\\
 \leq & \frac{3\functiongapboundsmooth}{4\eta}G_0+\frac{1}{4}\sumlimits{t=1}{T}\|\overline{g}_t\|^2+\frac{9\functiongapboundsmooth^2}{8\eta^2}\left(B+1\right)+\frac{3\functiongapboundsmooth}{4\eta}\sqrt{2T\left(A\functiongapboundsmooth+C\right)},
\end{align*}
where the second inequality is due to Lemma \ref{sqrt sum} and the last inequality follows from Young's inequality.
Re-arranging the above inequality, we have
\begin{align}
  \label{7.37}
  \frac{1}{T}\sumlimits{t=1}{T}\|\overline{g}_t\|^2\leq & \left(\frac{3\functiongapboundsmooth G_0}{\eta}+\frac{9\functiongapboundsmooth^2}{2\eta^2}\left(B+1\right)\right)\frac{1}{T}+ \frac{3\functiongapboundsmooth}{\eta\sqrt{T}}\sqrt{2\left(A\functiongapboundsmooth+C\right)}.
\end{align}
\end{proof}
\subsection{Convex Optimization}
Based on Proposition \ref{proposition 6.1}, we could provide the proof of Theorem \ref{theorem 6} along the strategy in Section \ref{section adaptive convex}.
Before proving the convergence rate in the convex case, we will first bound $\|x_t-x^*\|^2, \forall t\in[T]$, using the following lemma derived from Lemma \ref{lemma high probability}.
\begin{lemma}
  \label{high probability sum non-convex adaptive smooth}
  Suppose that $f(x)$ is an $L$-smooth function. Given $T\geq 1 $ and $\delta\in \left(0,1\right)$, if Assumptions \ref{assumption 2} and \ref{assumption 3} hold, then with probability at least $1-\delta$, $\forall l\in[T]$,
  \begin{align}
    \label{inequality high probability non-convex adaptive smooth 6.36}
    \sumlimits{t=1}{l}-\la\overline{g}_t,\xi_t\ra\leq \frac{1}{2}\sumlimits{t=1}{l}\frac{Q_t^2}{Q^2}\|\overline{g}_t\|^2+\frac{3Q^2}{2}\log\frac{T}{\delta},
  \end{align}
  where $Q_t$ and $Q$ are defined in \eqref{6.3} and \eqref{definition 6.4}, respectively.
\end{lemma}

\begin{proposition}
  \label{proposition 6.2}
  Under the conditions of Theorem \ref{theorem 6}, with probability at least $1-3\delta$, we have
  \begin{align}
    \label{inequality 6.34}
    \|x_t-x^*\|^2\leq D_L^2, \quad  \forall t\in[T],
  \end{align}
  where 
  \begin{align}
    \label{D_L}
    D_L^2=2\|x_1-x^*\|^2+4R_L+10\eta^2\mathcal{F}+\frac{\eta^2}{2}+ 32\left(A_{T,\delta}R_L+\mathcal{W}_L^2B_{T,\delta}\right)+\frac{8}{\eta^2}R_L^2,
  \end{align}
  \begin{align}
    \label{6.42}
  R_L=2\eta^2\log\left(1+\frac{T\left(A\functiongapboundsmooth+2BL\functiongapboundsmooth+C\right)}{2G_0^2}\right)+ 7\frac{\eta^2}{G_0^2}\left(A\functiongapboundsmooth+2BL\functiongapboundsmooth+C\right)\log\frac{T}{\delta},
  \end{align}
$\mathcal{W}_L=\mathcal{W}(\functiongapboundsmooth)$ with $\mathcal{W}(x)$  defined in \eqref{mathcal W}, and $\mathcal{F}$, $A_{T,\delta}$, $B_{T,\delta}$ are given in \eqref{f}, \eqref{A_T B_T}, respectively. 
\end{proposition}
\begin{proof}
  We will imitate the proof of Proposition \ref{proposition 5.2}.
  Suppose that \eqref{inequality 5.4}, \eqref{inequality high probability non-convex adaptive smooth} and \eqref{inequality high probability non-convex adaptive smooth 6.36} always hold and then we deduce that \eqref{inequality 6.34} always holds. 
  Since \eqref{inequality 5.4}, \eqref{inequality high probability non-convex adaptive smooth} and \eqref{inequality high probability non-convex adaptive smooth 6.36} hold with probability at least $1-\delta$ separately, \eqref{inequality 6.34} holds with probability at least $1-3\delta$.
  It is apparent that $\|x_1-x^*\|^2\leq D_L^2$. Suppose that for some $l\in[T]$, 
  $$\|x_t-x^*\|^2\leq D_L^2, \quad \forall t\in[l].$$
  Summing up \eqref{iterating step adaptive} over $t\in[l]$ with $\theta_t=\eta_t$, we have that 
  \begin{align}
    \label{6.33}
    \|x_{l+1}-x^*\|^2=&\|x_1-x^*\|^2-2\sumlimits{t=1}{l}\eta_t \la g_t,x_t-x^*\ra+\sumlimits{t=1}{l}\eta_t^2\|g_t\|^2\notag\\
    \leq & \|x_1-x^*\|^2-2\sumlimits{t=1}{l}\eta_t \la g_t,x_t-x^*\ra+\eta^2\mathcal{F},
  \end{align}
where the inequality follows from Lemma \ref{sum log smooth}.
Recalling \eqref{5.39}, we have 
\begin{align}
  \label{6.34}
  &-2\sumlimits{t=1}{l}\eta_t \la g_t, x_t-x^*\ra \notag\\
  \leq & \underbrace{-2\sumlimits{t=1}{l}\eta_t\la \overline{g}_t, x_t-\overline{x}_t\ra}_{\text{(i)}}+\underbrace{2\sumlimits{t=1}{l}\hat{\eta}_t\la \overline{g}_t-g_t,x_t-x^*\ra}_{\text{(ii)}}+\underbrace{2\sumlimits{t=1}{l}\left(\eta_t-\hat{\eta}_t\right)\la \overline{g}_t-g_t,x_t-x^*\ra}_{\text{(iii)}}.
\end{align}
\paragraph{Term (i)}
By \eqref{6.56} and \eqref{Term i},
\begin{align}
  \label{Term i smooth}
  -2\sumlimits{t=1}{l}\eta_t\la  \overline{g}_t, x_t-\overline{x}_t\ra \leq & 2\sumlimits{t=1}{l}\eta_{t-1}^2\|\overline{g}_t-g_t\|^2+4\sumlimits{t=1}{l}\eta_t^2\|g_t\|^2.
\end{align}
\paragraph{Term (ii)}
Applying Lemma \ref{Lemma 5.4}, we have
\begin{align}
  \label{Term ii smooth}
 2 \sumlimits{t=1}{l}\hat{\eta}_t \la \overline{g}_t-g_t,x_t-x^*\ra\leq & 4 \overline{D}_l \sqrt{A_{T,\delta}\sumlimits{t=1}{l}\eta_{t-1}^2\|g_t-\overline{g}_t\|^2+\mathcal{W}_L^2B_{T,\delta}}\notag\\
 \leq & \frac{\overline{D}_l^2}{4}+16\left(A_{T,\delta}\sumlimits{t=1}{l}\eta_{t-1}^2\|g_t-\overline{g}_t\|^2+\mathcal{W}_L^2B_{T,\delta}\right)\notag\\
 \leq & \frac{D_L^2}{4}+\frac{\eta^2}{4}+16\left(A_{T,\delta}\sumlimits{t=1}{l}\eta_{t-1}^2\|g_t-\overline{g}_t\|^2+\mathcal{W}_L^2B_{T,\delta}\right),
\end{align}
where the second inequality follows from Young's inequality and the last inequality holds since the assumption that $\|x_t-x^*\|^2\leq D_L^2, \forall t\in[l]$.
\paragraph{Term (iii)}
Applying Cauchy-Schwarz inequality, we have
\begin{align}
  \label{Term iii smooth}
  2\sumlimits{t=1}{l}\left(\eta_t-\hat{\eta}_t\right)\la \overline{g}_t-g_t,x_t-x^*\ra \leq & 2D_L\sumlimits{t=1}{l}|\eta_t-\hat{\eta}_t| \cdot \|\overline{g}_t-g_t\|\notag\\
  \leq & 2\eta D_L \sumlimits{t=1}{l} \frac{\|\overline{g}_t-g_t\|^2}{\sqrt{G_{t-1}^2+\|g_t\|^2}\sqrt{G_{t-1}^2+\|\overline{g}_t\|^2}}\notag\\
  \leq & 2\frac{D_L}{\eta}\sumlimits{t=1}{l}\eta_{t-1}^2\|\overline{g}_t-g_t\|^2,
\end{align}
where the second inequality holds since Lemma \ref{hat eta} and the third inequality holds since the definition of $\eta_{t}$.
Next we will bound $\sumlimits{t=1}{l}\eta_{t-1}^2\|\overline{g}_t-g_t\|^2$.
\thecase
\begin{case1}
  \begin{align*}
    \sumlimits{i=1}{T}\|\overline{g}_i-g_i\|^2\leq 6\left(A\functiongapboundsmooth+2BL\functiongapboundsmooth+C\right)\log\frac{T}{\delta}.
  \end{align*}
  Then, we have
  \begin{align}
    \label{7.35}
    \sumlimits{t=1}{l}\eta_{t-1}^2\|\overline{g}_t-g_t\|^2\leq \frac{6\eta^2}{G_0^2}\left(A\functiongapboundsmooth+2BL\functiongapboundsmooth+C\right)\log\frac{T}{\delta}
  \end{align}
\end{case1}
\begin{case1}
\begin{align*}
  \sumlimits{i=1}{t}\|\overline{g}_i-g_i\|^2>6\left(A\functiongapboundsmooth+2BL\functiongapboundsmooth+C\right)\log\frac{T}{\delta}.
\end{align*}
  Let
\begin{align}
  \label{j}
  j=\min\left\{t\in [T] : \sumlimits{i=1}{t}\|\overline{g}_i-g_i\|^2>6\left(A\functiongapboundsmooth+2BL\functiongapboundsmooth+C\right)\log\frac{T}{\delta}\right\}.
\end{align}
By  \eqref{inequality high probability non-convex adaptive smooth 6.36}, together with the assumption that for all $t\in [l]$, $\overline{\Delta}_t\leq \functiongapboundsmooth$,
\begin{align}
  \label{lambda smooth}
  \sumlimits{t=1}{l}\la \overline{g}_t, \overline{g}_t-g_t\ra \leq & \frac{1}{2}\sumlimits{t=1}{l}\|\overline{g}_t\|^2+\frac{3}{2}\left(A\functiongapboundsmooth+2BL\functiongapboundsmooth+C\right)\log\frac{T}{\delta}.
\end{align}
If $l<j$,
\begin{align}
  \label{6.39}
  \sumlimits{t=1}{l}\eta_{t-1}^2\|\overline{g}_t-g_t\|^2\leq \frac{\eta^2}{G_0^2}\sumlimits{t=1}{l}\|\overline{g}_t-g_t\|^2\leq 6\frac{\eta^2}{G_0^2}\left(A\functiongapboundsmooth+2BL\functiongapboundsmooth+C\right)\log\frac{T}{\delta}.
\end{align}
If $l\geq j$, $\forall j\leq t \leq l$,
\begin{align*}
  G_t^2=&G_0^2+\sumlimits{i=1}{t}\|\overline{g}_i\|^2+\sumlimits{i=1}{t}\|\overline{g}_i-g_i\|^2-2\sumlimits{i=1}{t}\la \overline{g}_i,\overline{g}_i-g_i\ra\notag\\
\geq & G_0^2+\sumlimits{i=1}{t} \|\overline{g}_i-g_i\|^2-3\left(A\functiongapboundsmooth+2BL\functiongapboundsmooth+C\right)\log\frac{T}{\delta}\notag\\
\geq & G_0^2+\frac{1}{2}\sumlimits{i=1}{t}\|\overline{g}_i-g_i\|^2,
\end{align*}
where the first inequality follows from \eqref{lambda smooth} and the last inequality follows from the definition of $j$ in \eqref{j}.
Hence,
\begin{align}
  \label{6.41}
  \sumlimits{t=j}{l}\eta_{t-1}^2\|\overline{g}_t-g_t\|^2=&\sumlimits{t=j}{l}\eta_t^2\|\overline{g}_t-g_t\|^2+\sumlimits{t=j}{l}\left(\eta_{t-1}^2-\eta_t^2\right)\|\overline{g}_t-g_t\|^2\notag\\
  \leq & \eta^2\sumlimits{t=1}{T}\frac{\|\overline{g}_t-g_t\|^2}{G_0^2+\frac{1}{2}\sum_{i=1}^t\|\overline{g}_i-g_i\|^2}+\frac{\eta^2}{G_0^2}\left(A\functiongapboundsmooth+2BL\functiongapboundsmooth+C\right)\notag\\
  \leq & 2\eta^2\log\left(1+\frac{T\left(A\functiongapboundsmooth+2BL\functiongapboundsmooth+C\right)}{2G_0^2}\right)+\frac{\eta^2}{G_0^2}\left(A\functiongapboundsmooth+2BL\functiongapboundsmooth+C\right),
\end{align}
where the last inequality follows from Lemma \ref{smooth function value gap}, Lemma \ref{sum log} and Proposition \ref{proposition 6.1}.
\end{case1}
Combining \eqref{7.35}, \eqref{6.39} and \eqref{6.41}, we have
\begin{align}
  \label{7.49}
  &\sumlimits{t=1}{l}\eta_{t-1}^2\|\overline{g}_t-g_t\|^2\notag\\
  \leq &2\eta^2\log\left(1+\frac{T\left(A\functiongapboundsmooth+2BL\functiongapboundsmooth+C\right)}{2G_0^2}\right)+ 7\frac{\eta^2}{G_0^2}\left(A\functiongapboundsmooth+2BL\functiongapboundsmooth+C\right)\log\frac{T}{\delta}=R_L.
\end{align}
Then, combining \eqref{6.33}, \eqref{6.34}, \eqref{Term i smooth}, \eqref{Term ii smooth} and \eqref{Term iii smooth}, we have
\begin{align*}
  \|x_{l+1}-x^*\|^2\leq & \|x_1-x^*\|^2+2\sumlimits{t=1}{l}\eta_{t-1}^2\|\overline{g}_t-g_t\|^2+4\sumlimits{t=1}{l}\eta_t^2\|g_t\|^2 +\frac{D_L^2}{4}+\frac{\eta^2}{4} \notag\\
  &+16\left(A_{T,\delta}\sumlimits{t=1}{l}\eta_{t-1}^2\|g_t-\overline{g}_t\|^2+\mathcal{W}_L^2B_{T,\delta}\right)+2\frac{D_L}{\eta}\sumlimits{t=1}{l}\eta_{t-1}^2\|\overline{g}_t-g_t\|^2+\eta^2\mathcal{F} \notag\\
  \leq & \|x_1-x^*\|^2+2R_L+4\eta^2\mathcal{F}+\frac{D_L^2}{4}+\frac{\eta^2}{4}+16\left(A_{T,\delta}R_L+\mathcal{W}_L^2B_{T,\delta}\right)  \notag\\
  &+\frac{1}{4}D_L^2+\frac{4}{\eta^2}R_L^2+\eta^2 \mathcal{F}\notag\\
  = & \|x_1-x^*\|^2+2R_L+5\eta^2\mathcal{F}+\frac{1}{2}D_L^2+\frac{\eta^2}{4}+ 16\left(A_{T,\delta}R_L+\mathcal{W}_L^2B_{T,\delta}\right)+\frac{4}{\eta^2}R_L^2= D_L^2,
\end{align*}
where the second inequality holds since Young's inequality, \eqref{7.49} and Lemma \ref{sum log smooth}. The last equation holds since 
\begin{align*}
  D_L^2=2\|x_1-x^*\|^2+4R_L+10\eta^2\mathcal{F}+\frac{\eta^2}{2}+ 32\left(A_{T,\delta}R_L+\mathcal{W}_L^2B_{T,\delta}\right)+\frac{8}{\eta^2}R_L^2.
\end{align*}
The proof is complete.
\end{proof}
 Based on Proposition \ref{proposition 6.2}, we are able to prove Theorem \ref{theorem 6}.
\begin{proof}[Proof of Theorem \ref{theorem 6}]
Suppose that \eqref{inequality high probability convex constant convergence}, \eqref{inequality 5.4}, \eqref{inequality high probability non-convex adaptive smooth} and \eqref{inequality high probability non-convex adaptive smooth 6.36} always hold and then we deduce \eqref{inequality theorem 6} always holds.
Since \eqref{inequality high probability convex constant convergence}, \eqref{inequality 5.4}, \eqref{inequality high probability non-convex adaptive smooth} and \eqref{inequality high probability non-convex adaptive smooth 6.36} hold with probability at least $1-\delta$ separately, it follows that \eqref{inequality theorem 6} holds with probability at least $1-4\delta$.
Applying Lemma \ref{sum log smooth} and Lemma \ref{sum square} to \eqref{convex sum adaptive}, together with the conclusion that  $\|x_t-x^*\|^2 \leq D_L^2, \forall t\in[T]$ in Proposition \ref{proposition 6.2}, we have
\begin{align*}
  \sumlimits{t=1}{T}\la g_t,\overline{x}_t-x^*\ra \leq &  \frac{G_0\|x_1-x^*\|^2}{2\eta}+\left(2\eta+\frac{D_L^2}{\eta}\right)\sqrt{\sumlimits{t=1}{T}\|g_t\|^2} +\eta \mathcal{F} \sqrt{G_0^2+\sum_{t=1}^T\|g_t\|^2}   \notag\\
  \leq &  \frac{G_0\|x_1-x^*\|^2}{2\eta}+\left(2\eta+\frac{D_L^2}{\eta}\right)\sqrt{\sumlimits{t=1}{T}\|g_t\|^2} +\eta \mathcal{F} G_0+\eta \mathcal{F} \sqrt{\sumlimits{t=1}{T}\|g_t\|^2}\notag\\
  \leq &  \frac{G_0\|x_1-x^*\|^2}{2\eta}+\left(2\eta+\frac{D_L^2}{\eta}+\eta \mathcal{F}\right)\sqrt{2\sumlimits{t=1}{T}\left(A\overline{\Delta}_t+\left(B+1\right)\|\overline{g}_t\|^2+C\right)}+\eta \mathcal{F} G_0\notag\\
  \leq & \frac{G_0\|x_1-x^*\|^2}{2\eta}+\left(2\eta+\frac{D_L^2}{\eta}+\eta \mathcal{F}\right)\sqrt{2\sumlimits{t=1}{T}\left(A\overline{\Delta}_t+2\left(B+1\right)L\overline{\Delta}_t+C\right)}+ \eta \mathcal{F} G_0,
\end{align*} 
where the second inequality holds since Lemma \ref{sqrt sum}, the third inequality follows from \eqref{triangle square} and the last inequality follows from Lemma \ref{smooth function value gap}.
Applying Lemma \ref{sqrt sum} again, we have
\begin{align}
  \label{6.43}
  \sumlimits{t=1}{T}\la g_t,\overline{x}_t-x^*\ra \leq & \frac{G_0\|x_1-x^*\|^2}{2\eta}+\left(2\eta+\frac{D_L^2}{\eta}+\eta \mathcal{F}\right)\left(\sqrt{2\left(A+2\left(B+1\right)L\right)\sumlimits{t=1}{T}\overline{\Delta}_t}+\sqrt{2TC}\right)\notag\\
  &+\eta \mathcal{F} G_0\notag\\ \leq & \frac{G_0\|x_1-x^*\|^2}{2\eta}+\left(2\eta+\frac{D_L^2}{\eta}+\eta \mathcal{F}\right)^2\left(A+2\left(B+1\right)L\right)+\frac{1}{2}\sumlimits{t=1}{T}\overline{\Delta}_t\notag\\
  &+\left(2\eta+\frac{D_L^2}{\eta}+\eta \mathcal{F}\right)\sqrt{2TC} +\eta \mathcal{F}G_0\notag\\
  \leq & \frac{G_0\|x_1-x^*\|^2}{2\eta}+\left(2\eta+\frac{D_L^2}{\eta}+\eta \mathcal{F}\right)^2\left(A+2\left(B+1\right)L\right)+\frac{1}{2}\sumlimits{t=1}{T}\la \overline{g}_t,\overline{x}_t-x^*\ra\notag\\
  &+\left(2\eta+\frac{D_L^2}{\eta}+\eta \mathcal{F}\right)\sqrt{2TC} +\eta \mathcal{F}G_0,
\end{align}
where the second inequality holds since Young's inequality and the last inequality follows from Lemma \ref{convexity}.
Applying Lemma \ref{high probability sum convex constant convergence} and letting $$\lambda=1/\left(\left(D_L+\eta\right)\sqrt{TC/\log\frac{1}{\delta}}+3\left(A+2BL\right)\left(D_L+\eta\right)^2\right),$$ we have
\begin{align}
  \label{inequality 6.44}
  \sumlimits{t=1}{T}\la \overline{g}_t-g_t,\overline{x}_t-x^*\ra \leq &  \frac{3\lambda}{4}\sumlimits{t=1}{T}\left(\left(A+2BL\right)\overline{\Delta}_t+C\right)\|\overline{x}_t-x^*\|^2+\frac{1}{\lambda}\log\frac{1}{\delta}\notag\\
  \leq & \frac{1}{4\left(D_L+\eta\right)^2}\sumlimits{t=1}{T}\overline{\Delta}_t\|\overline{x}_t-x^*\|^2+\frac{3}{4\left(D_L+\eta\right)\sqrt{T}}\sqrt{C\log\frac{1}{\delta}}\sumlimits{t=1}{T}\|\overline{x}_t-x^*\|^2\notag\\
  &+\left(D_L+\eta\right)\sqrt{TC\log\frac{1}{\delta}}+3\left(A+2BL\right)\left(D_L+\eta\right)^2\log\frac{1}{\delta}\notag\\
  \leq & \frac{1}{4}\sumlimits{t=1}{T}\overline{\Delta}_t+\frac{7}{4}\left(D_L+\eta\right)\sqrt{TC\log\frac{1}{\delta}}+3\left(A+2BL\right)\left(D_L+\eta\right)^2\log\frac{1}{\delta},
 \end{align}
 where the first inequality is due to Lemma \ref{smooth function value gap} and the third inequality holds since $\|\overline{x}_t-x^*\|\leq \|x_t-x^*\|+\|\overline{x}_t-x_t\|\leq D_L+\eta$.
 Applying Lemma \ref{convexity} to \eqref{inequality 6.44},
 \begin{align}
  \label{6.45}
  \sumlimits{t=1}{T}\la \overline{g}_t-g_t,\overline{x}_t-x^*\ra \leq & \frac{1}{4}\sumlimits{t=1}{T}\la \overline{g}_t,\overline{x}_t-x^*\ra +\frac{7}{4}\left(D_L+\eta\right)\sqrt{TC\log\frac{1}{\delta}}\notag\\
  &+3\left(A+2BL\right)\left(D_L+\eta\right)^2\log\frac{1}{\delta}.
 \end{align}
Combining \eqref{6.43} and \eqref{6.45}, we have
\begin{align*}
  \sumlimits{t=1}{T}\la \overline{g}_t,\overline{x}_t-x^*\ra\leq &\frac{1}{4}\sumlimits{t=1}{T}\la \overline{g}_t,\overline{x}_t-x^*\ra +\frac{7}{4}\left(D_L+\eta\right)\sqrt{TC\log\frac{1}{\delta}}\notag\\
  &+3\left(A+2BL\right)\left(D_L+\eta\right)^2\log\frac{1}{\delta}\notag\\
  &+\frac{G_0\|x_1-x^*\|^2}{2\eta}+\left(2\eta+\frac{D_L^2}{\eta}+\eta \mathcal{F}\right)^2\left(A+2\left(B+1\right)L\right)+\frac{1}{2}\sumlimits{t=1}{T}\la \overline{g}_t,\overline{x}_t-x^*\ra\notag\\
  &+\left(2\eta+\frac{D_L^2}{\eta}+\eta \mathcal{F}\right)\sqrt{2TC} +\eta \mathcal{F}G_0.
\end{align*}
Re-arranging the above inequality, we have
\begin{align*}
  \sumlimits{t=1}{T}\la \overline{g}_t,\overline{x}_t-x^*\ra\leq & 7\left(D_L+\eta\right)\sqrt{TC\log\frac{1}{\delta}}+12\left(A+2BL\right)\left(D_L+\eta\right)^2\log\frac{1}{\delta}\notag\\
  &+ \frac{2G_0\|x_1-x^*\|^2}{\eta}+4\left(2\eta+\frac{D_L^2}{\eta}+\eta \mathcal{F}\right)^2\left(A+2\left(B+1\right)L\right)\notag\\
  &+4\left(2\eta+\frac{D_L^2}{\eta}+\eta \mathcal{F}\right)\sqrt{2TC} +4\eta \mathcal{F}G_0.
\end{align*} 
Using Jensen's inequality and Lemma \ref{convexity}, we have
\begin{align}
  \label{7.62}
  f\left(\frac{1}{T}\sumlimits{t=1}{T}\overline{x}_t\right)-f^*\leq & \left(12\left(A+2BL\right)\left(D_L+\eta\right)^2\log\frac{1}{\delta}+\frac{2G_0\|x_1-x^*\|^2}{\eta}\right)\frac{1}{T}\notag\\
  &+\left(4\left(2\eta+\frac{D_L^2}{\eta}+\eta \mathcal{F}\right)^2\left(A+2\left(B+1\right)L\right)+4\eta \mathcal{F} G_0\right)\frac{1}{T}\notag\\
  & + \frac{7\left(D_L+\eta\right)\sqrt{C\log\frac{1}{\delta}}+4\left(2\eta+\frac{D_L^2}{\eta}+\eta \mathcal{F}\right)\sqrt{2C}}{\sqrt{T}}.
\end{align}
\end{proof}
\section{Numerical Experiment}
\label{experiment}
In this section, we run the RSAG method with different step size settings to solve the non-convex and non-smooth phase retrieval problem, i.e.,
\begin{align}
  \label{phase retrieval}
  \min\limits_{x\in\mathbb{R}^d}f(x)=\frac{1}{2m}\sumlimits{i=1}{m}\left(y_i-\left|a_i^{\top} x\right|^2\right)^2,
\end{align}
where $y_i=\left|a_i^{\top} z \right|^2$ is the ground-truth label with $z,a_i\in\mathbb{R}^d$ for all $i\in[m]$.
Phase retrieval is a classic signal processing problem in X-ray crystallography and coherent diffraction imaging applications \cite{drenth1994principles,miao1999extending}.
We prove that \eqref{phase retrieval} is $(L_0,L_1)$-smooth in Appendix \ref{phase retrieval proposition} and we refer interested readers to see more discussion about this model in \cite{chen2023generalized}.

In our experiment, we set the problem dimension $d=10$ and sample size $m=1000$.
The artificial dataset is generated by $y_i=\left|a_i^{\top}z\right|^2+\epsilon_i, i\in[m]$, where each coordinate of both the measurement vector $a_i\in\mathbb{R}^d$ and the true parameter $z$ satisfy Gaussian distribution $\mathcal{N}\left(0,0.5\right)$, and $\epsilon_i\sim \mathcal{N}\left(0,25\right)$ is the noise.

We first analyze Algorithm \ref{algorithm1} with constant step sizes and show the training loss curve in Figure \ref{fig a}. 
Corresponding to the previous theoretical analysis in Section \ref{section 4}, we set $\gamma_t=\theta_t$ as SGD and $\gamma_t=\left(1+\alpha_t\right)\theta_t$ as RSAG (called RSAG (I), motivated by the  parameters setup in \cite[Corollary 3 a)]{ghadimi2016accelerated}) where $\theta_t=\eta$ is a constant.
Additionally, we run RSAG with 
	$\gamma_t = \alpha_t \theta_t /2,$ $\theta_t=2\eta /\alpha_t$ 
	  (called RSAG (II), inspired by the parameters setting in  \cite[Corollary 3 b)]{ghadimi2016accelerated} where they obtained the accelerated convergence rate in the smooth convex case with bounded variance noise assumption).
	  Although we did not provide any convergence analysis for RSAG (II), we are still curious whether RSAG (II) performs better than RSAG (I) in our experiments. We leave the convergence analysis of RSAG (II) for future work.

In the field of the adaptive methods, we recall $\eta_t$ defined in \eqref{adaptive step size} with $G_0=\sqrt{0.1}$.
With the constraint of $\theta_t=\eta_t$, we show the convergence behavior of AdaGrad-Norm where $\gamma_t=\theta_t$ and RSAG (I) where $\gamma_t=\left(1+\alpha_t\right)\theta_t$ in Figure \ref{fig b}.
Similar to the constant setting, we additionally run Algorithm \ref{algorithm1} with
$\gamma_t=\alpha_t \theta_t/2$ and
$\theta_t=2\eta_t/\alpha_t$
as RSAG (II) (motivated by \cite{levy2018online,kavis2022high}).

From both figures, we observe that the performance of RSAG (I) and SGD (or AdaGrad-Norm when the step size is adaptive) is quite similar.
However, RSAG (I) performs slightly better when the step size is small although RSAG (I), SGD and AdaGrad-Norm satisfy the same convergence order in previous theoretical analysis.
Furthermore, we find that RSAG (II) could improve the convergence speed in this model.
\begin{figure}[!h]
\begin{subfigure}{0.49\linewidth}
  \centering
  \includegraphics[width=0.8\linewidth]{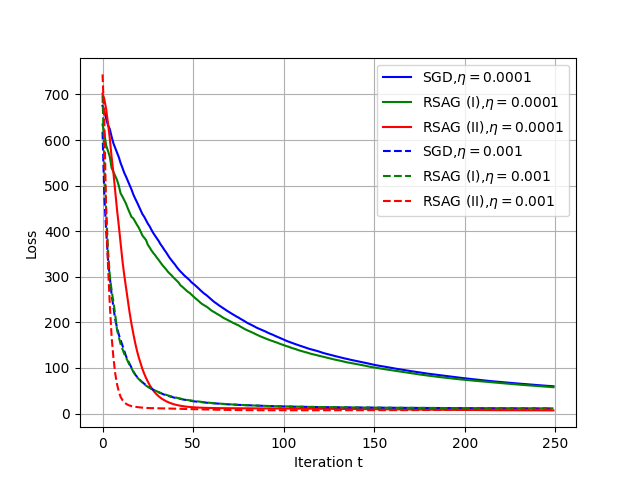}
  \caption{Constant step sizes}
  \label{fig a}
\end{subfigure}
\begin{subfigure}{0.49\linewidth}
  \centering
  \includegraphics[width=0.8\linewidth]{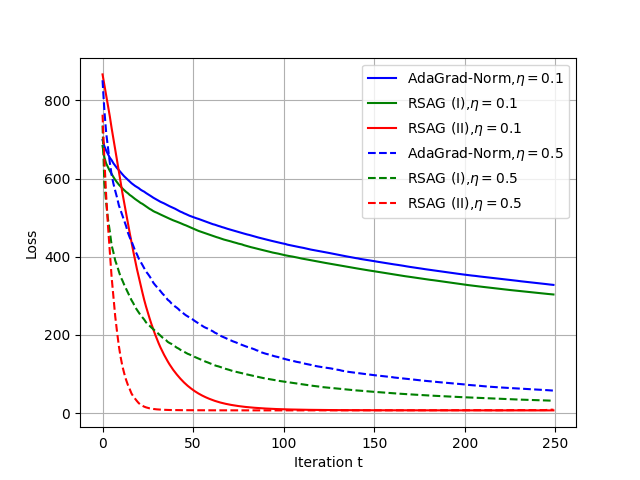}
  \caption{Adaptive step sizes}
  \label{fig b}
\end{subfigure}
\caption{Experiment results. At each iteration of the algorithm, we set the batch size as 10 to compute the stochastic gradient. 
We run each algorithm 100 times and plot the average loss at each iteration.}
\end{figure}

\section{Conclusion}
We analyze the high-probability convergence rates of some first-order algorithms, e.g. SGD, AdaGrad-Norm and RSAG,  under the relaxed affine variance noise assumption in both non-convex and convex cases. 
These algorithms converge at the rate of $\tilde{\mathcal{O}}\left(1/\sqrt{T}\right)$ for generalized smooth objective functions and could accelerate to $\tilde{\mathcal{O}}\left(1/T\right)$ if the specific noise parameters are sufficiently small.
Furthermore, we apply our method to smooth optimization and obtain the same convergence order for RSAG with adaptive step sizes and AdaGrad-Norm, without requiring the prior knowledge of problem parameters.
Finally, we run the algorithm on the phase retrieval problem and find that RSAG could perform slightly better than SGD (or AdaGrad-Norm) when the step size is small.

\section*{Acknowledgement}
We sincerely thank the editor and the referees for their valuable and constructive suggestions, which have greatly improved this work.
This work was supported in part by the NSFC under grant number 12471096, and the National Key Research and Development Program of China under grant number 2021YFA1003500. 
The corresponding author is Junhong Lin.

\bibliographystyle{abbrv}
\bibliography{ref_final}

\appendix
\label{Appendix}

\section{Complementary Lemmas}
\begin{lemma}[Remark 2.3 in \cite{zhang2020improved}]
  \label{descent lemma}
  Suppose that $f:\mathbb{R}^d \rightarrow \mathbb{R}$ is an $(L_0,L_1)$-generalized smooth function. When $\|x-y\|\leq 1/L_1$, we have 
  $$\left\lvert f(y)- f(x)-\la \nabla f(x),y-x\ra\right\rvert \leq \frac{L_0+L_1\nablaf{x}}{2}\|x-y\|^2.$$
\end{lemma}

 The following lemma plays an important role in high-probability analysis. We refer interested readers to see \cite{li2020high} for the proof.
\begin{lemma}[Lemma 1 in \cite{li2020high}]
  \label{lemma high probability}
  Assume that $\{Z_t\}_{t\in[T]}$ is a martingale difference sequence with respect to $\gamma_1, \gamma_2, \cdots, \gamma_T$ and $\mathbb{E}_t[\exp\left(Z_t^2/\sigma_t^2\right)]\leq \exp(1)$ for all $1 \leq t \leq T$, where $\sigma_t$ is a sequence of measurable random variables with respect to $\gamma_1, \gamma_2, \cdots, \gamma_{t-1}$. Then, for any fixed $\lambda >0$ and $\delta\in(0,1)$, with probability at least $1-\delta$, we have
  $$\sumlimits{t=1}{T}Z_t\leq \frac{3\lambda}{4}\sumlimits{t=1}{T}\sigma_t^2+\frac{1}{\lambda}\log\frac{1}{\delta}.$$
\end{lemma}

\begin{lemma}[Lemma 14 in \cite{attia2023sgd}]
  \label{attia}
  Let $X_t$ be a random variable with respect to $\gamma_1,\cdots, \gamma_t$ such that $|X_t| \leq 1$ with probability 1. Then, for every $\delta \in (0, 1)$ and any random variable $\hat{X}_t$ with respect to $\gamma_1,\cdots \gamma_{t-1}$ such that $|\hat{X}_t| \leq 1$ with probability 1,
  \begin{align*}
  \mathrm{Pr} \left(\exists t < \infty: \left| \sumlimits{s=1}{t} \left( X_s-\mathbb{E}_s[X_s]\right)  \right| \geq \sqrt{A_t(\delta) \sumlimits{s=1}{t}\left(X_s-\hat{X}_s\right)^2+B_t(\delta)   } \right) \leq \delta,
  \end{align*}
  where 
  \begin{align}
    \label{A B}
    A_t(\delta)=16 \log\left( \frac{60\log(6t)}{\delta}\right), \qquad B_t(\delta)=16 \log^2\left( \frac{60\log(6t)}{\delta}\right).
  \end{align}
  \end{lemma}
  
  \begin{lemma}
  \label{lemma 6.2}  
  Let $f(x) : \mathbb{R}^d\rightarrow \mathbb{R}$ be an $(L_0,L_1)$-smooth function with minimum $f^*$. 
  Then, $$\|\nabla f(x_t)\|^2\leq 2\left(L_0+L_1\|\nabla f(x_t)\|\right)\Delta_t,$$ and $$\nablaf{x_t}^2\leq 4L_0\Delta_t+4L_1^2\Delta_t^2,$$
  where $\Delta_t=f(x_t)-f^*$.
  \end{lemma}
  \begin{proof}
    Let $x=x_t-\frac{1}{L_0+L_1\|\nabla f(x_t)\|}\nabla f(x_t)$.
    It is easy to verify that $\|x-x_t\|\leq 1/L_1$.
    From Lemma \ref{descent lemma}, we have that
    $$f(x) \leq f(x_t)  + \la \nabla f(x_t), x - x_t \ra + {L_0+L_1 \|\nabla f(x)\|  \over 2} \|x - x_t\|^2.$$
    Hence, 
    \begin{align*}
      &\frac{\|\nabla f(x_t)\|^2}{2\left(L_0+L_1\|\nabla f(x_t)\|\right)}\leq f(x_t)-f(x)\leq \Delta_t.
    \end{align*}
  Re-arranging the above inequality and applying Cauchy-Schwarz inequality and Young's inequality, we have
  \begin{align*}
    \nablaf{x_{t}}^2 \leq & 2L_0\Delta_{t}+\frac{1}{2}\|\nabla f(x_{t})\|^2+2L_1^2\Delta_{t}^2.
  \end{align*}
  We obtain the desired result.
  \end{proof}
  As a direct corollary, we obtain the following lemma for smooth functions.
  \begin{lemma}
    \label{smooth function value gap}
    Let $f(x) : \mathbb{R}^d\rightarrow \mathbb{R}$ be an $L$-smooth function with minimum $f^*$.
    Then, 
    \begin{align}
      \label{lemma 9.5 inequality}
      \|\nabla f(x_t)\|^2\leq 2L\Delta_t,
    \end{align}
    where $\Delta_t=f(x_t)-f^*$.
  \end{lemma}
  
  \begin{lemma}
    \label{sqrt sum}
    Let $\{a_t\}_{t\in[n]}$ be a sequence of non-negative real numbers. We have
    \begin{align*}
      \sqrt{\sum_{i=1}^n a_i}\leq \sumlimits{i=1}{n}\sqrt{a_i}.
    \end{align*}
  \end{lemma}

  \begin{lemma}
    \label{convexity}
    Let $f(x) : \mathbb{R}^d\rightarrow \mathbb{R}$ be a convex function with minimum $f^*$ and let $x^*$ be the minimizer, i.e., $f(x^*)=f^*$. 
    Then, we have 
    \begin{align*}
      \la \nabla f(x), x-x^*\ra \geq f(x)-f^*\geq 0, \quad \forall x\in \mathbb{R}^d.
    \end{align*}
  \end{lemma}
  \begin{proof}
    By the convexity of $f$, we have
    \begin{align*}
       f^*\geq f(x)+ \la \nabla f(x), x^*-x\ra.
    \end{align*}
    Thus, 
    \begin{align*}
      \la \nabla f(x), x-x^*\ra \geq  f(x)-f^*\geq 0.
    \end{align*}
  \end{proof}
   Lemma \ref{sum square} and Lemma \ref{sum log} have broad applications in the analysis of adaptive algorithms. See \citep{duchi2011adaptive,streeter2010less,levy2018online} e.g..
  \begin{lemma}[Lemma A.2 in \cite{levy2018online}]
    \label{sum square}
    Let $\{a_t\}_{t\in[n]}$ be a sequence of non-negative real numbers. Then, it holds that
    $$\sqrt{\sumlimits{i=1}{n}a_i}\leq \sumlimits{i=1}{n}\frac{a_i}{\sqrt{\sum_{j=1}^{i}a_j}}\leq 2\sqrt{\sumlimits{i=1}{n}a_i}.$$
  \end{lemma}
  \begin{lemma}[Lemma A.3 in \cite{levy2018online}]
    \label{sum log}
    Let $\{a_t\}_{t\in[n]}$ be a sequence of non-negative real numbers. Then, it holds that
   \begin{align*}
    \sumlimits{i=1}{n}\frac{a_i}{1+\sum_{j=1}^{i}a_j}\leq \log\left(1+\sumlimits{i=1}{n}a_i\right).
   \end{align*}
  \end{lemma}

  The following two lemmas follow from \cite{attia2023sgd}, where the affine variance noise is assumed. Here, we give the corresponding two lemmas under the relaxed affine variance noise assumption.
  \begin{lemma}
    \label{tilde eta}
    Under Assumption \ref{assumption 3}, for all $1\leq t\leq T$, $$|\tilde{\eta}_t-\eta_t|\leq \frac{2\tilde{\eta}_t\sqrt{A\overline{\Delta}_t+B\|\overline{g}_t\|^2+C}}{\sqrt{G_0^2+\sumlimits{i=1}{t}\|g_i\|^2}},$$
    where $\eta_t$ and $\tilde{\eta}_t$ are defined in \eqref{adaptive step size} and \eqref{tilde eta_t}. 
  \end{lemma}
  \begin{proof}
    Using $\frac{1}{\sqrt{a}}-\frac{1}{\sqrt{b}}=\frac{b-a}{\sqrt{ab}\left(\sqrt{a}+\sqrt{b}\right)}$,
  \begin{align*}
    &|\tilde{\eta}_t-\eta_t| \notag\\
    = & \frac{\eta\cdot\left|\|g_t\|^2-A\overline{\Delta}_t-\left(B+1\right)\|\overline{g}_t\|^2-C\right|}{\etalow\tiletalow  \left(\etalow+\tiletalow\right)}\notag\\
    = & \frac{\tilde{\eta}_t}{\etalow} \cdot \frac{\left|\|g_t\|^2-A\overline{\Delta}_t-\left(B+1\right)\|\overline{g}_t\|^2-C\right|}{ \etalow+\tiletalow}\notag\\
    \leq & \frac{\tilde{\eta}_t}{\etalow} \cdot \frac{A\overline{\Delta}_t+B\|\overline{g}_t\|^2+C+\|g_t-\overline{g}_t\|\|g_t+\overline{g}_t\|}{\etalow+\tiletalow}\notag\\
    \leq & \frac{\tilde{\eta}_t}{\etalow} \cdot \left(\sqrt{A\overline{\Delta}_t+B\|\overline{g}_t\|^2+C}+\|g_t-\overline{g}_t\|\right)\notag\\
    \leq & \frac{2\tilde{\eta}_t\sqrt{A\overline{\Delta}_t+B\|\overline{g}_t\|^2+C}}{\sqrt{G_0^2+\sumlimits{i=1}{t}\|g_i\|^2}},
  \end{align*}
  where we use 
  \begin{align}
    \label{8.30}
    \left|\|g_t\|^2-\|\overline{g}_t\|^2\right|=\left|\la g_t- \overline{g}_t,g_t+\overline{g}_t\ra\right|\leq \|g_t- \overline{g}_t\|\|g_t+\overline{g}_t\|
  \end{align}
  in the first inequality and 
  \begin{align*}
    \|g_t+\overline{g}_t\|\leq \|g_t\|+\|\overline{g}_t\|\leq G_t+\tiletalow
  \end{align*} 
  in the second inequality.
  \end{proof}
  
  \begin{lemma}
    \label{hat eta}
    For all $t\in[T]$, 
    \begin{align*}
      |\eta_t-\hat{\eta}_t|\leq & \eta\frac{\|\overline{g}_t-g_t\|}{\sqrt{G_{t-1}^2+\|g_t\|^2}\sqrt{G_{t-1}^2+\|\overline{g}_t\|^2}},
    \end{align*}
    where $\hat{\eta}_t$ is defined in \eqref{definition of hat eta}.
  \end{lemma}
  \begin{proof}
    \begin{align*}
      |\eta_t-\hat{\eta}_t|= & \eta \left|\frac{1}{\sqrt{G_{t-1}^2+\|g_t\|^2}}-\frac{1}{\sqrt{G_{t-1}^2+\|\overline{g}_t\|^2}}\right|\notag\\
    = & \eta  \frac{\left|\sqrt{G_{t-1}^2+\|g_t\|^2}-\sqrt{G_{t-1}^2+\|\overline{g}_t\|^2}\right| }{\sqrt{G_{t-1}^2+\|g_t\|^2} \cdot \sqrt{G_{t-1}^2+\|\overline{g}_t\|^2}}   \notag\\
    = & \eta  \frac{\left|\|\overline{g}_t\|^2-\|g_t\|^2\right|}{\sqrt{G_{t-1}^2+\|g_t\|^2} \cdot \sqrt{G_{t-1}^2+\|\overline{g}_t\|^2}\left(\sqrt{G_{t-1}^2+\|\overline{g}_t\|^2}+\sqrt{G_{t-1}^2+\|g_t\|^2}\right)} \notag\\
    \leq & \eta \frac{\|\overline{g}_t-g_t\|\cdot \|\overline{g}_t+g_t\|}{\sqrt{G_{t-1}^2+\|g_t\|^2} \cdot \sqrt{G_{t-1}^2+\|\overline{g}_t\|^2}\left(\sqrt{G_{t-1}^2+\|\overline{g}_t\|^2}+\sqrt{G_{t-1}^2+\|g_t\|^2}\right)}\notag\\
    \leq & \eta \frac{\|\overline{g}_t-g_t\|\left(\|\overline{g}_t\|+\|g_t\|\right)}{\sqrt{G_{t-1}^2+\|g_t\|^2} \cdot \sqrt{G_{t-1}^2+\|\overline{g}_t\|^2}\left(\sqrt{G_{t-1}^2+\|\overline{g}_t\|^2}+\sqrt{G_{t-1}^2+\|g_t\|^2}\right)}\notag\\
    \leq & \eta \frac{\|\overline{g}_t-g_t\|}{\sqrt{G_{t-1}^2+\|g_t\|^2} \cdot \sqrt{G_{t-1}^2+\|\overline{g}_t\|^2}},
  \end{align*}
  where the first inequality follows from \eqref{8.30} and the second inequality is due to the triangle inequality. 
  \end{proof}
  \begin{proposition}
    \label{expectation}
    Under the conditions of Theorem \ref{theorem 7}, we have $$\mE[f(x_t)-f^*]\leq \functiongapexpectation, \quad \forall t\in[T],$$
    where 
    \begin{align}
      \label{functiongapexpectation}
      \functiongapexpectation=4\overline{\Delta}_1+8.
    \end{align}
  \end{proposition}
  \begin{proof}
    We obtain the result by induction. It is apparent that $\mE\left[f(x_1)-f^*\right]=f(x_1)-f^*\leq \functiongapexpectation.$ 
    Suppose that for some $l\in [T]$, 
    $$\mE[f(x_t)-f^*]\leq \functiongapexpectation, \quad \forall t\in[l].$$
    Using Lemma \ref{descent lemma}, we have
    \begin{align}
      \label{a}
      & f(x_{t+1})-f(x_t) \notag\\ 
      \leq & \la \nabla f(x_t), x_{t+1}-x_t \ra+\frac{L}{2}\|x_{t+1}-x_t\|^2\notag\\
      = & -\theta_t\la \nabla f(x_t), g_t \ra+\frac{L}{2}\theta_t^2\|g_t\|^2\notag\\
      = & -\theta_t\la \overline{g}_t+\nabla f(x_t)-\overline{g}_t,\overline{g}_t+\xi_t\ra + \frac{L}{2}\theta_t^2\|\overline{g}_t+\xi_t\|^2\notag\\
      = & -\theta_t \|\overline{g}_t\|^2-\theta_t\la \nabla f(x_t),\xi_t\ra -\theta_t \la \nabla f(x_t)-\overline{g}_t,\overline{g}_t\ra+\frac{L}{2}\theta_t^2\left(\|\overline{g}_t\|^2+2\la \overline{g}_t, \xi_t \ra+\|\xi_t\|^2\right).
    \end{align}
    Applying Cauchy-Schwarz inequality to \eqref{a} with $\theta_t=\eta$, we have
    \begin{align*}
      & f(x_{t+1})-f(x_t) \notag\\ 
      \leq & -\eta \|\overline{g}_t\|^2-\eta\la \nabla f(x_t),\xi_t\ra + \eta\|\nabla f(x_t)-\overline{g}_t\|\|\overline{g}_t\|+\frac{L}{2}\eta^2\left(\|\overline{g}_t\|^2+2\la \overline{g}_t, \xi_t \ra+\|\xi_t\|^2\right)\notag\\
      \leq & -\eta\left(1-\frac{L}{2}\eta\right) \|\overline{g}_t\|^2-\eta\la \nabla f(x_t),\xi_t\ra + L\eta\|x_t-\overline{x}_t\|\|\overline{g}_t\|+L\eta^2\la \overline{g}_t, \xi_t \ra+\frac{L}{2}\eta^2\|\xi_t\|^2\notag\\
      \leq & -\eta\left(1-L\eta\right) \|\overline{g}_t\|^2-\eta\la \nabla f(x_t),\xi_t\ra + \frac{L}{2}\|x_t-\overline{x}_t\|^2+L\eta^2\la \overline{g}_t, \xi_t \ra+\frac{L}{2}\eta^2\|\xi_t\|^2,
    \end{align*}
    where the second inequality follows from the definition of $L$-smoothness in \eqref{definition L smoothness} and the last inequality holds since Young's inequality.
    Combining with Proposition \ref{proposition overline x_t - x_t}, we have that
    \begin{align*}
      & f(x_{t+1})-f(x_t) \notag\\ 
      \leq & -\eta\left(1-L\eta\right) \|\overline{g}_t\|^2-\eta\la \nabla f(x_t),\xi_t\ra + \frac{L}{2}    \left(1-\alpha_t\right)\varGamma_t\sumlimits{k=1}{t}\frac{\alpha_k}{\varGamma_k}\eta^2\|g_k\|^2       +L\eta^2\la \overline{g}_t, \xi_t \ra+\frac{L}{2}\eta^2\|\xi_t\|^2.
    \end{align*}
    Summing over $t\in[l]$, we have that
    \begin{align}
      \label{sum expectation}
      &f(x_{l+1})-f(x_1)\notag\\
      \leq &  -\eta\left(1-L\eta\right)\sumlimits{t=1}{l}\|\overline{g}_t\|^2-\eta\sumlimits{t=1}{l}\la \nabla f(x_t),\xi_t\ra\notag\\
      & +\frac{L}{2}  \sumlimits{t=1}{l}  \left(1-\alpha_t\right)\varGamma_t\sumlimits{k=1}{t}\frac{\alpha_k}{\varGamma_k}\eta^2\|g_k\|^2+L\eta^2 \sumlimits{t=1}{l}\la \overline{g}_t, \xi_t \ra+\frac{L}{2}\eta^2\sumlimits{t=1}{l}\|\xi_t\|^2\notag\\
      \leq &    -\eta\left(1-L\eta\right)\sumlimits{t=1}{l}\|\overline{g}_t\|^2-\eta\sumlimits{t=1}{l}\la \nabla f(x_t),\xi_t\ra\notag\\
      & +\frac{L\eta^2}{2}\sumlimits{t=1}{l}\left[\sumlimits{k=t}{l}\left(1-\alpha_k\right)\varGamma_k\right]\frac{\alpha_t}{\varGamma_t}\|g_t\|^2+L\eta^2 \sumlimits{t=1}{l}\la \overline{g}_t, \xi_t \ra+\frac{L}{2}\eta^2\sumlimits{t=1}{l}\|\xi_t\|^2.
    \end{align}
   Note that $$\mE \left[\la \overline{g}_t, \xi_t\ra\right]=\mE\left[\mE_t\la \overline{g}_t,\xi_t\ra\right]= \mE \left[\la \overline{g}_t,\mE_t\left[\xi_t\right] \ra\right]=0,$$ and $$\mE\left[\la \nabla f(x_t),\xi_t\ra\right]=\mE\left[\mE_t\la \nabla f(x_t),\xi_t\ra\right]=\mE\left[\nabla f(x_t), \mE_t\left[\xi_t\right]\ra\right]=0.$$
   Therefore, taking expectation on both sides of \eqref{sum expectation}, we have
   \begin{align}
    \label{b}
     &\mathbb{E}\left[f(x_{l+1})-f(x_1)\right]\notag\\
     \leq & -\eta\left(1-L\eta\right)\mathbb{E}\left[\sumlimits{t=1}{l}\|\overline{g}_t\|^2\right] + \frac{L\eta^2}{2}\sumlimits{t=1}{l}\left[\sumlimits{k=t}{l}\left(1-\alpha_k\right)\varGamma_k\right]\frac{\alpha_t}{\varGamma_t}\mathbb{E}\left[\|g_t\|^2\right]+\frac{L}{2}\eta^2\mathbb{E}\left[\sumlimits{t=1}{l}\|\xi_t\|^2\right]\notag\\
     \leq & -\eta\left(1-L\eta\right)\mathbb{E}\left[\sumlimits{t=1}{l}\|\overline{g}_t\|^2\right] + L\eta^2 \mathbb{E}\left[\sumlimits{t=1}{l}\|g_t\|^2\right]+\frac{L}{2}\eta^2\mathbb{E}\left[\sumlimits{t=1}{l}\|\xi_t\|^2\right],
   \end{align}
   where the second inequality holds since \eqref{proposition 1.2}.
   Also, we have that 
   \begin{align}
    \label{triangle expectation}
    \mathbb{E}_t\left[\|g_t\|^2\right]=\mathbb{E}_t\left[\|\overline{g}_t\|^2\right]+2\mathbb{E}_t\left[\la \overline{g}_t,\xi_t\ra\right]+\mathbb{E}_t\left[\|\xi_t\|^2\right]=\|\overline{g}_t\|^2+\mathbb{E}_t\|\xi_t\|^2\leq A\overline{\Delta}_t+\left(B+1\right)\|\overline{g}_t\|^2+C,
   \end{align}
  where the inequality follows from the relaxed affine variance noise assumption.
  Combining \eqref{b} and \eqref{triangle expectation}, we have
  \begin{align*}
    &\mathbb{E}\left[f(x_{l+1})-f(x_1)\right]\notag\\
    \leq & -\eta\left(1-L\eta\right)\mathbb{E}\left[\sumlimits{t=1}{l}\|\overline{g}_t\|^2\right]\notag\\
    &+L\eta^2\mathbb{E} \left[\sumlimits{t=1}{l}\left(A\overline{\Delta}_t+\left(B+1\right)\|\overline{g}_t\|^2+C\right)\right]+\frac{L}{2}\eta^2\mathbb{E}\left[\sumlimits{t=1}{l}\left(A\overline{\Delta}_t+B\|\overline{g}_t\|^2+C\right)\right]\notag\\
    \leq & -\eta\left(1-L\eta-\left(B+1\right)L\eta-\frac{1}{2}BL\eta\right)\mathbb{E}\left[\sumlimits{t=1}{l}\|\overline{g}_t\|^2\right]+\frac{3}{2}L\eta^2A\mathbb{E}\left[\sumlimits{t=1}{l}\overline{\Delta}_t\right]+\frac{3}{2}L\eta^2Cl\notag\\
    = & -\eta\left(1-L\eta\left(2+\frac{3}{2}B\right)\right)\mathbb{E}\left[\sumlimits{t=1}{l}\|\overline{g}_t\|^2\right]+\frac{3}{2}L\eta^2A\mathbb{E}\left[\sumlimits{t=1}{l}\overline{\Delta}_t\right]+\frac{3}{2}L\eta^2Cl.
  \end{align*}
  With the assumption that $\mathbb{E}[\overline{\Delta}_t]\leq \functiongapexpectation , \forall t\in[l]$, and the constraint that $\eta\leq \frac{1}{4L\left(B+1\right)}$, we have
  \begin{align}
    \label{8.6}
    \mathbb{E}\left[f(x_{l+1})-f(x_1)\right]\leq & \frac{3}{2}L\eta^2 Al \functiongapexpectation+\frac{3}{2}L\eta^2Cl.
  \end{align}
  Taking expectation on \eqref{lemma 9.5 inequality}, together with the assumption that $\overline{\Delta}_t\leq \functiongapexpectation, \forall t\in[l]$, we have 
  \begin{align*}
    \mE\left[\|\overline{g}_t\|^2\right]\leq 2L\mE\left[\overline{\Delta}_t\right]\leq 2L \functiongapexpectation, \quad \forall t\in[l].
  \end{align*}
  By Proposition \ref{proposition overline x_t - x_t} and \eqref{triangle expectation}, we have that for all $t\in[l+1]$,
  \begin{align}
    \label{8.10}
    \mE\left[\|\overline{x}_t-x_t\|^2\right] \leq & \eta^2 \varGamma_t\sumlimits{k=1}{t-1}\frac{\alpha_k}{\varGamma_k}\mE\left[\|g_k\|^2\right]\notag\\
    \leq & \eta^2 \varGamma_t\sumlimits{k=1}{t-1}\frac{\alpha_k}{\varGamma_k} \left(A\functiongapexpectation+2L\left(B+1\right)\functiongapexpectation+C\right)\notag\\
    \leq & \eta^2 \left(A\functiongapexpectation+2L\left(B+1\right)\functiongapexpectation+C\right),
  \end{align}
 where the last inequality holds since \eqref{proposition 1.1}.
  For simplicity, let 
  \begin{align}
    \label{Y3}
    Y_e=\eta\sqrt{\left(A\functiongapexpectation+2L\left(B+1\right)\functiongapexpectation+C\right)}.
  \end{align}
  By Jensen's inequality, 
  \begin{align*}
    \mE\left[\|\overline{x}_t-x_t\|\right]\leq \sqrt{\mE\left[\|\overline{x}_t-x_t\|^2\right]}\leq Y_e. 
  \end{align*}
   Applying Young's inequality to \eqref{l smooth descent lemma}, we have
  \begin{align*}
    f(\overline{x}_{l+1})\leq & f(x_{l+1})+\frac{1}{8L}\|\overline{g}_{l+1}\|^2+\frac{5L}{2}\|\overline{x}_{l+1}-x_{l+1}\|^2.
  \end{align*}
  Taking expectation on both sides, we have
  \begin{align*}
    \mE\left[f(\overline{x}_{l+1})\right]\leq &  \mE[f(x_{l+1})]+\frac{1}{8L}\mE\left[\|\overline{g}_{l+1}\|^2\right]+\frac{5L}{2}\mE\left[\|\overline{x}_{l+1}-x_{l+1}\|^2\right]\notag\\
    \leq & \mE\left[f(x_{l+1})\right]+ \frac{1}{4} \mE[\overline{\Delta}_{l+1}]+\frac{5L}{2}Y_e^2,
  \end{align*}
  where the last inequality holds since Lemma \ref{smooth function value gap} and \eqref{8.10}.
  Combining with \eqref{8.6}, we have
  \begin{align*}
    \mE[f(\overline{x}_{l+1})]\leq & f(x_1)+\frac{3}{2}L\eta^2 Al \functiongapexpectation+\frac{3}{2}L\eta^2Cl+ \frac{1}{4}\mE[\overline{\Delta}_{l+1}]+\frac{5L}{2}Y_e^2.
  \end{align*}
  Subtracting $f^*$ from both sides and applying the constraints of $\eta$ in \eqref{expectation eta constraint}, we have
  \begin{align*}
    \mE[\overline{\Delta}_{l+1}]\leq \overline{\Delta}_1+\frac{1}{2}\functiongapexpectation+1+\frac{1}{4}\mE[\overline{\Delta}_{l+1}]+1.
  \end{align*}
  Therefore, 
  \begin{align*}
    \frac{3}{4}\mE[\overline{\Delta}_{l+1}]\leq \overline{\Delta}_1+\frac{1}{2}\functiongapexpectation+2=\frac{3}{4}\functiongapexpectation.
  \end{align*}
  Now we finish the induction and obtain the desired result.
  \end{proof}

\section{Omitted Proofs}
\begin{proof}[Proof of Remark \ref{remark 4.1}]
  Let $\zeta_t=\frac{\|g_t-\overline{g}_t\|^2}{A\overline{\Delta}_t+B\|\overline{g}_t\|^2+C}$, $\forall t\in[T]$, where $T$ is fixed.
  By the definition of sub-Gaussian, $$\mathbb{E}_{t}\left[\exp\left(\zeta_t\right)\right]\leq \mathrm{e}, \quad\text{thus},\quad \mathbb{E}\left[\exp\left(\zeta_t\right)\right]\leq \mathrm{e}.$$
  By Markov's inequality, for any $\beta\in\mathbb{R}$,
  \begin{align*}
  \mathbb{P}\left(\max\limits_{t\in [T]}\zeta_t\geq \beta\right)=&\mathbb{P}\left(\exp\left(\max\limits_{t\in [T]}\zeta_t\right)\geq \mathrm{e}^\beta\right)\notag\\
  \leq & \text{e}^{-\beta}\mathbb{E}\left[\exp\left(\max\limits_{t\in [T]}\zeta_t\right)\right]\leq \text{e}^{-\beta}\mathbb{E}\left[\sumlimits{t=1}{T}\exp\left(\zeta_t\right)\right]\leq \text{e}^{-\beta}T \text{e}.
  \end{align*}
  Therefore, with probability at least $1-\delta$, we have 
  \begin{align*}
  \|g_t-\overline{g}_t\|^2\leq \log\left(\frac{T \mathrm{e}}{\delta}\right)\left(A\overline{\Delta_t}+B\|\overline{g}_t\|^2+C\right), \quad \forall t\in[T].
  \end{align*}
  Note that an additional factor $\log\left(\frac{T \mathrm{e}}{\delta}\right)$ emerges compared to Assumption \ref{assumption 3} but the convergence rate still holds up to a logarithm factor. 
  We can obtain the convergence rate using the induction argument similar to those in Section \ref{section 4}, \ref{section 5} and \ref{section 6}.
\end{proof}

\begin{proof}[Proof of Proposition \ref{proposition 1}]
  From \cite{ghadimi2016accelerated}, we have
  \begin{align*}
    \sumlimits{k=1}{t}\frac{\alpha_k}{\varGamma_k}=\frac{\alpha_1}{\varGamma_1}+\sumlimits{k=2}{t}\frac{\alpha_k}{\varGamma_k}=\frac{1}{\varGamma_1}+\sumlimits{k=2}{t}\frac{1}{\varGamma_k}\left(1-\frac{\varGamma_{k}}{\varGamma_{k-1}}\right)=\frac{1}{\varGamma_1}+\sumlimits{k=2}{t}\left(\frac{1}{\varGamma_k}-\frac{1}{\varGamma_{k-1}}\right)=\frac{1}{\varGamma_t}.
  \end{align*}
  Therefore, 
  \begin{align*}
    \varGamma_{t}\sumlimits{k=1}{t}\frac{\alpha_k}{\varGamma_k}=1.
  \end{align*}
  For any $t\geq 2$, we have
 \begin{align*}
  \varGamma_t=\left(1-\alpha_t\right)\varGamma_{t-1}=\frac{t-1}{t+1}\varGamma_{t-1}=\cdots=\frac{t-1}{t+1}\cdot\frac{t-2}{t}\cdots \frac{2-1}{2+1}=\frac{2}{t\left(t+1\right)}.
\end{align*}
Hence,
  \begin{align*}
    \left[\sumlimits{k=t}{T}\left(1-\alpha_k\right)\varGamma_k\right]\frac{\alpha_t}{\varGamma_t} & =\left[\sumlimits{k=t}{T}\frac{k-1}{k+1}\frac{2}{k\left(k+1\right)}\right]\cdot t\notag\\
    \leq & 2t \sumlimits{k=t}{T}\frac{1}{k\left(k+1\right)}\notag\\
    = & 2t\left(\frac{1}{t}-\frac{1}{T+1}\right) \leq 2.
  \end{align*}
\end{proof}

\begin{proof}[Proof of Proposition \ref{proposition overline x_t - x_t}]
  Following from the iteration steps in Algorithm \ref{algorithm1} and referring to the proof in \cite{ghadimi2016accelerated}, we have 
  \begin{align*}
    \overline{x}_t-x_t = & \left(1-\alpha_t\right)[\tilde{x}_t-x_t]\notag\\
    = & \left(1-\alpha_t\right)[\overline{x}_{t-1}-x_{t-1}+\left(\theta_{t-1}-\gamma_{t-1}\right)g_{t-1}]\notag\\
    = & \left(1-\alpha_t\right)[\left(1-\alpha_{t-1}\right)\left(\tilde{x}_{t-1}-x_{t-1}\right)+\left(\theta_{t-1}-\gamma_{t-1}\right)g_{t-1}]\notag\\
    = & \left(1-\alpha_t\right)\sumlimits{k=1}{t-1}\left(\prod\limits_{j=k+1}^{t-1}\left(1-\alpha_j\right)\right)\left(\theta_k-\gamma_k\right)g_k\notag\\
    = & \left(1-\alpha_t\right)\sumlimits{k=1}{t-1}\frac{\varGamma_{t-1}}{\varGamma_{k}}\left(\theta_k-\gamma_k\right)g_k\notag\\
    = & \left(1-\alpha_t\right)\varGamma_{t-1}\sumlimits{k=1}{t-1}\frac{\alpha_k}{\varGamma_k}\frac{\left(\theta_k-\gamma_k\right)}{\alpha_k}g_k.
  \end{align*}
  By the convexity of square of the norm and \eqref{proposition 1.1}, we have 
  \begin{align}
    \label{overline x_t-x_t square norm}
    \|\overline{x}_t-x_t\|^2 = & \left\Vert\left(1-\alpha_t\right)\varGamma_{t-1}\sumlimits{k=1}{t-1}\frac{\alpha_k}{\varGamma_k} \frac{\left(\theta_k-\gamma_k\right)}{\alpha_k}  g_k\right\Vert^2\notag\\
    \leq & \left(1-\alpha_t\right)^2\varGamma_{t-1}\sumlimits{k=1}{t-1}\frac{\alpha_k}{\varGamma_k}\frac{\left(\theta_k-\gamma_k\right)^2}{\alpha_k^2}\|g_k\|^2\notag\\
    = & \left(1-\alpha_t\right)\varGamma_t\sumlimits{k=1}{t-1}\frac{\alpha_k}{\varGamma_k}\frac{\left(\theta_k-\gamma_k\right)^2}{\alpha_k^2}\|g_k\|^2.
  \end{align}
\end{proof}

\begin{proof}[Proof of Lemma \ref{high probability sum non-convex constant}]
  Let $Z_t=-\la \overline{g}_t,\xi_t\ra$. 
  Note that $\overline{g}_t$ is a random variable dependent on $z_1,\cdots,z_{t-1}$ and $\xi_t$ is dependent on $z_1,\cdots,z_t$. Therefore, it is apparent that $Z_t$ is a martingale difference sequence since
  $$\mathbb{E}\left[- \la \overline{g}_t,\xi_t\ra|z_1,\cdots,z_{t-1}\right]=-\la \overline{g}_t,\mathbb{E}_t[\xi_t]\ra=0.$$
  Using Cauchy-Schwarz inequality,
  \begin{align*}
    \mathbb{E}_t\left[\exp\left(\frac{Z_t^2}{\|\overline{g}_t\|^2\left(A\overline{\Delta}_t+B\|\overline{g}_t\|^2+C\right)}\right)\right]\leq\mathbb{E}_t\left[\exp\left(\frac{\|\xi_t\|^2}{A\overline{\Delta}_t+B\|\overline{g}_t\|^2+C}\right)\right] \leq \text{e}.
  \end{align*}
  Therefore, given any $l\in [T]$, applying Lemma \ref{lemma high probability}, we have that for any $\lambda>0$, with probability at least $1-\delta$,
  \begin{align*}
    \sumlimits{t=1}{l}Z_t\leq &\frac{3\lambda}{4}\sumlimits{t=1}{l}\|\overline{g}_t\|^2\left(A\overline{\Delta}_t+B\|\overline{g}_t\|^2+C\right)+\frac{1}{\lambda}\log\frac{1}{\delta}\notag\\
    \leq &\frac{3\lambda}{4}\sumlimits{t=1}{l}\|\overline{g}_t\|^2 P_t^2+\frac{1}{\lambda}\log\frac{1}{\delta},
  \end{align*}
  where the last inequality follow from Lemma \ref{lemma 6.2} and $P_t$ defined in \eqref{P_t}.
 For any fixed $\lambda$, we can  re-scale over $\delta$ and have that  with probability at least $1-\delta$, for all $l\in[T]$,
 \begin{align}
  \label{4.14}
  \sumlimits{t=1}{l}-\la\overline{g}_t,\xi_t\ra\leq \frac{3\lambda}{4}\sumlimits{t=1}{l}\|\overline{g}_t\|^2 P_t^2+\frac{1}{\lambda}\log\frac{T}{\delta}.
 \end{align}
 Let $\lambda=\frac{1}{3 P_c^2}$, and we obtain the desired result.
\end{proof}

\begin{proof}[Proof of Lemma \ref{high probability sum convex constant}]
  Let $Z_t=-\eta\la \xi_t,x_t-x^*\ra$. Note that $x_t$ is a random variable dependent on $z_1,\cdots,z_{t-1}$ and $\xi_t$ is dependent on $z_1,\cdots,z_t$. 
  Therefore, it is apparent that $Z_t$ is a martingale difference sequence since
  $$\mathbb{E}\left[-\eta\la \xi_t,x_t-x^*\ra|z_1,\cdots,z_{t-1}\right]=-\eta\la \mathbb{E}_t[\xi_t],x_t-x^*\ra=0.$$
  Using Cauchy-Schwarz inequality,
  \begin{align*}
    \mathbb{E}_t\left[\exp\left(\frac{Z_t^2}{\eta^2\left(A\overline{\Delta}_t+B\|\overline{g}_t\|^2+C\right)\|x_t-x^*\|^2}\right)\right] \leq \text{e}.
  \end{align*}
  Therefore, given any $l\in [T]$, applying Lemma \ref{lemma high probability}, we have that for any $\lambda>0$, with probability at least $1-\delta$,
  \begin{align*}
    \sumlimits{t=1}{l}Z_t\leq &\frac{3\lambda}{4}\eta^2\sumlimits{t=1}{l}\left(A\overline{\Delta}_t+B\|\overline{g}_t\|^2+C\right)\|x_t-x^*\|^2+\frac{1}{\lambda}\log\frac{1}{\delta}.
  \end{align*}
 For any fixed $\lambda$, we can  re-scale over $\delta$ and have that with probability at least $1-\delta$, for all $l\in[T]$, 
 \begin{align*}
  \sumlimits{t=1}{l}-\eta \la \xi_t,x_t-x^*\ra\leq \frac{3\lambda}{4}\eta^2\sumlimits{t=1}{l}\left(A\overline{\Delta}_t+B\|\overline{g}_t\|^2+C\right)\|x_t-x^*\|^2+\frac{1}{\lambda}\log\frac{T}{\delta}.
 \end{align*}
 Let $\lambda=\frac{2}{3 D_c}$, and we obtain the desired result.
\end{proof}

\begin{proof}[Proof of Lemma \ref{high probability sum convex constant convergence}]
  Let $Z_t=-\la \xi_t,\overline{x}_t-x^*\ra$. Note that $\overline{x}_t$ is a random variable dependent on $z_1,\cdots,z_{t-1}$ and $\xi_t$ is dependent on $z_1,\cdots,z_{t}$. Therefore, $Z_t$ is a martingale difference sequence since 
$$\mathbb{E}\left[-\la \xi_t,\overline{x}_t-x^*\ra|z_1, \cdots z_{t-1}\right]=-\la \mathbb{E}_t\left[\xi_t\right],\overline{x}_t-x^*\ra=0.$$
Also, applying Cauchy-Schwarz inequality,
$$\mathbb{E}_t\left[\exp\left(\frac{Z_t^2}{\left(A\overline{\Delta}_t+B\|\overline{g}_t\|^2+C\right)\|\overline{x}_t-x^*\|^2}\right)\right]\leq \mathrm{e}.$$
By Lemma \ref{lemma high probability}, for any $\lambda>0$ and $\delta\in(0,1)$, with probability at least $1-\delta$, 
\begin{align*}
  -\sumlimits{t=1}{T}\la \xi_t,\overline{x}_t-x^*\ra \leq & \frac{3\lambda}{4}\sumlimits{t=1}{T}\left(A\overline{\Delta}_t+B\|\overline{g}_t\|^2+C\right)\|\overline{x}_t-x^*\|^2+\frac{1}{\lambda}\log\frac{1}{\delta}.
\end{align*}
\end{proof}

\begin{proof}[Proof of Lemma \ref{lemma 6}]
  Let $Z_t=-\tilde{\eta}_t \la \overline{g}_t,\xi_t\ra$. Note that $\overline{g}_t$ and $\tilde{\eta}_t$ are random variables dependent on $z_1,\cdots,z_{t-1}$ and $\xi_t$ is dependent on $z_1,\cdots,z_t$. Therefore, it is apparent that $Z_t$ is a martingale difference sequence since 
  $$\mathbb{E}\left[-\tilde{\eta}_t\la \overline{g}_t,\xi_t\ra|z_1,\cdots,z_{t-1}\right]=-\tilde{\eta}_t\la \overline{g}_t,\mathbb{E}_t\left[\xi_t\right]\ra=0.$$
  Using Cauchy-Schwarz inequality, 
  $$\mathbb{E}_t\left[\exp\left(\frac{Z_t^2}{\tilde{\eta}_t^2\|\overline{g}_t\|^2\left(A\overline{\Delta}_t+B\|\overline{g}_t\|^2+C\right)}\right)\right]\leq \mathbb{E}_t\left[\exp\left(\frac{\|\xi_t\|^2}{A\overline{\Delta}_t+B\|\overline{g}_t\|^2+C}\right)\right]\leq \text{e}.$$
  Therefore, given any $l\in [T]$, applying Lemma \ref{lemma high probability}, we have that for any $\lambda>0$, with probability at least $1-\delta$,
  \begin{align}
    \label{5.10}
    \sumlimits{t=1}{l}-\tilde{\eta}_t \la \overline{g}_t,\xi_t\ra \leq & \frac{3\lambda}{4}\sumlimits{t=1}{l}\tilde{\eta}_t^2\|\overline{g}_t\|^2\left(A\overline{\Delta}_t+B\|\overline{g}_t\|^2+C\right)+\frac{1}{\lambda}\log\frac{1}{\delta}\notag\\
    \leq & \frac{3\lambda \eta}{4}\sumlimits{t=1}{l}\tilde{\eta}_t\|\overline{g}_t\|^2\sqrt{A\overline{\Delta}_t+B\|\overline{g}_t\|^2+C}+\frac{1}{\lambda}\log\frac{1}{\delta},
  \end{align}
where the second inequality follows from the definition of $\tilde{\eta}_t$.
Combining with Lemma \ref{lemma 6.2} and $P_t$ defined in \eqref{P_t}, we have that 
\begin{align*}
  \sumlimits{t=1}{l}-\tilde{\eta}_t \la \overline{g}_t,\xi_t\ra
  \leq \frac{3\lambda \eta}{4}\sumlimits{t=1}{l}\tilde{\eta}_t\|\overline{g}_t\|^2P_t+\frac{1}{\lambda}\log\frac{1}{\delta}.
\end{align*}
For any fixed $\lambda$, we can re-scale over $\delta$ and have that with probability at least $1-\delta$, for all $l\in[T]$,
\begin{align*}
  \sumlimits{t=1}{l}-\tilde{\eta}_t \la \overline{g}_t,\xi_t\ra \leq \frac{3\lambda \eta}{4}\sumlimits{t=1}{l}\tilde{\eta}_t\|\overline{g}_t\|^2P_t+\frac{1}{\lambda}\log\frac{T}{\delta}.
\end{align*} 
Let $\lambda=\frac{1}{3\eta P_a}$, and we finish the proof.
\end{proof}
\begin{proof}[Proof of Lemma \ref{Lemma 5.4}]
  Denote $\Phi_k=2^{k-1}\overline{D}_1$ and $k_t=\lceil \log_2 \left(\overline{D}_t/\overline{D}_1\right)\rceil +1$.
  Applying the triangle inequality, we have
  \begin{align*}
    \|x_{t+1}-x^*\|=\|x_{t+1}-x_t+x_t-x^*\|\leq \|x_{t+1}-x_t\|+\|x_t-x^*\|\leq \eta+\|x_t-x^*\|. 
   \end{align*}
   Hence, 
   \begin{align*}
    D_t\leq D_1+\eta\left(t-1\right).
   \end{align*}
   Also we have that for all $2\leq t\leq T$,
   \begin{align*}
    \frac{\overline{D}_t}{\overline{D}_1}=\frac{\max \left\{ D_t,\eta \right\}}{\max\left\{ D_1,\eta \right\}}\leq \frac{\max \left\{ D_1+\eta(t-1),\eta \right\}}{\max\left\{ D_1,\eta \right\}}=\frac{D_1+\eta(t-1)}{\max\left\{D_1,\eta\right\}}\leq t.
   \end{align*}
   Therefore, 
   \begin{align*}
    1\leq k_t\leq \lceil \log_2 t\rceil+1\leq \log_2 t+2 = \log_2\left(4T\right), \quad \forall t\in[T].
   \end{align*}
    By Assumption \ref{assumption 3} and the definition of $\hat{\eta}_t$ in \eqref{definition of hat eta}, we have that 
  \begin{align*}
    \|\hat{\eta}_t\left(g_t-\overline{g}_t\right)\|\leq \frac{\eta\sqrt{A\overline{\Delta}_{\max}+B\|\overline{g}_t\|^2+C}}{\sqrt{G_0^2+\|\overline{g}_t\|^2}}\leq \frac{\eta\sqrt{A\overline{\Delta}_{\max}+C}}{G_0}+\eta\sqrt{B}=\mathcal{W}_{\max}, \quad \forall t\in[T].
  \end{align*}
  Define the projection to the unit ball, $\Pi(x)=x/\max\left\{1,\|x\|\right\}$.
  Thus,
  \begin{align}
    \label{6.36}
    \sumlimits{t=1}{l}\frac{\hat{\eta}_t \la g_t-\overline{g}_t, x_t-x^*\ra}{\mathcal{W}_{\max}\Phi_{k_l}} = & \sumlimits{t=1}{l}\frac{\hat{\eta}_t}{\mathcal{W}_{\max}}\left\la g_t-\overline{g}_t,\Pi\left(\frac{x_t-x^*}{\Phi_{k_l}}\right)\right\ra\notag\\
    \leq & \left|\sumlimits{t=1}{l}\frac{\hat{\eta}_t}{\mathcal{W}_{\max}}\left\la g_t-\overline{g}_t,\Pi\left(\frac{x_t-x^*}{\Phi_{k_l}}\right)\right\ra\right| \notag\\
    \leq & \max_{1\leq k \leq \lfloor \log_2(4T) \rfloor} \left|\sumlimits{t=1}{l}\frac{\hat{\eta}_t}{\mathcal{W}_{\max}}\left\la g_t-\overline{g}_t,\Pi\left(\frac{x_t-x^*}{\Phi_k}\right)\right\ra\right|,
  \end{align}
  where the equation holds since $\|x_t-x^*\|\leq D_t\leq \Phi_{k_t}\leq \Phi_{k_l}, \forall t\in[l]$.
  Let 
  \begin{align}\
    \label{X_t^k}
    X_t^k=\frac{\hat{\eta}_t}{\mathcal{W}_{\max}}\left\la g_t-\overline{g}_t,\Pi\left(\frac{x_t-x^*}{\Phi_k}\right)\right\ra.
  \end{align}
  Following Assumption \ref{assumption 2}, we have $\mathbb{E}_t\left[X_t^k\right]=0$.
  Also, using Cauchy-Schwarz inequality, 
  \begin{align*}
    X_t^k\leq \frac{\hat{\eta}_t}{\mathcal{W}_{\max}}\|g_t-\overline{g}_t\|\cdot\left\|\Pi\left(\frac{x_t-x^*}{\Phi_k}\right)\right\|\leq 1.
  \end{align*}
  Applying Lemma \ref{attia} with $X_t^k$ defined in \eqref{X_t^k} and $\hat{X}_t=0$, for any $k\in\left[\lfloor \log_2(4T) \rfloor\right]$ and $\delta'\in(0,1)$, with probability at least $1-\delta'$, we have that for all $l\in[T]$,
  \begin{align}
    \label{6.38}
   \left| \sumlimits{t=1}{l}X_t^k \right|\leq \sqrt{A_l(\delta')\sumlimits{t=1}{l}\left(X_t^k\right)^2+B_l(\delta')}.
  \end{align} 
  Next we will bound $\left(X_t^k\right)^2$.
  \begin{align*}
    \left(X_t^k\right)^2\leq & \frac{\hat{\eta}_t^2\|g_t-\overline{g}_t\|^2}{\mathcal{W}_{\max}^2}\left\|\Pi\left(\frac{x_t-x^*}{\Phi_k}\right)\right\|^2
    \leq  \frac{\eta_{t-1}^2\|g_t-\overline{g}_t\|^2}{\mathcal{W}_{\max}^2},
  \end{align*}
  where the last inequality holds since the definitions of $\hat{\eta}_t$ and $\Pi(x)$.
  Combining with \eqref{6.36} and \eqref{6.38}, we have that with probability at least $1-\log_2(4T)\delta'$,
  \begin{align*}
    \sumlimits{t=1}{l}\hat{\eta}_t \la \overline{g}_t-g_t, x_t-x^*\ra \leq \mathcal{W}_{\max}\Phi_{k_l} \sqrt{A_l(\delta')\sumlimits{t=1}{l}\frac{\eta_{t-1}^2\|g_t-\overline{g}_t\|^2}{\mathcal{W}_{\max}^2}+B_l(\delta')}.
  \end{align*}
  Using the fact that $\Phi_{k_l} =2^{\lceil \log_2\left(\overline{D}_l/\overline{D}_1\right) \rceil}\overline{D}_1 \leq 2\overline{D}_l $ and letting $\delta'=\delta/\log_2(4T)$, we obtain the final result.
  \end{proof}
\begin{proof}[Proof of Lemma \ref{Lemma 5.5}]
  Let $\lambda=\frac{2}{3 P_a^2}$ in \eqref{4.14}, and we get the result immediately.
 \end{proof}

\begin{proof}[Proof of Lemma \ref{lemma 6.1}]
  The proof here is similar to the proof of Lemma \ref{lemma 6}, which is under the generalized smoothness condition.
  It is easy to verify that $-\tilde{\eta}_t \la \overline{g}_t,\xi_t\ra$ is a martingale difference sequence.
  Combining \eqref{5.10} and Lemma \ref{smooth function value gap},
\begin{align*}
  \sumlimits{t=1}{l}-\tilde{\eta}_t \la \overline{g}_t,\xi_t\ra
  \leq \frac{3\lambda \eta}{4}\sumlimits{t=1}{l}\tilde{\eta}_t\|\overline{g}_t\|^2Q_t+\frac{1}{\lambda}\log\frac{1}{\delta}.
\end{align*}
For any fixed $\lambda$, we can re-scale over $\delta$ and have that with probability at least $1-\delta$, for all $l\in[T]$,
\begin{align*}
  \sumlimits{t=1}{l}-\tilde{\eta}_t \la \overline{g}_t,\xi_t\ra \leq \frac{3\lambda \eta}{4}\sumlimits{t=1}{l}\tilde{\eta}_t\|\overline{g}_t\|^2Q_t+\frac{1}{\lambda}\log\frac{T}{\delta}.
\end{align*} 
Let $\lambda=\frac{1}{3\eta Q}$, and we finish the proof.
\end{proof}

\begin{proof}[Proof of Lemma \ref{high probability sum non-convex adaptive smooth}]
  From the proof of Lemma \ref{high probability sum non-convex constant}, we know that $-\la \overline{g}_t,\xi_t\ra$ is a martingale difference sequence.
  Therefore, given any $l\in [T]$, applying Lemma \ref{lemma high probability}, we have that for any $\lambda>0$, with probability at least $1-\delta$,
  \begin{align*}
    \sumlimits{t=1}{l}-\la \overline{g}_t,\xi_t\ra \leq &\frac{3\lambda}{4}\sumlimits{t=1}{l}\|\overline{g}_t\|^2\left(A\overline{\Delta}_t+B\|\overline{g}_t\|^2+C\right)+\frac{1}{\lambda}\log\frac{1}{\delta}\notag\\
    \leq &\frac{3\lambda}{4}\sumlimits{t=1}{l}\|\overline{g}_t\|^2 Q_t^2+\frac{1}{\lambda}\log\frac{1}{\delta},
  \end{align*}
  where the last inequality follows from Lemma \ref{smooth function value gap}.
 For any fixed $\lambda$, we can  re-scale over $\delta$ and have that  with probability at least $1-\delta$, for all $l\in[T]$,
 \begin{align*}
  \sumlimits{t=1}{l}-\la\overline{g}_t,\xi_t\ra\leq \frac{3\lambda}{4}\sumlimits{t=1}{l}\|\overline{g}_t\|^2 Q_t^2+\frac{1}{\lambda}\log\frac{T}{\delta}.
 \end{align*}
 Let $\lambda=\frac{2}{3 Q^2}$, and we obtain the desired result.
\end{proof}
\begin{proof}[Proof of Theorem \ref{theorem 7}]
  By \eqref{constant gt} and letting $L_0=L, L_1=0$, we have
  \begin{align*}
    \sumlimits{t=1}{T}\|\overline{g}_t\|^2\leq &\frac{\overline{\Delta}_1}{\eta}-\sumlimits{t=1}{T}\la \overline{g}_t,\xi_t\ra+\frac{L\eta}{2}\sumlimits{t=1}{T}\left[\sumlimits{k=t}{T}\left(1-\alpha_k\right)\varGamma_k\right]\frac{\alpha_k}{\varGamma_k}\|g_t\|^2+L\eta\sumlimits{t=1}{T}\|g_t\|^2\notag\\
    \leq &\frac{\overline{\Delta}_1}{\eta}-\sumlimits{t=1}{T}\la \overline{g}_t,\xi_t\ra+2L\eta\sumlimits{t=1}{T}\|g_t\|^2,
  \end{align*}
  where the second inequality holds since \eqref{proposition 1.2}.
  Taking expectation on both sides, we have
  \begin{align*}
   \mE \left[\sumlimits{t=1}{T}\|\overline{g}_t\|^2\right]\leq &\frac{\overline{\Delta}_1}{\eta}+2L\eta\mE\left[\sumlimits{t=1}{T}\|g_t\|^2\right]\notag\\
   \leq & \frac{\overline{\Delta}_1}{\eta}+2L\eta \sumlimits{t=1}{T}\mE\left[A\overline{\Delta}_t+\left(B+1\right)\|\overline{g}_t\|^2+C\right]\notag\\
   \leq & \frac{\overline{\Delta}_1}{\eta}+2L\eta T\left(A\functiongapexpectation+C\right)+2L\eta\left(B+1\right)\mE\left[\sumlimits{t=1}{T}\|\overline{g}_t\|^2\right],
  \end{align*}
  where the second inequality follows from \eqref{triangle expectation} and the last inequality holds since Proposition \ref{expectation}.
  Since $2L\eta\left(B+1\right)\leq 1/2$, we have
  \begin{align}
    \label{expectation non convex}
    \mE\left[\frac{1}{T}\sumlimits{t=1}{T}\|\overline{g}_t\|^2\right]\leq &\frac{2\overline{\Delta}_1}{\eta T}+4L\eta\left(A\functiongapexpectation+C\right)\notag\\
  \leq &\frac{2\overline{\Delta}_1}{\sqrt{T}}\left(\sqrt{3LA}+\sqrt{\frac{3LC}{2}}\right)+\frac{4}{\sqrt{3T}}\left(\functiongapexpectation\sqrt{LA}+\sqrt{2LC}\right) \notag\\
  &+\frac{2\overline{\Delta}_1}{T}\left(\sqrt{\frac{5L\left(A\functiongapexpectation+2L\left(B+1\right)\functiongapexpectation+C\right)}{2}}+4L\left(B+1\right)\right).
  \end{align}
  The proof is complete.
\end{proof}

  \begin{proof}[Proof of Theorem \ref{theorem 8}]
  By \eqref{constant convex sum}, we have
  \begin{align}
    \label{8.22}
    \sumlimits{t=1}{T}\la g_t, \overline{x}_t-x^*\ra\leq \frac{1}{2\eta}\|x_1-x^*\|^2+2\eta\sumlimits{t=1}{T}\|g_t\|^2.
  \end{align}
  Therefore, 
  \begin{align*}
    \mE\sumlimits{t=1}{T}\la \overline{g}_t,\overline{x}_t-x^*\ra= & \sumlimits{t=1}{T}\mE\left[\mE_t\la g_t, \overline{x}_t-x^*\ra\right]\notag\\
    \leq & \frac{1}{2\eta}\|x_1-x^*\|^2+2\eta\sumlimits{t=1}{T}\mE\left[\mE_t\left[\|g_t\|\right]^2\right]\notag\\
    \leq & \frac{1}{2\eta}\|x_1-x^*\|^2+2\eta\mE\sumlimits{t=1}{T}\left(A\overline{\Delta}_t+\left(B+1\right)\|\overline{g}_t\|^2+C\right)\notag\\
    \leq & \frac{1}{2\eta}\|x_1-x^*\|^2+2\eta\left(A+2L\left(B+1\right)\right)\mE\sumlimits{t=1}{T}\overline{\Delta}_t+2\eta TC,
  \end{align*}
  where the first inequality follows from \eqref{8.22}, the second inequality holds since \eqref{triangle expectation} and the last inequality follows from Lemma \ref{smooth function value gap}.
  Combining with Lemma \ref{convexity} and the constraint that $$2\eta\left(A+2L\left(B+1\right)\right)\leq 1/2,$$
  we have
  \begin{align*}
    \mE\sumlimits{t=1}{T}\left(f(\overline{x}_t)-f^*\right)\leq \mE\sumlimits{t=1}{T}\la \overline{g}_t,\overline{x}_t-x^*\ra\leq \frac{1}{2\eta}\|x_1-x^*\|^2+\frac{1}{2}\mE\sumlimits{t=1}{T}\overline{\Delta}_t+2\eta TC.
  \end{align*}
  Applying Jensen's inequality, we have
  \begin{align}
    \label{expectation convex}
    \mE\left[f\left(\frac{1}{T}\sumlimits{t=1}{T}\overline{x}_t\right)-f^*\right]\leq & \frac{1}{\eta T}\|x_1-x^*\|^2+4\eta C\notag\\
      \leq & \frac{4\left(A+2L\left(B+1\right)\right)\|x_1-x^*\|^2}{T}+ \frac{\sqrt{C}\left(\|x_1-x^*\|^2+4\right)}{\sqrt{T}}.
  \end{align}
  \end{proof}
\section{Smoothness of Phase Retrieval Problem}
  \label{phase retrieval proposition}
In this section,
we prove that the objective function from the phase retrieval model in \eqref{phase retrieval} satisfies $\left(L_0,L_1\right)$-smoothness in Definition \ref{definition 1}.
We first introduce the following lemma from \citep{chen2023generalized}.
\begin{lemma}
  The function $f(x)$ in \eqref{phase retrieval} belongs to $\mathcal{L}^*_{\text{sym}}\left(\frac{2}{3}\right)$, i.e., for any $x,y\in\mathbb{R}^d$,
\begin{align}
  \label{phase retrial inequality}
  \|\nabla f(x)-\nabla f(y)\|\leq \left(L_0'+L_1'\|\nabla f(x)\|^{\frac{2}{3}}+L_2'\|x-y\|^2\right)\|x-y\|,
\end{align}
where $L_0',L_1',L_2'$ are non-negative constants.
\end{lemma}
With the above lemma, we can prove that $f(x)$ in \eqref{phase retrieval} is also $(L_0,L_1)$-smooth.
\begin{lemma}
  \label{lemma B.4}
  Suppose that $f(x)$ is the objective function from the phase retrieval model defined in \eqref{phase retrieval}. Then, $f(x)$ is $(L_0,L_1)$-smooth, where $L_0=L_0'+\frac{4}{27}L_1'+L_2'/L_1^2$ and $L_1=L_1'$.
\end{lemma}
\begin{proof}
  By \eqref{phase retrial inequality}, for any $x, y\in \mathbb{R}^d$ such that $\|x-y\|\leq\frac{1}{L_1}$, we have 
  \begin{align*}
    \|\nabla f(x)-\nabla f(y)\|\leq & \left(L_0'+L_1'\|\nabla f(x)\|^{\frac{2}{3}}+L_2'/L_1^2\right)\|x-y\|.
  \end{align*}
  Applying Young's inequality, we obtain that 
  \begin{align*}
    \|\nabla f(x)-\nabla f(y)\|\leq &\left(L_0'+L_1'\|\nabla f(x)\|+\frac{4}{27}L_1'+L_2'/L_1^2\right)\|x-y\|\notag\\
    = & \left(L_0+L_1\|\nabla f(x)\|\right)\|x-y\|.
  \end{align*}
\end{proof}

\end{document}